\documentclass[reqno,12pt]{amsart}
\title{Hyperbolic and Parabolic Unimodular Random Maps}
\author{Omer Angel \quad Tom Hutchcroft \quad Asaf Nachmias \quad Gourab Ray}
\date{\small \today}

\usepackage{amsmath,amssymb,amsfonts,amsthm}

\usepackage[colorlinks=true,linkcolor=blue,citecolor=black,urlcolor=black,pdfborder={0 0 0}]{hyperref}

\usepackage[protrusion=true,expansion=true]{microtype}

\usepackage{color}

\usepackage{graphicx}
\usepackage[font=sf, labelfont={sf,bf}, margin=1cm]{caption}

\usepackage{enumitem}

\usepackage{bbm}
\usepackage{caption}
\usepackage{mathrsfs}

\usepackage[english]{isodate}   \isodate

\usepackage[nameinlink]{cleveref}

\crefname{theorem}{Theorem}{Theorems}
\crefname{thm}{Theorem}{Theorems}
\crefname{mainthm}{Theorem}{Theorems}
\crefname{lemma}{Lemma}{Lemmas}
\crefname{lem}{Lemma}{Lemmas}
\crefname{remark}{Remark}{Remarks}
\crefname{prop}{Proposition}{Propositions}
\crefname{defn}{Definition}{Definitions}
\crefname{corollary}{Corollary}{Corollaries}
\crefname{conjecture}{Conjecture}{Conjectures}
\crefname{question}{Question}{Questions}
\crefname{chapter}{Chapter}{Chapters}
\crefname{section}{Section}{Sections}
\crefname{figure}{Figure}{Figures}
\newcommand{\crefrangeconjunction}{--}

\theoremstyle{plain}
\newtheorem{mainthm}{Theorem}
\newtheorem{thm}{Theorem}[section]
\newtheorem*{thm*}{Theorem}
\newtheorem{lemma}[thm]{Lemma}
\newtheorem{lem}[thm]{Lemma}
\newtheorem{corollary}[thm]{Corollary}
\newtheorem{prop}[thm]{Proposition}
\newtheorem{conjecture}[thm]{Conjecture}

\newtheorem{question}[thm]{Question}
\theoremstyle{definition}

\newtheorem{example}[thm]{Example}

\theoremstyle{remark}
\newtheorem*{remark}{Remark}

\numberwithin{equation}{section}

\usepackage[medium]{titlesec}
\titleformat{\section}{\normalfont\Large}{\thesection}{1em}{}
\titleformat{\subsection}{\normalfont\bfseries}{\thesubsection}{1em}{}
\titleformat{\subsubsection}{\normalfont\bfseries}{\thesubsubsection}{1em}{}
\titleformat{\paragraph}[runin]{\bfseries}{\theparagraph}{}{}
\titlespacing*{\section}{0pt}{3.5ex plus 1ex minus .2ex}{2.3ex plus .2ex}
\titlespacing*{\subsection}{0pt}{3.25ex plus 1ex minus .2ex}{1.5ex plus .2ex}
\titlespacing*{\subsubsection}{0pt}{3.25ex plus 1ex minus .2ex}{1.5ex plus .2ex}
\titlespacing*{\paragraph}{0pt}{1ex plus .2ex minus .2ex}{1.5ex plus .2ex}

\usepackage[margin=1in]{geometry}

\renewcommand{\P}{\mathbb P}
\newcommand{\E}{\mathbb E}
\newcommand{\C}{\mathbb C}
\newcommand{\R}{\mathbb R}
\newcommand{\Z}{\mathbb Z}
\renewcommand{\H}{\mathbb H}
\newcommand{\N}{\mathbb N}

\newcommand{\F}{\mathfrak F}
\newcommand{\eps}{\varepsilon}

\newcommand{\eqd} {\overset{d}{=}}
\newcommand{\dual}{\dagger}
\newcommand{\Ffin}{F_{\mathrm{fin}}}
\DeclareMathOperator{\genus}{genus}

\newcommand{\cE}{\mathcal E}
\newcommand{\cG}{\mathcal G}
\newcommand{\cM}{\mathcal M}
\newcommand{\cR}{\mathcal R}

\newcommand{\sA}{\mathscr A}
\newcommand{\sE}{\mathscr E}

\newcommand{\bbH}{\mathbb H}
\newcommand{\bbK}{\mathbb K}
\newcommand{\bbX}{\mathbb X}

\renewcommand{\H}{\mathbb H}

\newcommand{\FUSF}{\mathsf{FUSF}}
\newcommand{\WUSF}{\mathsf{WUSF}}
\newcommand{\FMSF}{\mathsf{FMSF}}
\newcommand{\WMSF}{\mathsf{WMSF}}
\newcommand{\UST}{\mathsf{UST}}

\newcommand{\Prev}{\P_\mathrm{rev}}
\newcommand{\Reff}{{\mathcal R}_\mathrm{eff}}
\DeclareMathOperator{\ang}{ang}

\DeclareMathOperator{\area}{area}

\DeclareMathOperator{\EEL}{EEL}

\newcommand{\asaf}[1]{{#1}}

\newcommand{\tom}[1]{{ #1}}

\newcommand{\angl}{{\rm ang}}

\AtEndDocument{
  \small

  \bigskip  \noindent
  \textsc{Omer Angel} \\
  \textsc{Department of Mathematics, University of British Columbia} \\
  \textit{Email:} \texttt{angel@math.ubc.ca} \par

 \medskip  \noindent
  \textsc{Tom Hutchcroft} \\
  \textsc{Statistical Laboratory, DPMMS, University of Cambridge} \\
  \textit{Email:} \texttt{t.hutchcroft@maths.cam.ac.uk}

  \medskip  \noindent
  \textsc{Asaf Nachmias} \\
  \textsc{School of Mathematical Sciences, Tel Aviv University} \\
  \textit{Email:} \texttt{asafnach@post.tau.ac.il}

  \medskip  \noindent
  \textsc{Gourab Ray} \\
  \textsc{Department of Mathematics and Statistics, University of Victoria}\\
  \textit{Email:} \texttt{gourab1987@gmail.com}
}

\begin{document}

\begin{abstract}
  We show that for infinite planar unimodular random rooted maps, many
  global geometric and probabilistic properties are equivalent, and are
  determined by a natural, local notion of average curvature. This
  dichotomy includes properties relating to amenability, conformal
  geometry, random walks, uniform and minimal spanning forests, and
  Bernoulli bond percolation. We also prove that every simply connected
  unimodular random rooted map is sofic, that is, a Benjamini-Schramm
  limit of finite maps.
\end{abstract}

\maketitle  

\vspace{-1cm}

\section{Introduction}

In the classical theory of Riemann surfaces, the Uniformization Theorem
states that every simply connected, non-compact Riemann surface is
conformally equivalent to either the plane or the disc, which are
inequivalent to each other by Liouville's theorem. The dichotomy
provided by this theorem manifests itself in several different ways,
relating to analytic, geometric and probabilistic properties of
surfaces.  In particular, if $S$ is a simply connected, non-compact
Riemann surface, then either
\begin{quote}
  \textbf{$S$ is parabolic:} it is conformally equivalent to the plane,
  admits a compatible Riemannian metric of constant curvature $0$, does not
  admit non-constant bounded harmonic functions, and is recurrent for
  Brownian motion,
\end{quote}
or else
\begin{quote}
  $S$ \textbf{is hyperbolic:} it is conformally equivalent to the disc,
  admits a compatible Riemannian metric of constant curvature $-1$, admits
  non-constant bounded harmonic functions, and is transient for Brownian
  motion.
\end{quote}

In the 1990's, a discrete counterpart to this dichotomy began to develop in
the setting of bounded degree planar graphs \cite{beardon1991circle,HeSc,HS93,BS96a,BS96b}.  A major milestone in this theory was the seminal
work of He and Schramm \cite{HeSc,HS93}, who studied \emph{circle
  packings} of infinite triangulations of the plane.  They proved that
every infinite triangulation of the plane can be circle packed in either
the unit disc or in the plane, but not both.  A triangulation is called
\textbf{CP hyperbolic} or \textbf{CP parabolic} accordingly.
He and Schramm also connected the circle packing type to isoperimetric
and probabilistic properties of the triangulation, showing in particular
that, in the bounded degree case, CP parabolicity is equivalent to the
recurrence of simple random walk. Later, Benjamini and Schramm
\cite{BS96a,BS96b} provided an analytic aspect to this dichotomy,
showing that every bounded degree, infinite planar graph admits
non-constant bounded harmonic functions if and only if it is transient
for simple random walk, and that in this case the graph also admits
non-constant bounded harmonic functions of finite Dirichlet energy.
Most of this theory fails without the assumption of bounded degrees, as
one can easily construct pathological counter-examples to the theorems
above.

The goal of this paper is to develop a similar theory for
\emph{unimodular random rooted maps}, without the assumption of bounded
degree. In our earlier work \cite{AHNR15}, we studied circle packings
of, and random walks on, random plane triangulations of unbounded
degree.  In this paper, we study many further properties of unimodular
random planar maps, which we do not assume to be triangulations. Our
results also have consequences for unimodular random maps that are not
planar, which we develop in \cref{Sec:multiplyconnected}.  Our main
result may be stated informally as follows; see \cref{thm:dichotomy} for
a complete and precise statement.

\begin{thm*}[The Dichotomy Theorem]
  Every infinite, planar, unimodular random rooted planar map is either
  hyperbolic or parabolic.  The map is hyperbolic if and only if its
  average curvature is negative and is parabolic if and only if its
  average curvature is zero.  The type of a unimodular random rooted map
  determines many of its properties.
\end{thm*}

The many properties we show to be determined by the type of the map are
far-reaching, relating to aspects of the map including amenability,
random walks, harmonic functions, spanning forests, Bernoulli bond
percolation, and the conformal type of associated Riemann surfaces.  The
seeds of such a dichotomy were already apparent in \cite{AHNR15}, in
which we proved that a unimodular random rooted plane triangulation is
CP parabolic almost surely if and only if the expected degree of the
root is six (which is equivalent to the average curvature being zero),
if and only if the triangulation is \textbf{invariantly amenable} -- a
notion of amenability due to Aldous and Lyons \cite{AL07} that is
particularly suitable to unimodular random rooted graphs. A notable
property that is \emph{not} a part of \cref{thm:dichotomy} is recurrence
of the random walk: while every hyperbolic unimodular random rooted map
is transient, not every parabolic unimodular random rooted map is
recurrent. (\cref{thm:dichotomy} can be combined with the work of
Gurel-Gurevich and the third author \cite{GGN13} to deduce that a
parabolic unimodular random planar map is recurrent under the additional
assumption that the degree of the root has an \emph{exponential tail}.)

A \textbf{map} is a \emph{proper} (i.e., locally finite) embedding of a
graph into an oriented surface viewed up to orientation preserving
homeomorphisms of the surface.  (Other definitions extend to
non-orientable maps, but we shall not be concerned with those here.)  A
\textbf{rooted map} is a map together with a distinguished root vertex.
The map is called \textbf{planar} if the surface is homeomorphic to an
open subset of the sphere, and is \textbf{simply connected} if the
surface is homeomorphic to the sphere or the plane.  (In particular,
every simply connected map is planar.)  A random rooted map is said to
be \textbf{unimodular} if it satisfies the \textbf{mass-transport
  principle}, which can be interpreted as meaning that `every vertex of
the map is equally likely to be the root'.  See \cref{Sec:background}
for precise definitions of each of these terms.

The \textbf{curvature} of a map is a local geometric property, closely
related to the Gaussian curvature of manifolds that may be constructed
from the map; see \cref{Sec:curvature} for a precise definition. For one
natural manifold constructed from the map by gluing together regular
polygons, the curvature at each vertex $v$ is
\[\kappa(v) = 2\pi- \sum_{f \perp v} \frac{\deg(f)-2}{\deg(f)}\pi,\]
where the sum is taken over faces of the map incident to $v$, and a face
is counted with multiplicity if more than one of the corners of the face
are located at $v$.  We can define the average curvature of $(M,\rho)$
to be the expectation $\E[\kappa(\rho)]$. \cref{thm.invcurvature} states
that the average curvature is a canonical quantity associated to the
random map, in the sense that any unimodular way of associating a
manifold to the map will result in the same average curvature. Observe
that the average curvature of a unimodular random triangulation is equal
to $(6 - \E[\deg(\rho)])\pi/3$, so that a unimodular triangulation has
expected degree greater than six if and only if it has negative average
curvature. This relates the dichotomy described in \cite{AHNR15} to that
of \cref{thm:dichotomy}.

Classical examples of unimodular random maps are provided by Voronoi
diagrams of stationary point processes \cite{BPP} and (slightly
modified) Galton-Watson trees \cite[Example 1.1]{AL07}, as well as
lattices in the Euclidean and hyperbolic planes, and arbitrary local
limits of finite maps.  Many local modifications of maps, such as taking
Bernoulli percolation or uniform or minimal spanning trees, preserve
unimodularity, giving rise to many additional examples.  Unimodular
random maps, most notably the uniform infinite planar triangulation
(UIPT) \cite{UIPT1} and quadrangulation (UIPQ)
\cite{Krikun05,UIPQinfty}, have also been studied in the context of
$2$-dimensional quantum gravity; see the survey \cite{Garban13} and
references therein.  More recently, hyperbolic variants of the UIPT have
been constructed \cite{AR13,PSHIT}.  Many of these examples do not have
uniformly bounded degrees, so that the deterministic theory is not
applicable to them.

\subsection{The Dichotomy Theorem}\label{sec:dichotomystatement}

Since many of the notions tied together in the following theorem are well
known, we first state the theorem, and defer  definitions to
individual sections dealing with each of the properties.

\begin{mainthm}[The Dichotomy Theorem]\label{thm:dichotomy}
  Let $(M,\rho)$ be an infinite, ergodic, unimodular random rooted planar
  map and suppose that $\E[\deg(\rho)]<\infty$.  Then the average curvature
   of $(M,\rho)$ is non-positive and the following are
  equivalent:
  \begin{enumerate}[nosep]
  \item \label{iCurv} $(M,\rho)$ has average curvature zero.
  \item \label{iAmen} $(M,\rho)$ is invariantly amenable.
  \item \label{iAmenablesub} Every bounded degree subgraph of $M$ is
    amenable almost surely.
  \item \label{iTree2} Every subtree of $M$ is amenable almost surely.
  \item \label{iRecurrent} Every bounded degree subgraph of $M$ is
    recurrent almost surely.
  \item \label{iTree} Every subtree of $M$ is recurrent almost surely.
  \item \label{iBS} $(M,\rho)$ is a Benjamini-Schramm limit of finite
    planar maps.
  \item \label{iBS2} $(M,\rho)$ is a Benjamini-Schramm limit of a sequence
    $\langle M_n \rangle_{n\geq0}$ of finite maps such that
    \[
    \frac{\text{genus}(M_n)}{\#\{\text{vertices of $M_n$}\}}
    \xrightarrow[n\to\infty]{}0.
    \]
  \item \label{iConf} The Riemann surface associated to $M$ is conformally
    equivalent to either the plane $\C$ or the cylinder $\C/\Z$ almost
    surely.
  \item \label{iLiouville} $M$ does not admit any non-constant bounded
    harmonic functions almost surely.
  \item \label{iDirichelet} $M$ does not admit any non-constant harmonic
    functions of finite Dirichlet energy almost surely.
  \item \label{iUSF} The laws of the free and wired uniform spanning
    forests of $M$ coincide almost surely.
  \item \label{iWUSFConn} The wired uniform spanning forest of $M$ is
    connected almost surely.
  \item \label{iIntersect} Two independent random walks on $M$ intersect
    infinitely often almost surely.
  \item \label{iMSF} The laws of the free and wired minimal spanning
    forests of $M$ coincide almost surely.
  \item \label{iPerc} Bernoulli$(p)$ bond percolation on $M$ has at most
    one infinite connected component for every $p\in[0,1]$ almost surely
    (in particular, $p_c=p_u$).
  \item \label{iVEL} $M$ is vertex extremal length parabolic almost surely.
  \end{enumerate}
\end{mainthm}

In light of this theorem, we call a unimodular random rooted map
$(M,\rho)$ with $\E[\deg(\rho)]<\infty$ \textbf{parabolic} if its
average curvature is zero (and, in the planar case, clauses (1)--(17)
all hold), and \textbf{hyperbolic} if its average curvature is negative
(and, in the planar case, the clauses all fail).

\begin{figure}
  \centering
  \includegraphics[width=\textwidth]{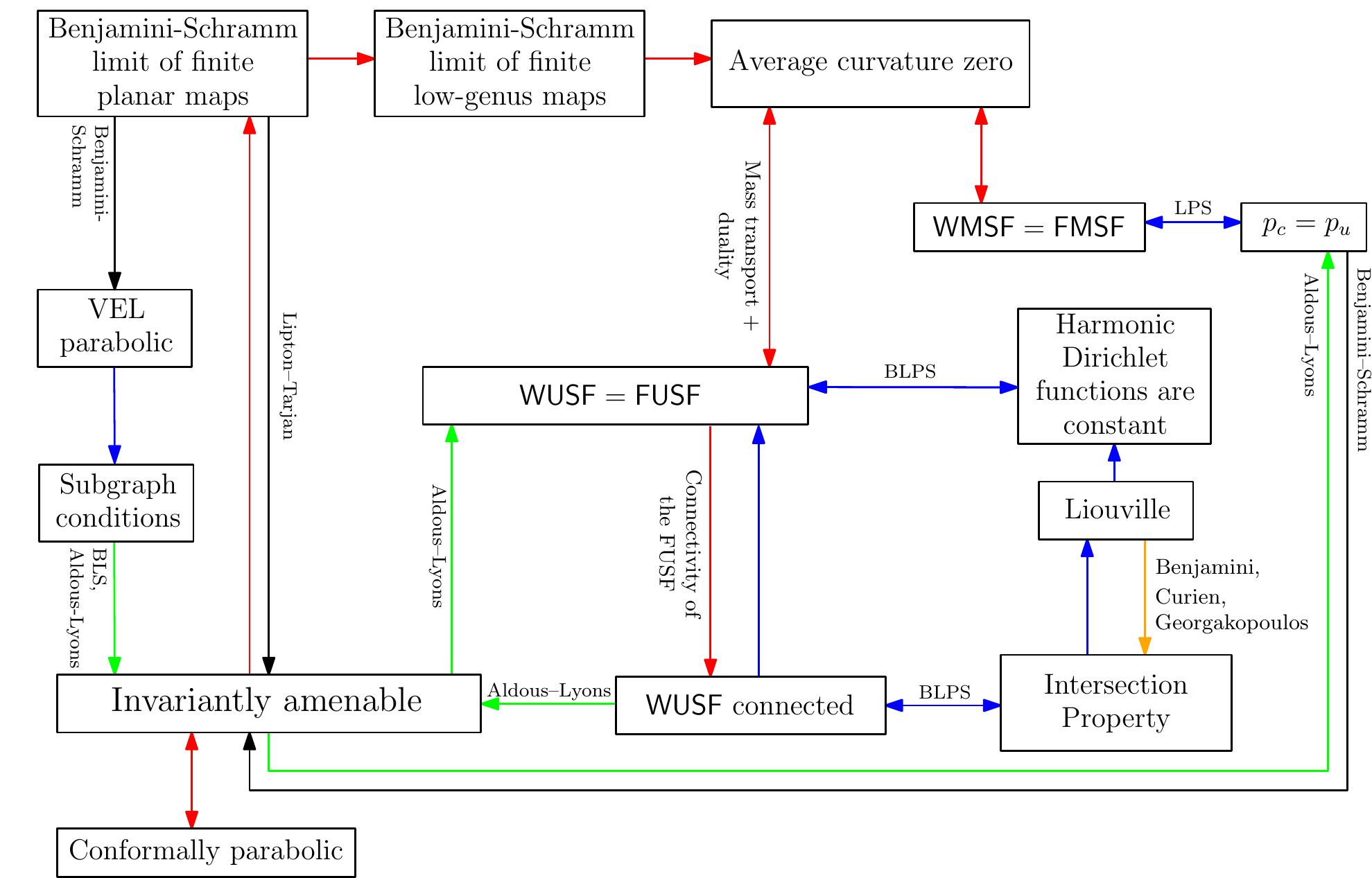}

  \caption{The logical structure for the proof of \cref{thm:dichotomy} in
    the simply connected case.  Implications new to this paper are in red.
    Blue implications hold for arbitrary graphs; the orange implication
    holds for arbitrary planar graphs, and green implications hold for
    unimodular random rooted graphs even without planarity.  A few
    implications between items that are known but not used in the proof are
    omitted.}
  \label{fig:logic}
\end{figure}

\begin{figure}
\includegraphics[width=\textwidth]{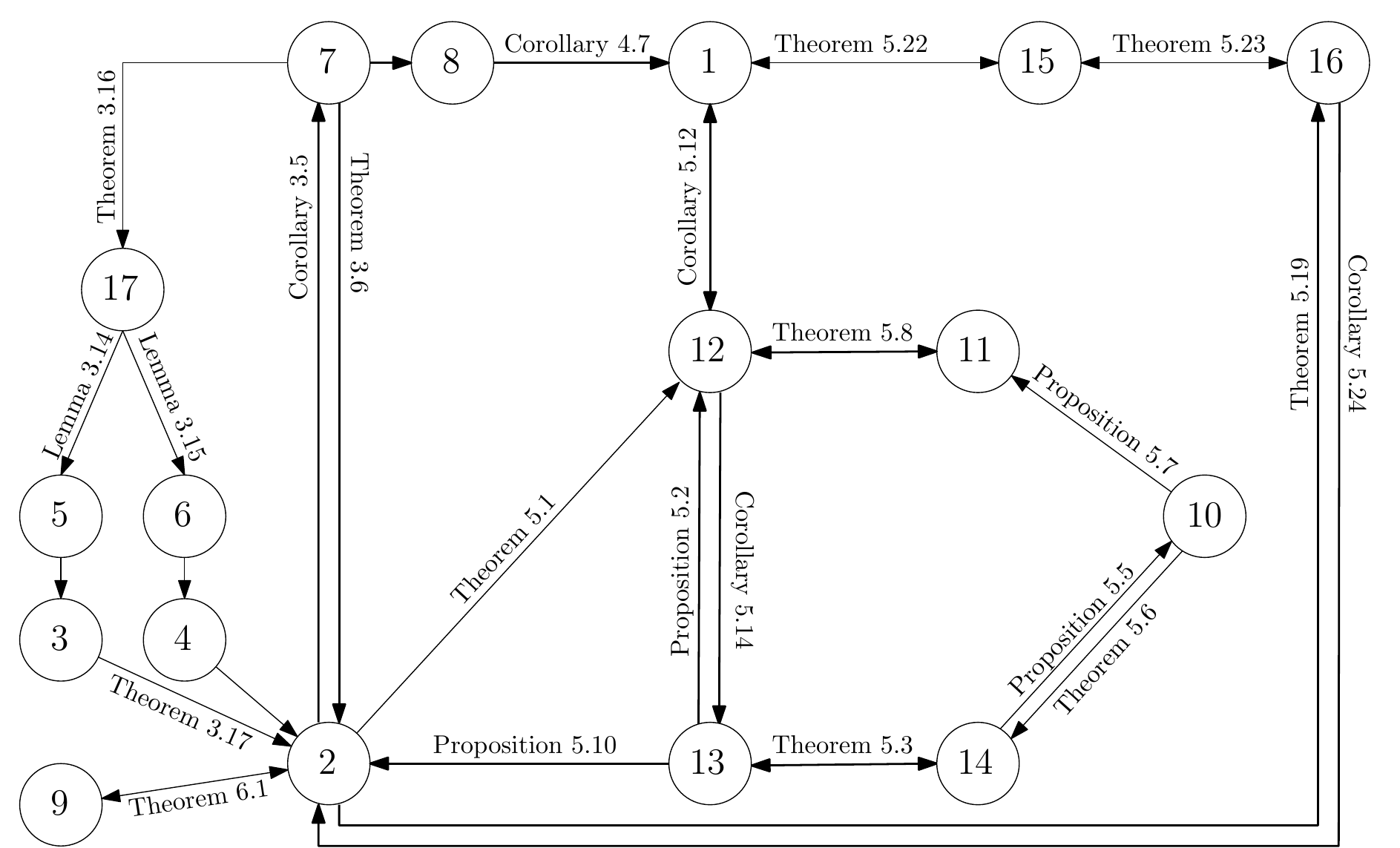}
\caption{The numbers of the theorems, propositions, lemmas and corrolaries forming the individual implications used to prove \cref{thm:dichotomy} in the simply connected case. Unlabelled implications are trivial.}
\label{fig:logic2}
\end{figure}

\medskip

As one might guess from the structure of \cref{thm:dichotomy}, the proof
consists of many separate arguments for the different implications.
 Some of the implications
are already present in the literature, and part of this paper is spent
surveying the earlier works that form the individual implications between
the long list of equivalent items in \cref{thm:dichotomy}.  For the sake of
completeness we also include proofs of several standard technical results
from the ergodic theory literature using probabilisitic terminology.

%
Some of the implications in \cref{thm:dichotomy} hold in any graph.  For
example, \eqref{iLiouville} implies \eqref{iDirichelet} for any graph, and
\eqref{iIntersect} and \eqref{iWUSFConn} are always equivalent (see
\cite{BLPS}).  Other implications hold for any planar graph.  For example
\eqref{iLiouville} is equivalent to \eqref{iIntersect}, see \cite{BCG12}.
We do not provide a comprehensive list of the assumptions needed for each
implication, but some of this information is encoded in \cref{fig:logic}.

Most of the paper is dedicated to proving the theorem under the additional
assumption that $M$ is simply connected; the multiply-connected case is
easier and is handled separately in \cref{Sec:multiplyconnected}.  The
logical structure of the proof in the simply connected case is summarized
in \cref{fig:logic,fig:logic2}. \cref{fig:logic} also shows which implications were already known and
which are proved in the present paper. 

\tom{
Let us note that there are several natural ways of turning a planar map into a triangulation, which allow some of the implications between items in \cref{thm:dichotomy} to be reduced to the case of triangulations, in which some of the implications were already treated in \cite{AHNR15}. However, many of the properties listed in \cref{thm:dichotomy} are not invariant to taking subgraphs, and most of the properties were not treated in \cite{AHNR15}. Thus, we cannot rely on the methods of our previous paper \cite{AHNR15}. 
}




\subsection{Unimodular planar maps are sofic}



Let $\langle G_n \rangle_{n\geq0}$ be a sequence of (possibly random)
finite graphs.  We say that a random rooted graph $(G,\rho)$ is the
\textbf{Benjamini-Schramm limit} of the sequence
$\langle G_n \rangle_{n\geq0}$ if the random rooted graphs $(G_n,\rho_n)$
converge in distribution to $(G,\rho)$ with respect to the local topology
on rooted graphs (see \cref{Sec:unimod}), where $\rho_n$ is a uniform
vertex of $G_n$.  Benjamini-Schramm limits of finite maps, marked graphs and marked maps are
defined similarly, except that the local topology takes into account the
additional structure.  When $G_n$ is a uniformly chosen map of size $n$, this
construction gives rise to the aforementioned UIPT and UIPQ.

It is easy to see that every (possibly random) finite graph with
conditionally uniform root is unimodular.  Moreover, unimodularity is
preserved under distributional limits in the local topology.  It follows
that every Benjamini-Schramm limit of finite random graphs is unimodular.
A random rooted graph that can be obtained in this way is called
\textbf{sofic}.  It is a major open problem to determine whether the
converse holds, that is, whether every unimodular random rooted graph is
sofic \cite[Section 10]{AL07}.  The next theorem answers this question
positively for simply connected unimodular random maps.

\begin{mainthm}\label{thm:soficmaps}
  Every simply connected unimodular random rooted map is sofic.
\end{mainthm}

The proof of \cref{thm:soficmaps} relies on the corresponding result for
trees, which is due to Bowen \cite{Bowen03}, Elek \cite{Elek10}, and
Benjamini, Lyons and Schramm \cite{URT}.  As was observed by Elek and
Lippner \cite{ElekLipp10}, it follows that \emph{treeable} unimodular graphs (that is, unimodular graphs that exhibit an invariant spanning tree) are
sofic, and so the key new step is proving the connectivity of the free
uniform spanning forest.

Note that the finite maps converging to a given infinite unimodular random
rooted map $(M,\rho)$ need not be planar. Indeed, \cref{thm:dichotomy}
characterises the Benjamini-Schramm limits of finite \emph{planar} maps
exactly as the parabolic unimodular random rooted maps.  Moreover, if
$\langle M_n \rangle_{n\geq0}$ is a sequence of finite maps converging to
an infinite hyperbolic unimodular random rooted map, then the approximating
maps $M_n$ must have genus comparable to their number of vertices as $n$
tends to infinity.


\section{Unimodular Maps}\label{Sec:background}

\subsection{Maps}\label{Sec:maps} We provide here a brief background to the concept of maps, and refer the reader to \cite[Chapter 1.3]{LaZv04} for a comprehensive treatment. Let $G=(V,E)$ be a connected graph, which may contain self-loops and
multiple edges.
An \textbf{embedding} of a graph in a surface $S$ is a drawing of the graph in
the surface with non-crossing edges. Given an embedding, the connected
components of the complement of the image of $G$ are called \textbf{faces}.
An embedding is said to be \textbf{proper} if the following conditions hold:
\begin{enumerate}
\item it is locally finite (every compact set in $S$ intersects finitely
  many edges),
\item every face is homeomorphic to an open disc, and
\item for every face $f$, if we consider the oriented edges of $G$ that
  have their right hand side incident to $f$, and consider the
  permutation that maps each such oriented edge to the oriented edge
  following it in the clockwise order around the face, then this
  permutation has a single orbit.
\end{enumerate}
For example, the complete graph on three vertices can be properly
embedded in the sphere but not in the plane.  If $S$ is simply connected
or compact, then any embedding that satisfies (1) and (2) must also
satisfy (3). For this reason, the condition (3) is not included in many
references that deal primarily with finite maps.  An example of an
embedding that satisfies (1) and (2) but not (3) is given by drawing
$\Z$ along a straight line in an infinite cylinder.

We define a (locally finite) \textbf{map} $M$ to be a connected, locally
finite graph $G$ together with an equivalence class of proper embeddings
of $G$ into oriented surfaces, where two embeddings are equivalent if
there is an orientation preserving homeomorphism between the two
surfaces that sends one embedding to the other.  If $M$ is a map with
underlying graph $G$, we refer to any proper embedding of $G$ that falls
into the equivalence class of embeddings corresponding to $M$ as an
embedding of $M$. A map is said to be \textbf{planar} if it is embedded
into a surface homeomorphic to an open subset of the sphere, and is said
to be \textbf{simply connected} if it is embedded into a simply
connected surface (which is necessarily homeomorphic to either the
sphere or the plane).





Let $M$ be a map with underlying graph $G$ and let $z$ be a proper embedding
of $M$ into a surface $S$. If every face of $M$ has finite degree, the \textbf{dual map} of
$M$, denoted $M^\dagger$ is defined as follows. The underlying graph of
$M^\dagger$, denoted $G^\dagger$, has the faces of $M$ as vertices, and has
an edge drawn between two faces of $M$ for each edge in $M$ that is
incident to both of the faces. We define an embedding $z^\dagger$ of
$G^\dagger$ into $S$ by placing each vertex of $G^\dagger$ in the interior of
the corresponding face of $M$ and each edge of $G^\dagger$ so that it
crosses the corresponding edge of $G$ but no others.  We define $M^\dual$
to be the map with underlying graph $G$ represented by the pair
$(S,z^\dual)$: Although the embedding $z^\dual$ is not uniquely defined,
every choice of $z^\dual$ defines the same map.  The construction gives a
canonical bijection between edges of $G$ and edges of $G^\dual$.  We write
$e^\dual$ for the edge of $G^\dual$ corresponding to $e$. If $e$ is an
oriented edge, we let $e^\dual$ be oriented so that it crosses $e$ from
right to left as viewed from the orientation of $e$.

Despite their
topological definitions, maps and their duals can in fact be defined
entirely combinatorially.  Given any graph, we consider each edge as two
oriented edges in opposite directions.  We write
$E^\rightarrow=E^\rightarrow(G)$ for the set of oriented edges of a graph.
For each directed edge $e$, we have a head $e^+$ and tail $e^-$, and write $-e$ for the reversal of $e$.  Given a map
$M$ and a vertex $v$ of $M$, let $\sigma_v=\sigma_v(M)$ be the cyclic
permutation of the set $\{e \in E^\rightarrow : e^-=v\}$ of oriented edges
emanating from $v$ corresponding to counter-clockwise rotation in $S$.
This procedure defines a bijection between maps and graphs labelled by
cyclic permutations.

\begin{thm}[\cite{LaZv04}]\label{thm:permmaps}
  Given a connected, locally finite graph $G$ and a collection of cyclic
  permutations $\sigma_v$ of the sets $\{e\in E^\rightarrow : e^-=v\}$,
  there exists a unique map $M$ with underlying graph $G$ such that
  $\sigma(M)=\sigma$.
\end{thm}

In light of \cref{thm:permmaps}, we formally identify a map $M$ with the
pair $(G,\sigma)$.  Given such a combinatorial specification of a map
$M$ as a pair $(G,\sigma)$, we may form an embedding of the map into a
surface $S(M)$ by gluing topological discs according to the
combinatorics of the map (see \cref{fig:map2}).

Note that we can write $\sigma(e)$ for $\sigma_v(e)$, since necessarily
$v=e^-$.  Thus $\sigma$ is a permutation on the set of directed edges of
$M$.  Formally, a \textbf{corner} in the map at a vertex $v$ is an
ordered pair of directed edges $(e,\sigma(e))$, with $e^-=v$.  Of
course, a corner is determined by the directed edge $e$.  Just as
vertices of $M$ are orbits of $\sigma$, the faces of a map
$M=(G,\sigma)$ can be defined abstractly as orbits of the permutation
$\sigma^\dual: E^\rightarrow \to E^\rightarrow$ defined by
$\sigma^\dual(e) = \sigma^{-1}(-e)$ for each $e \in E^\rightarrow$.  The
dual $e^\dual$ of a directed edge $e$ is defined to have the orbit of
$e$ as its tail and the orbit of $-e$ as its head, so that we again have
a bijection between directed edges of $M$ and their duals. Using this
bijection, we can consider $\sigma^\dual$ to acts on dual edges, and the
dual map $M^\dual$ is then constructed abstractly as
$M^\dual = (G^\dual,\sigma^\dual)$.  We define maps with infinite degree
vertices, and duals of maps with infinite degree faces, directly through
this abstract formalism.

\begin{figure}[t] \centering
  \includegraphics[height=25mm]{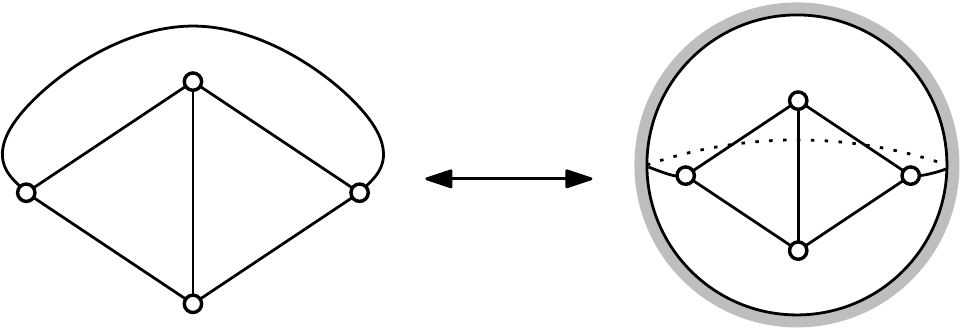}
  \hfill \includegraphics[height=25mm]{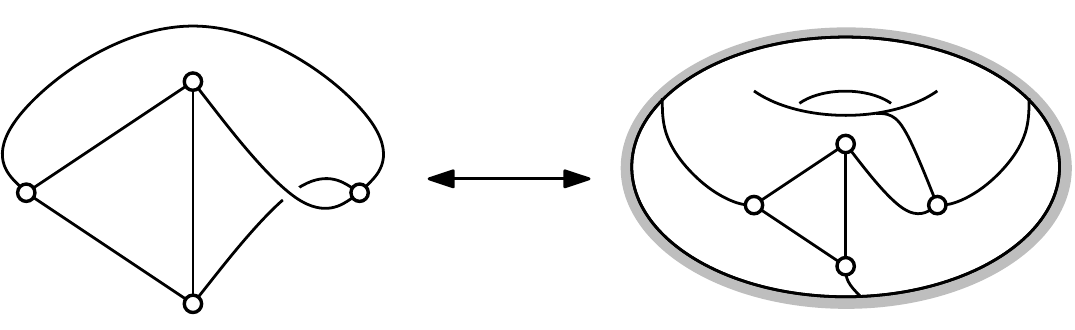}
  \footnotesize{
  \caption{Different maps with the same underlying graph.
      The two maps both have $K_4$ as their underlying graph, but
    the left is a sphere map while the right is a torus map. Both maps are represented both abstractly as a graph together with a cyclic permutation of the edges emanating from each vertex (left) and as a graph embedded in a surface (right).}
    \label{fig:map2}
    }
\end{figure}


If $e$ is an oriented edge in a map, we write $e^\ell$ for the face of $M$
to the left of $e$ and $e^r$ for the face of $M$ to the right of $e$, so
that $e^\ell=(e^\dual)^-$ and $e^r=(e^\dual)^+$. Given a map $M$, we write
$f \perp v$ if the face $f$ is incident to the vertex $v$, that is, if
there exists an oriented edge $e$ of $M$ such that $e^-=v$ and
$e^\ell=f$. When writing a sum of the form $\sum_{f \perp v}$, we use the
convention that a face $f$ is counted with multiplicity according to the
number of oriented edges $e$ of $M$ such that $e^-=v$ and $e^\ell=f$.
Similarly, when writing a sum of the form $\sum_{u\sim v}$, we count each
vertex $u$ with multiplicity according to the number of oriented edges $e$
of $M$ such that $e^-=v$ and $e^+=u$. The degree of a face is defined to be the number of oriented edges with $e^r=f$, i.e., the degree of the face in the dual.

\subsection{Unimodularity and the mass transport principle}\label{Sec:unimod}

As noted, a \textbf{rooted graph} $(G,\rho)$ is a connected, locally finite
(multi)graph $G=(V,E)$ together with a distinguished vertex $\rho$, called
the \textbf{root}.  A graph isomorphism $\phi: G \rightarrow G'$ is an
isomorphism of rooted graphs is a graph isomorphism that maps the root to
the root.

The \textbf{local topology} (see \cite{BeSc}) is the topology on the set
$\cG_\bullet$ of isomorphism classes of rooted graphs induced by the metric
\[
  d_\mathrm{loc} \left((G,\rho),(G',\rho')\right) = e^{-R},
\]
where
\[
  R = R\left((G,\rho),(G',\rho')\right)
  = \sup\left\{R\geq 0 : B_R(G,\rho) \cong B_R(G',\rho') \right\},
\]
i.e., the maximal radius such that the balls $B_R(G,\rho)$ and
$B_R(G',\rho')$ are isomorphic as rooted graphs.

A \textbf{random rooted graph} is a random variable taking values in the
space $\cG_\bullet$ endowed with the local topology.  Similarly, a
\textbf{doubly-rooted graph} is a graph together with an ordered pair of
distinguished (not necessarily distinct) vertices.  Denote the space of
isomorphism classes of doubly-rooted graphs equipped with this topology by
$\cG_{\bullet\bullet}$.

The spaces of rooted and doubly-rooted maps are defined similarly and are
denoted $\cM_\bullet$ and $\cM_{\bullet\bullet}$ respectively.  For this,
an isomorphism $\phi$ of rooted maps is preserves the roots and the map
structure, so that $\phi\circ \sigma = \sigma' \circ \phi$ for vertices
inside the balls.

A further generalisation of these spaces will also be useful. A
\textbf{marked graph} (referred to by Aldous and Lyons \cite{AL07} as a
network) is defined to be a locally finite, connected graph
together with a function $m:E \cup V \to\bbX$ assigning each vertex
and edge of $G$ a \textbf{mark} in some Polish space $\mathbb{X}$, referred
to as the \textbf{mark space}.  The local topology on the set of
isomorphism classes of rooted graphs with marks in $\bbX$ is the topology
induced by the metric
\[
  d_\mathrm{loc}((G,\rho,m),(G',\rho',m')) = e^{-R}
\]
where $R((G,\rho,m),(G',\rho',m'))$ is the largest $R$ such that there
exists an isomorphism of rooted graphs
$\phi: (B_G(\rho,b), m, \rho) \to (B_{G'}(\rho',b),\rho')$ such that
$d_{\mathbb{X}}(m'(\phi(x)),m(x))\leq 1/R$ for every vertex or edge $x$ of
$B_G(\rho,n)$, and $d_\bbX$ is a metric compatible with the topology of
$\bbX$.  The space of isomorphism classes of rooted marked graphs with
marks in $\bbX$ is denoted $\cG_\bullet^\bbX$.  The space of rooted marked
maps is defined similarly.

It is possible to consider rooted maps and rooted marked maps as rooted
marked graphs by encoding the permutations $\sigma_v$ in the marks -- in
particular, this means that any statement that holds for all unimodular
random rooted marked graphs also holds for all unimodular random rooted
marked maps.  See \cite[Example 9.6]{AL07}.


A \textbf{mass transport} is a Borel function
$f:\mathcal{G}_{\bullet\bullet} \to [0,\infty]$.  A random rooted graph
$(G,\rho)$ is said to be \textbf{unimodular} if it satisfies the
\textbf{Mass Transport Principle}: for every mass transport $f$,
\begin{equation*}\label{eq:MTP}\tag{MTP}
  \E\bigg[\sum_{v\in V} f(G,\rho,v)\bigg] =
  \E\bigg[\sum_{u\in V} f(G,u,\rho)\bigg].
\end{equation*}
That is,
\begin{center}
  \emph{Expected mass out equals expected mass in.}
\end{center}
Unimodularity of random rooted maps, marked graphs and marked maps are defined similarly.
A probability measure $\P$ on $\cG_\bullet$ is said to be unimodular if a
random rooted map with law $\P$ is unimodular.  Unimodular probability
measures on rooted maps, marked graphs and marked maps are defined
similarly.

It will also be convenient to work with signed mass transports.  If a measurable
$f:\mathcal{G}_{\bullet\bullet} \to \R$ satsifies
\[
  \E\left[\sum_{v\in V} \left|f(G,\rho,v)\right|\right] < \infty,
\]
then the conclusion of the mass transport principle holds for $f$.
This follows easily from the usual mass transport principle for
non-negative $f$.


The mass transport principle was first introduced by H\"aggstr\"om
\cite{Hagg97} to study dependent percolation on Cayley graphs.  The
formulation of the mass transport principle presented here was suggested by
Benjamini and Schramm \cite{BeSc} and developed systematically by Aldous
and Lyons \cite{AL07}.  The unimodular probability measures on
$\cG_\bullet$ form a weakly closed, convex subset of the space of
probability measures on $\cG_\bullet$, so that weak limits of unimodular
random graphs are unimodular.  In particular, a weak limit of finite graphs
with uniformly chosen roots is unimodular: such a limit of finite graphs is
referred to as a \textbf{Benjamini-Schramm limit}.  It is a major open
problem to determine whether all unimodular random rooted graphs arise as
Benjamini-Schramm limits of finite graphs \cite[\S 10]{AL07}.



\medskip

We will make frequent use of the fact that unimodularity is stable under most reasonable ways of modifying a graph locally without changing its vertex set. The following lemma is a formalization of this fact.

\begin{lemma}
\label{lem:unimodstability}
Let $\bbX_1$ and $\bbX_2$ be polish spaces.
Let $(G,\rho,m)$ be a unimodular random rooted $\bbX_1$-marked graph, and suppose that $(G',\rho,m')$ is a random rooted marked graph with the same vertex set as $G$, such that, for every pair of vertices $u,v$ in $G$, the conditional distribution of $(G',u,v,m')$ given $(G,\rho,m)$ coincides a.s.\ with some measurable function of the isomorphism class of $(G,u,v,m)$. Then $(G',\rho,m')$ is unimodular.
\end{lemma}

\begin{proof}
Let $f:\cG^{\bbX_2}_{\bullet\bullet}\to[0,\infty]$ be a mass transport.
 For each pair of vertices $u,v$ in $G$,
  the conditional expectation
  \[F(G,u,v,m)=\E\left[f(G',u,v,m')\mid (G,\rho,m)\right]\] coincides a.s.\ with a measurable function of  $(G,u,v,m)$, and so is itself a mass transport. Since $(G,\rho,m)$ is unimodular we deduce that
\begin{multline}
\label{eq:conditionalMTP}
\E\sum_vf(G',\rho,v,m') = \E\sum_v F(G,\rho,v,m)\\= \E\sum_v F(G,v,\rho,m) = \E\sum_uf(G',u,\rho,m'),
\end{multline}
verifying that $(G',\rho,m')$ satisfies the mass transport principle.
\end{proof}

Later in the paper, we will often claim that various random rooted
graphs obtained from unimodular random rooted graphs in various ways are
unimodular. These claims can always be justified either as a direct
consequence of \cref{lem:unimodstability}, or otherwise by applying the
mass transport principle to the conditional expectation as in
\cref{eq:conditionalMTP}.



\subsection{Reversibility}\label{Sec:reverse}

Recall that the \textbf{simple random walk} on a locally finite graph $G=(V,E)$ with no isolated vertices is the Markov process $\langle X_n\rangle_{n\geq0}$ on the state space $V$ with transition probabilities $p(u,v)$ defined to be the fraction of edges emanating from $u$ that end in $v$.
A probability measure $\P$ on $\cG_\bullet$ is said to be \textbf{stationary} if, when $(G,\rho)$ is a random rooted graph with law $\P$ and $\langle X_n \rangle_{n\geq 0}$ is a simple random walk on $G$ started at the root, $(G,\rho)$ and  $(G,X_n)$ have the same distribution for all $n$ (in particular we require that every vertex of $(G,\rho)$ has at least one edge a.s., since otherwise the random walk is not defined).

The measure $\P$ is said to be \textbf{reversible} if furthermore
\[ (G,\rho,X_n) \eqd (G,X_n,\rho)\]
for all $n$. A random rooted graph is said to be reversible if its law is reversible. Reversible random rooted maps, marked graphs and marked maps are defined similarly. Every reversible random rooted graph is clearly stationary, but the converse need not hold in general \cite[Examples 3.1 and 3.2]{BLPS99}.


Let $\mathcal{G}_\leftrightarrow$  be the  space of isomorphism classes of connected, locally finite graphs equipped with a distinguished
bi-infinite path, which we endow with a natural variant of the local topology.  If $(G,\rho)$ is a reversible random rooted graph and $\langle X_n\rangle_{n\geq 0}$ and $\langle X_{-n}\rangle_{n\geq 0}$ are independent simple random walks started from $\rho$, then the sequence
\[\langle (G,\langle X_{n+k} \rangle_{n \in \Z} )\rangle_{k\in \Z}\]
of $\cG_\leftrightarrow$-valued random variables is stationary.

The following correspondence between unimodular and reversible random rooted graphs is implicit in \cite{AL07} and was proven explcitly in \cite{BC2011}. Similar correspondences also hold between unimodular and reversible random rooted maps, marked graphs and marked maps.
\begin{prop}[\cite{AL07,BC2011}]
Let $\P$ be a unimodular probability measure on $\cG_\bullet$ with $\asaf{\E}[\deg(\rho)]<\infty$ and let $\P_\mathrm{rev}$ denote the $\deg(\rho)$-biasing of $\P$, defined by
\[\Prev((G,\rho)\in\sA) :=\frac{\asaf{\E}[\deg(\rho)\mathbbm{1}((G,\rho)\in\sA)]}{\asaf{\E}[\deg(\rho)]}\]
for every Borel set $\sA\subseteq \cG_\bullet$.  Then  $\Prev$ is reversible. Conversely, if $\P$ is a reversible probability measure on $\cG_\bullet$ \asaf{with $\P(\deg(\rho)>0)=1$}, then biasing $\P$ by $\deg(\rho)^{-1}$ yields a unimodular probability measure on $\cG_\bullet$.
\end{prop}

\subsection{Ergodicity}\label{Sec:ergod}

A probability measure $\P$ on $\cG_\bullet$ is said to be \textbf{ergodic} if $\P(\sA)\in\{0,1\}$ for every event $\sA \subseteq \mathcal{G}_\bullet$ that is invariant to
    changing the root in the sense that
    \[(G,\rho) \in \sA \iff (G,v)\in\sA \text{ for all } v \in V.\]
    A random rooted graph is said to be ergodic if its law is ergodic.
    Aldous and Lyons \cite[\S4]{AL07} proved that the ergodic unimodular probability measures on $\cG_\bullet$ are exactly the extreme point of the weakly closed,
    convex set of  unimodular  probability measures on. It follows by Choquet theory that every unimodular  random rooted graph is a mixture of ergodic unimodular  random rooted graphs, meaning that  the graph may be sampled by first sampling a random ergodic unimodular probability measure on $\cG_\bullet$ from some distribution,  and then sampling from this randomly chosen measure. In particular, to prove almost sure statements about unimodular  random rooted graphs, it suffices to consider the ergodic case. Analogous statements hold for unimodular random rooted maps, marked graphs, and marked maps.

\subsection{Duality}\label{Sec:dual}
Let $\P$ be a unimodular probability measure on $\cM_\bullet$ such that $M^\dagger$ is locally finite $\P$-a.s., and
let $(M,\rho)$ be a random rooted map with law $\Prev$. Conditional on $(M,\rho)$, let $\eta$ be an oriented edge of $M$ sampled uniformly at random from the set $E_\rho^\rightarrow$, and let $\rho^\dagger=\eta^r$. We define $\Prev^\dagger$ to the law of the random rooted map $(M^\dagger,\rho^\dagger)$.
\begin{prop}[Aldous-Lyons {\cite[Example 9.6]{AL07}}]\label{prop:dualmaps}
    Let $\P$ be a unimodular probability measure on $\cM_\bullet$ such that $M^\dagger$ is locally finite $\P$-a.s. and $\P[\deg(\rho)]<\infty$. Then $\Prev^\dagger$ is a reversible probability measure on $\cM_\bullet$.
\end{prop}

Intuitively, under the measure $\P_\mathrm{rev}$, the edge $\eta$ is `uniformly distributed' among the edges of $M$. Since the map sending each edge to its dual is a bijection, it should follow that $\eta^\dagger$ is uniformly distributed among the edges of $M^\dagger$, making the measure $\P^\dagger_\mathrm{rev}$ reversible also.

We define $\P^\dagger$ to be the unimodular measure on $\cM_\bullet$ obtained as the $\deg(\rho^\dagger)^{-1}$-biasing of $\Prev^\dagger$. We refer to $\P^\dagger$ as the \textbf{dual} of $\P$ and say that $\P$ is \textbf{self-dual} if $\P^\dagger=\P$. We write $\E_\mathrm{rev}$, $\E_\mathrm{rev}^\dagger$, and $\E^\dagger$ for the associated expectation operators.
We may express $\P^\dagger$ directly in terms of $\P$ by
\begin{equation} \tom{\P^\dagger\big((M^\dagger,\rho^\dagger) \in \sA\big)} = \E\bigg[\sum_{f\perp\rho}\deg(f)^{-1}\bigg]^{-1}\E\bigg[\sum_{f \perp \rho}\deg(f)^{-1}\mathbbm{1}\left((M^\dagger,f) \in \sA\right)\bigg]
\label{eq:dualmeasure}
\end{equation}
for every Borel set $\sA\subseteq \cM_\bullet$. In particular, we can calculate
\begin{align}\E^\dagger[\deg(\rho^\dagger)]&= \E\bigg[\sum_{f\perp\rho}\deg(f)^{-1}\bigg]^{-1}\E\bigg[\sum_{f \perp \rho}\deg(f)^{-1}\deg(f)\bigg]\nonumber\\
& = \E\bigg[\sum_{f\perp\rho}\deg(f)^{-1}\bigg]^{-1}\E[\deg(\rho)]<\infty.\label{eq:dualdegree}\end{align}
The factor $\E[\sum_{f\perp\rho}\deg(f)^{-1}]^{-1}$ may be thought of as the ratio of the number of vertices of $M$ to the number of faces of $M$.
\begin{example} Let $M$ be a finite map and let $\P$ be the law of the unimodular random rooted map $(M,\rho)$ where $\rho$ is a vertex of $M$ chosen uniformly at random. Then $\P^\dagger$ is the law of the unimodular random rooted map $(M^\dagger,\rho^\dagger)$ obtained by rooting the dual $M^\dagger$ of $M$ at a uniformly chosen face $\rho^\dagger$ of $M$.
\end{example}

\section{Percolations and Invariant Amenability}

\subsection{Markings and Percolations}

Let $(G,\rho)$ be a unimodular random rooted graph with law $\P$. Given
a random rooted graph $(G,\rho)$ and a polish space $\bbX$, an
\textbf{$\bbX$-marking} of $(G,\rho)$ is a random assignment
$m: E \cup V \to \bbX$ of marks to the edges and vertices of $(G,\rho)$
(possibly defined on some larger probability space) such that the random
rooted marked graph $(G,\rho,m)$ is unimodular.  A \textbf{percolation}
on $(G,\rho)$ is a $\{0,1\}$-marking of $(G,\rho)$, which we think of as
a random subgraph of $G$ consisting of the \textbf{open} edges
satisfying $\omega(e)=1$, and open vertices satisfying $\omega(v)=1$. We
assume without loss of generality that if an edge is open then so are
both of its endpoints. We call $\omega$ a \textbf{bond percolation} if
$\omega(v)=1$ for every vertex $v \in V$ almost surely.  The\textbf{
  cluster} $K_\omega(v)$ of $\omega$ at the vertex $v$ is the connected
component of $v$ in $\omega$. A percolation is said to be
\textbf{finitary} if all of its clusters are finite almost surely.  A
percolation is said to be \textbf{connected} if all open vertices form a
single cluster almost surely.

Let us gather here the following simple and useful technical lemmas
concerning markings and percolations.

\begin{lemma}\label{lem:unimodclusters}
  Let $(G,\rho,m)$ be a unimodular random marked graph. Let $\omega$ be
  a percolation on $(G,\rho,m)$ and let $m|_{K_\omega(\rho)}$ denote the
  restriction of $m$ to $K_\omega(\rho)$. Then the conditional law of
  $(K_\omega(\rho),\rho,m|_{K_\omega)\rho)})$ given that
  $\omega(\rho)=1$ is also unimodular. In particular, if $\omega$ is
  finitary, then $\rho$ is uniformly distributed on its cluster.
\end{lemma}

\begin{proof}
  Let $f$ be a mass transport, and let $g$ be the mass transport
  \[g(G,u,v,m,\omega) := \mathbbm{1}\left(v \in K_\omega(u)\right)
    f(K_\omega(u),u,v,m|_{K_\omega(u)}).
  \]
  Applying the mass transport principle for $(G,\rho,\omega)$, we have
  \begin{align*}
    \E\sum_{v \in K_\omega(\rho)} f(K_\omega(\rho),\rho,v,m|_{K_\omega(\rho)})
    &= \E\sum_{v \in V(G)}g(G,\rho,v ,m,\omega)
    = \E\sum_{v \in V(G)}g(G,v,\rho ,m,\omega)\\
    &= \E\sum_{v \in K_\omega(\rho)}
    f(K_\omega(\rho),v,\rho,m|_{K_\omega(\rho)}). \qedhere
  \end{align*}
\end{proof}

The next two lemmas are proved by abstract probabilistic arguments.
They state (roughly) that multiple markings on a graph can be combined,
and that a marking on a percolation cluster can be extended to a marking
on the entire graph.

\begin{lemma}[Coupling markings I]\label{lem:couplingmarks}
  Let $(G,\rho)$ be a unimodular random rooted graph and let
  $\{m_i : i \in I\}$ be a set of markings of $(G,\rho)$ indexed by a
  countable set $I$ with mark spaces $\{\bbX_i:i\in I\}$.
  Then there exists a $\prod_{i \in I}\bbX_i$-marking $m$ of
  $(G,\rho)$ such that $(G,\rho,\pi_i\circ m)$ and $(G,\rho,m_i)$ have
  the same distribution for all $i\in I$, where $\pi_i$ denotes the
  projection of $\prod_{i\in I}\bbX_i$ onto $\bbX_i$ for each $i\in I$.
\end{lemma}

\begin{proof}
  One such marking is given by first sampling $(G,\rho)$ and then
  sampling the markings $m_i$ independently from their conditional
  distributions given $(G,\rho)$.  The details of this construction are
  are omitted.
\end{proof}

\begin{lemma}[Coupling markings II]\label{lem:couplingmarks2}
  Let $(G,\rho)$ be a unimodular random rooted graph, let $\omega$ be a
  connected percolation on $(G,\rho)$, and let $m$ be an $\bbX$-marking
  of the unimodular random rooted graph given by sampling
  $(K_\omega(\rho),\rho)$ conditional on $\omega(\rho)=1$. Then there
  exists a $\bbX$-marking $\hat m$ of $(G,\rho)$ such that the laws of
  $(K_\omega(\rho),\rho,\hat m|_{K_\omega(\rho)})$ and
  $(K_\omega(\rho),\rho,m)$ coincide.
\end{lemma}

\begin{proof}
  One such marking is given as follows: First sample
  $(K_\omega(\rho),\rho)$ from its conditional distribution given
  $\omega(\rho)=1$. Then, independently, sample $(G,\rho)$ and $m$
  independently conditional on $(K_\omega(\rho),\rho)$. Extend $m$ to the
  vertices and edges of $G$ not present in $\omega$ by setting $m$ to be
  some constant $x_0\in \bbX$ on those vertices.  This yields the law of a
  random rooted marked graph $(G,\rho,\omega,m)$ in which the root always
  satisfies $\omega(\rho)=1$.

  Now, for each vertex $u$ of $G$, let $v(u)$ be chosen uniformly from the
  set of vertices of $\omega$ closest to $u$, independently from everything
  else. (In particular, $v(u)=u$ if $\omega(u)=1$.) Note that (by the mass
  transport principle), the expectation of $|\{ u : v(u)=\rho\}|$ is
  $\P(\rho\in\omega)^{-1}$ and so is finite. Sample the marked graph
  $(G,\rho,\omega,m)$ biased by $|\{ u : v(u)=\rho\}|$, and let $\rho'$ be
  uniform in $\{ u : v(u)=\rho\}$.  Then $(G,\rho',\omega,m)$
  yields the desired coupling.
\end{proof}

\subsection{Amenability}

Recall that the \textbf{(edge) Cheeger constant} of an infinite graph
$G=(V,E)$ is defined to be
\begin{equation*}
  \mathbf{i}_E(G)=\inf \left\{\frac{|\partial_E W|}{|W|}
    : W \subset V \text{ finite}\right\}
\end{equation*}
where $\partial_E W$ denotes the set of edges with exactly one end in
$W$. The graph is said to be \textbf{amenable} if its Cheeger constant
is zero and \textbf{nonamenable} if it is positive.


Nonamenability is often too strong a condition to hold for random graphs.
For example, it is easily seen that every infinite cluster of a
Bernoulli percolation on a nonamenable Cayley graph is almost surely
amenable, since the cluster will contain atypical `bad regions' that
have small expansion.  In light of this, Aldous and Lyons introduced the
following weakened notion of nonamenability for unimodular random
graphs.  (Another way to overcome this issue is to use \emph{anchored
  expansion}, which is less relevant in our setting. See \cite{AL07} for
a comparison of the two notions.)
The \textbf{invariant Cheeger constant} of an ergodic unimodular random
rooted graph $(G,\rho)$ is defined to be
\begin{equation}
  \mathbf{i}^{\mathrm{inv}}(G,\rho)= \inf \left\{
    \E\left[ \frac{|\partial_E K_\omega(\rho)|}{|K_\omega(\rho)|}\right]
    : \omega \text{ a finitary percolation on $G$}
  \right\}.
\end{equation}
(This is a slight abuse of notation:
$\mathbf{i}^{\mathrm{inv}}(G,\rho)$ is really a function of the
\emph{law} of $(G,\rho)$.)  An ergodic unimodular random rooted graph
$(G,\rho)$ is said to be \textbf{invariantly amenable} if
$\mathbf{i}^{inv}(G,\rho)=0$ and \textbf{invariantly nonamenable}
otherwise. More generally, we say that a unimodular random rooted graph
is invariantly amenable if its ergodic decomposition is supported on
invariantly amenable unimodular random rooted graphs, and invariantly
nonamenable if its ergodic decomposition is supported on invariantly
nonamenable unimodular random rooted graphs. This is a property of the
law of $(G,\rho)$ and not of an individual graph. Invariant amenability
and nonamenability is defined similarly for ergodic unimodular random
rooted maps, marked graphs and marked maps.

Let $(G,\rho)$ be a unimodular random rooted graph and let $\omega$ be a
finitary percolation on $G$.  Let $\deg_\omega(\rho)$ denote the degree
of $\rho$ in $\omega$ and let
\[
  \alpha(G,\rho)
  = \sup \left\{\E\!\left[\deg_\omega(\rho)\right] : \omega \text{ a finitary
      percolation on $G$}\right\}
\]
An easy application of the mass transport principle \cite[Lemma
8.2]{AL07} shows that for any finitary percolation
\[
  \E[\deg_\omega(\rho)]
  = \E\left[\frac{\sum_{v \in K_\omega(\rho)}
      \deg_\omega(v)}{|K_\omega(\rho)|}\right]
  = \E\left[\frac{\sum_{v \in K_\omega(\rho)} \deg(v) -
      |\partial_E K_\omega(\rho)|}{|K_\omega(\rho)|}\right].
\]
It follows that, if $\E[\deg(\rho)]<\infty$,
\[
  \mathbf{i}^\mathrm{{inv}}(G,\rho) = \E[\deg(\rho)] - \alpha(G,\rho).
\]
Similarly, if $(G,\rho)$ is an ergodic unimodular random rooted graph
with $\E[\deg(\rho)]=\infty$, then $\alpha(G,\rho)<\infty$ is a
sufficient condition for $\mathbf{i}^\mathrm{inv}(G,\rho)$ to be
positive.

\subsection{Hyperfiniteness}


A closely related notion to amenability is hyperfiniteness.  A
unimodular random rooted graph $(G,\rho)$ is said to be
\textbf{hyperfinite} if there exists a sequence of percolations
$(\omega_i)_{i\geq1}$ of $(G,\rho)$ such that each of the percolations
$\omega_i$ is finitary, $\omega_i \subseteq \omega_{i+1}$ almost surely,
and $\bigcup_{i\geq1}\omega_i = G$ almost surely.  We call such a
sequence a \textbf{finitary exhaustion} of $(G,\rho)$.  The following is
standard in the measured equivalence relations literature and was noted
to carry through to the unimodular random rooted graph setting by Aldous
and Lyons \cite{AL07}.

\begin{thm}[{\cite[Theorem 8.5]{AL07}}]\label{prop:invexh}
  Let $(G,\rho)$ be a unimodular random rooted graph with
  $\E[\deg(\rho)]<\infty$.  Then $(G,\rho)$ is invariantly amenable if
  and only if it hyperfinite.
\end{thm}

We provide a short proof below for completeness.

\begin{proof}
  First suppose that $(G,\rho)$ is hyperfinite, and let
  $(\omega_i)_{i\geq0}$ be a finitary exhaustion of $(G,\rho)$.  By the
  monotone convergence theorem,
  $\E[\deg_{\omega_n}(\rho)] \to \E[\deg(\rho)]$, so that
  \[
    \alpha(G,\rho) \geq \limsup_{i\to\infty}\E[\deg_{\omega_i}(\rho)]=
    \E[\deg(\rho)]
  \]
  and hence $(G,\rho)$ is invariantly amenable.  Suppose conversely that
  $(G,\rho)$ is invariantly amenable. For each $i\geq1$, there exists a
  finitary percolation $\omega_i$ on $G$ such that
  $\E [\deg(\rho)-\deg_{\omega_i}(\rho)] \leq 2^{-i}$.
  By \cref{lem:couplingmarks}, these can be coupled.
  For each $i\geq 1$, let $\hat \omega_i = \bigcap_{j\geq i} \omega_j$.
  Then $\hat\omega_i$ is an increasing sequence of finitary
  percolations.  Furthermore, by construction,
  \[
    \E [\deg(\rho)-\deg_{\hat \omega_i}(\rho)] \leq
    \sum_{j\geq i} \E [\deg(\rho)-\deg_{\omega_i}(\rho)] \leq 2^{-i+1}.
  \]
  By Borel-Cantelli, $\rho$ is in the interior of its $\omega_i$-cluster
  for all sufficiently large $i$ almost surely. It follows by
  unimodularity that $\bigcup \hat\omega_i = G$, so that
  $(\hat\omega_i)_{i\geq0}$ is a finitary exhaustion of $(G,\rho)$.
\end{proof}

\begin{remark}
  Invariantly amenable unimodular random rooted graphs are always
  hyperfinite whether or not $\E[\deg(\rho)]<\infty$.  However, the
  graph obtained by replacing each edge of the canopy tree (i.e., the
  Benjamini-Schramm limit of the balls in a $3$-regular tree) at height
  $n$ by $2^n$ parallel edges is hyperfinite but nonamenable (since we
  use edge boundaries).
\end{remark}


\begin{corollary}[\cref{thm:dichotomy}, \eqref{iAmen} implies \eqref{iBS}]
  \label{cor:invariantexhaustion}
  Let $(M,\rho)$ be a hyperfinite ergodic random rooted map. Then
  $(M,\rho)$ is a Benjamini-Schramm limit of finite planar maps.
\end{corollary}

\begin{proof}
  Let $(\omega_i)_{i\geq1}$ be a finitary exhaustion of $(M,\rho)$.  Let
  $M_i$ be the finite map with underlying graph $K_{\omega_i}(\rho)$ and
  map structure inherited from $M$.  Then $(M_i,\rho_i)$ is a finite
  unimodular random rooted map, with the root is uniform in $M_i$ (by By
  \cref{lem:unimodclusters}).  Moreover, $(M_i,\rho)$ converges to
  $(M,\rho)$ almost surely and hence also in distribution.
\end{proof}

\begin{thm}[\cref{thm:dichotomy}: \eqref{iBS} implies \eqref{iAmen}]
  Let $(M,\rho)$ be an ergodic unimodular random planar map that is a
  Benjamini-Schramm limit of finite planar maps. Then $(M,\rho)$ is
  hyperfinite.
\end{thm}

This is a corollary to the Lipton-Tarjan planar separator theorem
\cite{LT80}; see \cite{AHNR15} for details.  Similarly, it is possible
to deduce that item \eqref{iBS2} of \cref{thm:dichotomy} implies item
\eqref{iAmen} as an application of the low-genus separator theorem of
Gilbert, Hutchinson, and Tarjan~\cite{gilbert1984separator}.

\subsection{Ends}

Recall that an infinite connected graph $G=(V,E)$ is said to be
\textbf{$k$-ended} if $k$ is minimal such that for every finite set $W$
of vertices in $G$, the graph induced by the complement $V\setminus W$
has at most $k$ distinct infinite connected components.  In particular,
an infinite tree is one-ended if it does not contain a simple
bi-infinite path, and is two-ended if it contains a unique bi-infinite
path.

The following proposition, due primarily to Aldous and Lyons
\cite{AL07}, connects the number of ends of a unimodular random rooted
graph to invariant amenability.

\begin{prop}[Ends and Amenability \cite{AL07,AHNR15}]
  \label{prop:endsamen}
  Let $(G,\rho)$ be an infinite ergodic unimodular random rooted graph
  with $\E[\deg(\rho)]<\infty$.  Then $G$ has either one, two, or
  infinitely many ends almost surely.  If $G$ has infinitely many ends,
  then $(G,\rho)$ is invariantly nonamenable.  If $G$ has two ends
  almost surely, then it is almost surely recurrent and $(G,\rho)$ is
  invariantly amenable.
\end{prop}

Similarly, a connected topological space $X$ is said to be $k$-ended if
over all compact subsets $K$ of $X$, the complement $X\setminus K$ has a
maximum of $k$ connected components that are not precompact.

\begin{lem}\label{Lem:endscomparison}
  Let $M$ be a map. Then the underlying graph of $M$ has at least as
  many ends as the associated surface $S=S(M)$.
\end{lem}

Note that $M$ can have more ends than $S(M)$.  For example the doubly
infinite path has $2$ ends, and the 3-regular tree is infinitely-ended,
but both associated surfaces are homeomorphic to the plane and hence
one-ended.

\begin{proof}
  For each compact subset $K$ of $S$, let $W_K$ be the set
  of vertices of $M$ that are adjacent to an edge intersecting $K$.
  Since every face of $M$ in $S$ is a topological disc, each
  non-precompact connected component of $S\setminus K$ contains
  infinitely many edges of $M$, and there are no connections in
  $M\setminus W_K$ between these components.  Thus $M \setminus W_K$
  has at least as many infinite connected components as $S\setminus K$
  has non-precompact connected components.
\end{proof}

Finally, let us note the following simple topological fact.

\begin{lem}\label{lem:endsfaces}
  Let $M$ be a map.  Then the underlying graph of $M$ has at least as
  many ends as $M$ has faces of infinite degree.
\end{lem}

\subsection{Unimodular Couplings}

Let $(G_1,\rho_1)$ and $(G_2,\rho_2)$ be unimodular random rooted
graphs.  A \textbf{unimodular coupling} of $(G_1,\rho_1)$ and
$(G_2,\rho_2)$ is a unimodular random rooted $\{0,1\}^2$-marked graph
$(G,\rho,\omega_1,\omega_2)$ such that $\omega_1$ and $\omega_2$ are
both connected and the law of the subgraph $(\omega_i,\rho)$ conditioned
on the event that $\omega_i(\rho)=1$ is equal to the law of
$(G_i,\rho_i)$ for each $i$.  We say that two unimodular random rooted
graphs $(G_1,\rho_1)$ and $(G_2,\rho_2)$ are \textbf{coupling
  equivalent} if they admit a unimodular coupling.  (It is not difficult
to show that this is an equivalence relation, arguing along similar
lines to \cref{lem:couplingmarks,lem:couplingmarks2}.)  Coupling
equivalence is closely related to the notion of two graphings generating
the same equivalence relation in the theory of measured equivalence
relations (see \cite[Example 9.9]{AL07}).
%
Unimodular maps are defined to be coupling equivalent if the underlying
graphs are.  Equivalence of marked graphs and maps is defined similarly.

\begin{example}
  Let $(G,\rho)$ be a unimodular random rooted graph and let $\omega$ be
  a connected bond percolation on $G$.  Then $(G,\rho)$ and
  $(\omega,\rho)$ are coupling equivalent.
 \end{example}

\begin{example}
  Let $(M,\rho)$ be a unimodular random rooted map with locally finite
  dual.  Then $(M,\rho)$ is coupling equivalent to its unimodular dual:
  both $(M,\rho)$ and its dual can be represented as percolations of the
  graph formed by combining $M$ and its dual by adding a vertex wherever
  a primal edge crosses a dual edge, as well as keeping the original
  primal and dual edges.  (see \cite[Example 9.6]{AL07} for a similar
  construction).  Note that the graph containing $M$ and $M^\dual$ is
  not a map itself.
\end{example}



The main use of unimodular coupling will be that hyperfiniteness and
strong soficity are both preserved under unimodular coupling. The
corresponding statements for measured equivalence relations are due to
Elek and Lippner \cite{ElekLipp10}.

\begin{prop}\label{lem:hypcoupling}
  Let $(G_1,\rho_1)$ and $(G_2,\rho_2)$ be coupling equivalent
  unimodular random rooted graphs.  Then $(G_2,\rho_2)$ is hyperfinite
  if and only if $(G_1,\rho_1)$ is.
\end{prop}

\begin{proof}
  It suffices to prove that if $(G,\rho)$ is a unimodular random rooted
  graph and $\omega$ is a connected percolation on $(G,\rho)$, then
  $(G,\rho)$ is hyperfinite if and only if $(\omega,\rho)$ is
  hyperfinite.

  Suppose first that $(G,\rho)$ is hyperfinite and let
  $(\tilde \omega_i)_{i\geq1}$ be a finitary exhaustion of $G$. By
  \cref{lem:couplingmarks}, we may couple $\omega$ and
  $(\omega_i)_{i\geq 1}$ so that
  $(G,\rho,(\omega,(\tilde \omega_i)_{i\geq1}))$ is unimodular.  Under
  such a coupling, $(\omega \cap \tilde\omega_i)_{i \geq 1}$ is a finitary exhaustion of $(\omega,\rho)$, and consequently
  $(\omega,\rho)$ is hyperfinite.

  Suppose conversely that $(\omega,\rho)$ is hyperfinite, and let
  $(\tilde \omega_i)_{i\geq1}$ be a finitary exhaustion of 
  $(\omega,\rho)$.  By \cref{lem:couplingmarks2}, we may assume that
  $(G,\rho,\omega,(\tilde \omega_i)_{i\geq1})$ is unimodular.
  For each vertex $v$ of $G$, let $w(v)$ be a chosen uniformly from the
  set of vertices of $\omega$ minimising the graph distance to $v$ in
  $G$.  For each $i$, define a subgraph $\hat \omega_i$ of $G$ by
  \begin{equation}
    \hat \omega_i(u,v) = 1 \iff w(u) \text{ and } w(v) \text{ are
      connected in $\tilde \omega_i$}.
  \end{equation}
  Then $(\hat \omega_i)_{i\geq1}$ is a finitary exhaustion of $(G,\rho)$.
\end{proof}

\begin{remark}
  It can be deduced from \cref{lem:hypcoupling} and \cite[Theorem
  8.9]{AL07} that a unimodular random rooted graph with finite expected
  degree is hyperfinite if and only if it is coupling equivalent to
  $\Z$.
\end{remark}

One useful application of \cref{lem:hypcoupling} is the following.

\begin{lem}\label{lem:hyperfinitefaces} Let $(M,\rho)$ be an ergodic
  unimodular random rooted map. The number of infinite degree faces is
  either $0$, $1$, $2$, or $\infty$, and if it is $1$ or $2$ then
  $(M,\rho)$ is hyperfinite.
 \end{lem}

\begin{proof}
  By applying \cref{prop:endsamen,lem:endsfaces}, it suffices to prove
  that if $M$ has a non-zero but finite number of infinite degree faces
  almost surely, then $(M,\rho)$ is hyperfinite.  Suppose that $M$ has
  finitely many infinite degree faces almost surely.  Conditional on
  $(M,\rho)$, choose one infinite face, $f$, uniformly at random.  The
  boundary of this face gives a coupling of $M$ and $\Z$, except for the
  fact that this boundary need not be simple.

  If for every vertex $v$ of $f$ there exists a vertex $v'$ of $f$ such
  that $v$ is contained in a finite component of $M \setminus \{v'\}$,
  then inductively there exists a sequence of vertices $(v_n)_{n\geq1}$
  of $M$ such that $v_n$ is contained in a finite connected component of
  $M\setminus\{v_{n+1}\}$ for every $n\geq1$.  We deduce in this case
  that $M$ is hyperfinite by taking as a finitary exhaustion the sequence
  of random subgraphs $(\tilde \omega_m)_{m\geq1}$ induced by the
  standard monotone coupling of Bernoulli $1-1/m$ site percolations on
  $M$.  All clusters of $\omega_m$ are almost surely finite due to the
  existence of the cutpoints $(v_n)_{n\geq1}$.

  Otherwise, define a percolation $\omega$ on $(M,\rho)$ as follows.
  Vertices of $f$ that are not separated from infinity by any other
  vertex of $f$ are open, as are the edges of $f$ between such vertices,
  and no other vertices or edges.
  Then $\omega$ is connected and is isomorphic to the bi-infinite line
  graph $\Z$. It follows that $(G,\rho)$ is coupling equivalent to
  $(\Z,0)$ and hence hyperfinite by \cref{lem:hypcoupling}
\end{proof}



\subsection{Vertex Extremal Length and Recurrence of Subgraphs}\label{sec:vel}


Vertex extremal length was introduced by He and Schramm \cite{HeSc} and
is closely connected to circle packing.  Let $G$ be an infinite graph.
For a vertex $v$ of $G$, the \textbf{vertex extremal length} from $v$ to
infinity is defined to be
\begin{equation}
  \label{eq:VELdef}
  \mathrm{VEL}_G(v,\infty) = \sup_m\frac{\inf_{\gamma:v\to\infty}
    m(\gamma)^2}{\|m\|^2},
\end{equation}
where the supremum is over measures $m$ on the vertex set of $G$ such
that $\|m\|^2 = \sum m(u)^2 < \infty$, and the infimum is over paths
$\gamma$ from $v$ to $\infty$ in $G$.  A connected graph is said to be
\textbf{VEL parabolic} if $\mathrm{VEL}(v \to \infty) = \infty$ for some
vertex $v$ of $G$ (and hence for every vertex), and \textbf{VEL
  hyperbolic} otherwise.  As noted in \cite{HeSc}, if $G$ is VEL
parabolic then there exists a vertex measure $m$ where $\|m\|<\infty$
and $m(\gamma)=\infty$ for every path $\gamma:v\to\infty$.

The \textbf{edge extremal length} from $v$ to infinity is defined in the
same way except that the measure $m$ is on the edges of $G$, and is
equal to the \textbf{effective resistance} from $v$ to infinity.  In
particular, $G$ is recurrent if and only if
$\Reff(v,\infty) = \EEL(v,\infty) = \infty$ \cite{Duffin62}. See
\cite{LP:book} for further background.

The VEL type is monotone in the sense that subgraphs of VEL parabolic
graphs are also VEL parabolic.  Consequently, the following two lemmas
give the corresponding implications in \cref{thm:dichotomy}.

\begin{lem}[\cite{HeSc}, \cref{thm:dichotomy},
  \eqref{iVEL} implies \eqref{iRecurrent}]
  Let $G$ be a locally finite, connected graph.  If $G$ is VEL
  hyperbolic, then it is transient.  If $G$ is transient and has
  bouneded degrees then it is VEL hyperbolic.
\end{lem}

For trees, as for graphs with bounded degrees, VEL parabolicity is
equivalent to recurrence:

\begin{lem}[\cref{thm:dichotomy}, \eqref{iVEL} implies \eqref{iTree}]
  Let $T$ be a tree.  Then $T$ is transient if and only if it is VEL
  hyperbolic.
\end{lem}


\begin{proof}
  If $T$ is VEL hyperbolic then it is transient by \cite{HeSc}.  Suppose
  conversely that $T$ is VEL parabolic.  Then there is a vertex measure
  $m$ on $T$ with $\|m\| < \infty$ and $m(\gamma)=\infty$ for every
  infinite path from $v$.  For each edge $e$ of $T$, let $u(e)$ be the
  endpoint of $e$ farther from $v$, and define an edge measure $\hat m$
  on $T$ by setting $\hat m(e)=m(u(e))$ for every edge $e$ of $T$.  Then
  $\|\hat m\|^2 = \|m\|^2 - m(v)^2 <\infty$, and every simple path
  $\gamma:v\to\infty$ has $\hat m(\gamma) = m(\gamma)-m(v) = \infty$.
  Thus $\Reff(v,\infty)=\infty$ and $G$ is recurrent.
\end{proof}

Benjamini and Schramm \cite{BeSc} used circle packing to prove the
following remarkable theorem.  See \cite{AHNR15} for an alternative
proof, and \cite{GN12,angel2016half} for related results establishing recurrence in certain unbounded degree cases.

\begin{thm}[\cite{BeSc}; \cref{thm:dichotomy},  \eqref{iBS} implies
  \eqref{iVEL}]
  Let $(M,\rho)$ be a Benjamini-Schramm limit of finite planar maps.
  Then $M$ is VEL parabolic almost surely.  In particular, if $M$ has
  bounded degrees, then it is almost surely recurrent for simple random
  walk.
\end{thm}

We finish this section with the following theorem.

\begin{thm}[Benjamini, Lyons and Schramm \cite{BLS99}; Aldous and Lyons
  \cite{AL07}: \cref{thm:dichotomy}, \eqref{iAmenablesub} and \eqref{iTree2} each imply
  \eqref{iAmen}]
  Let $(G,\rho)$ be an invariantly nonamenable unimodular random rooted
  graph.  Then there exists a percolation $\omega$ on $G$ such that
  a.s.\ every connected component of $\omega$ is nonamenable, and there
  exists a constant $M$ such that $\deg(v) \leq M$ for every vertex $v$
  of $G$ such that $\omega(v)=1$.  Furthermore, the percolation $\omega$
  can be taken to be a forest.
\end{thm}

In particular, if a unimodular random map is nonamenable then it
contains nonamenable sub-trees with bounded degrees, which in turn are
VEL hyperbolic.
This theorem has been very useful for studying random walks on
invariantly nonamenable unimodular random rooted graphs; see
\cite[Section 5.1]{AHNR15}.  See \cite[Section 3.3.1]{AHNR15} for a
complete proof of the first part of the theorem (in which $\omega$ is
not taken to be a forest).

\section{Curvature}
\label{Sec:curvature}

In this section, we study the average curvature of a unimodular
random rooted map, and its basic properties.  Essentially, the
average curvature is the average Gaussian curvature per
vertex when the map is embedded on a Riemannian manifold.  We begin by
giving a combinatorial definition of the average curvature.  We then
show that the average curvature is a canonical quantity associated to
the random map, in the sense that any unimodular embedding the
map in a Riemannian manifold (satisfying certain integrability
conditions) will result in the same average curvature.

\medskip

Recall that the internal angles of a regular $k$-gon are given by
$(k-2)\pi/k$.  We define the \textbf{angle sum} at a vertex $v$ of a map
$M$ to be
\[
  \theta(v) =\theta_M(v) = \sum_{f \perp v}\frac{\deg(f)-2}{\deg(f)}\pi.
\]
(Recall that this sum counts each face with the appropriate
multiplicity.)  This definition extends to maps with infinite faces,
with the convention that $(\infty-2)/\infty = 1$.  In the case that
every face of $M$ has degree at least 3, embed $M$ in a manifold by gluing together regular polygons (with the $\infty$-gon being the upper half-space $\{x+iy \in \C : y>0\}$ with edges $\{[n,n+1]:n \in \Z\}$).
We then have that $\theta(v)$ is the total angle
of the corners at $v$.

Of course, in general $M$ cannot necessarily be drawn in the Euclidean plane so that all faces are regular polygons, and the angle sum at a vertex of $M$ need not be
$2\pi$.  We define the \textbf{curvature} of $M$ at the vertex $v$ to be
the angle sum deficit
\[
\kappa(v) = \kappa_M(v) = 2\pi - \theta(v).
\]
See \cref{fig:curv} for examples.
This combinatorial definition of curvature is well-known and (in the deterministic setting) has been extensively studied in the literature, see e.g.\ \cite{chen2009gauss,Aleks,higuchi2001combinatorial} and references therein.
We define the \textbf{average curvature} of a unimodular random rooted map $(M,\rho)$, denoted $\bbK(M,\rho)$, to be the expected curvature at the root:
\[
\bbK(M,\rho) = \E[\kappa(\rho)] = 2\pi - \E\left[\sum_{f \perp
    \rho}\frac{\deg(f)-2}{\deg(f)}\pi\right].
\]
Note that if $\E[\deg(\rho)]$ is finite then $\bbK(M,\rho)$ is also finite.

 \begin{figure}[t]
   \centering
   \includegraphics[width = 2.75cm]{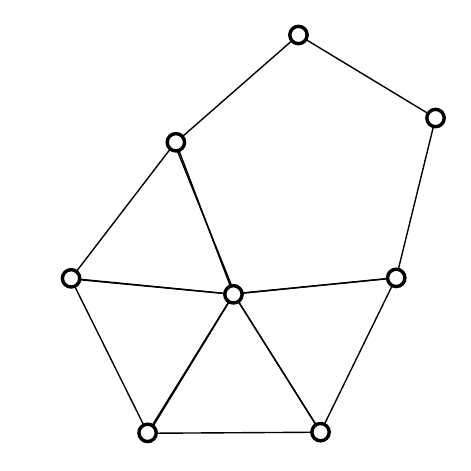}
   \qquad
   \includegraphics[width = 2.75cm]{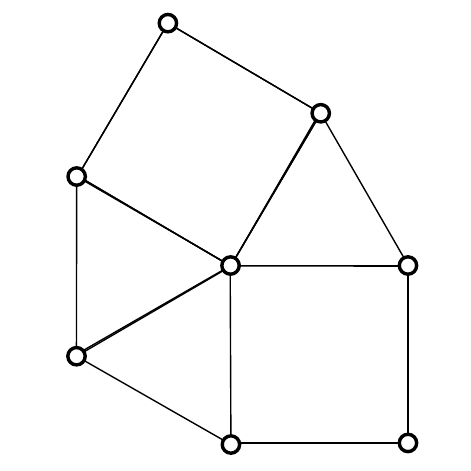}
   \qquad
    \includegraphics[width = 2.75cm]{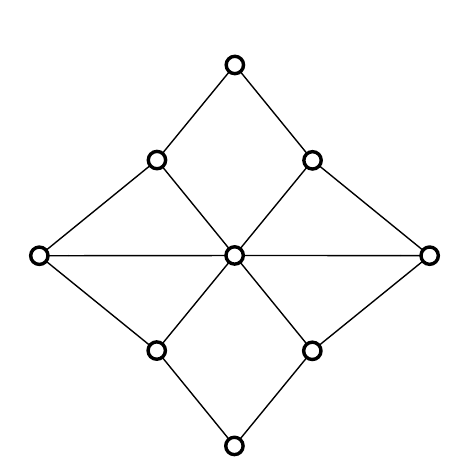}
   \caption{The centre vertices have curvature $\pi/15$, $0$ and $-\pi/3$
     (left to right).}
   \label{fig:curv}
 \end{figure}

\begin{example}[Finite Maps]\label{ex:finite}
  Let $M$ be a finite map and let $\rho$ be a vertex of $M$ chosen
  uniformly at random.  Then a simple calculation using Euler's formula
  gives that
  \begin{equation}
    \label{eq:Euler}
    \bbK(M,\rho) = \frac{1}{|V|}\sum_{v\in V} \kappa(v) =
    \frac{2\pi}{|V|}\left(2-2\genus(M)\right).
  \end{equation}
  %

  If $S$ is a Riemannian surface of genus $g$, and $\kappa(\cdot)$ is
  the associated Gaussian curvature, then the Gauss-Bonnet theorem
  implies that $\int_S \kappa(x)\,dx = 2\pi(2-2g)$.
  Consequently, for an arbitrary Riemannian metric on $S(M)$ we have
  \[
    \bbK(M,\rho) = \frac{1}{|V|}\int_S\!\kappa(x)\,dx.
  \]
  Thus, our definition of curvature agrees with the average Gaussian
  curvature per vertex, and is independent of the choice of metric.
  This is a special case of \cref{thm.invcurvature} below.
\end{example}

\begin{example}[$k$-angulations]
  If $M$ is a $k$-angulation (i.e., every face of $M$ has degree $k$), then
  \[\theta(v)=\frac{k-2}{k}\deg(v)\pi\]
  for every vertex $v$ of $M$.
  Consequently, if $(M,\rho)$ is a unimodular random $k$-angulation, then
  \[
    \bbK(M,\rho)=\left(2-\frac{k-2}{k}\E[\deg(\rho)]\right)\pi.
  \]
  A special case corresponding to $k=\infty$, is when $M$ is an infinite
  plane tree, the angle sum at a vertex $v$ of $M$ is simply
  $\theta(v)=\pi\deg(v)$, and the average curvature is
  $\pi(2-\E\deg(\rho))$.
\end{example}

\begin{example}[Curvature of the dual measure]
  Let $\P$ be a unimodular probability measure on $\cM_\bullet$ such
  that $\P[\deg(\rho)]<\infty$ and $M$ has locally finite dual
  $\P$-a.s., and let $\P^\dagger$ be the dual measure as defined
  above. By
  \eqref{eq:dualmeasure},
  \begin{align*} \bbK(M^\dagger,\rho^\dagger) &= 2\pi - \E^\dagger\bigg[\sum_{v \perp \rho^\dagger}\frac{\deg(v)-2}{\deg(v)}\bigg]\pi\\
  &= 2\pi-\E\Big[\sum_{f \perp \rho}\frac{1}{\deg(f)}\Big]^{-1}\E\bigg[\sum_{f\perp\rho}\frac{1}{\deg(f)}\sum_{v \perp f}\frac{\deg(v)-2}{\deg(v)}\bigg]\pi.\end{align*}
  Applying the mass transport principle yields that
  \begin{align*}
    \bbK(M^\dagger,\rho^\dagger) &= 2\pi-\pi \E\Big[\sum_{f
      \perp \rho} \frac{1}{\deg(f)}\Big]^{-1}
    \E\Big[\sum_{f\perp\rho}\frac{1}{\deg(f)}\sum_{v
      \perp f}\frac{\deg(\rho)-2}{\deg(\rho)}\Big].
  \end{align*}
  This can be rearranged to give
  \begin{align*}
    \bbK(M^\dagger,\rho^\dagger)&= \E\Big[\sum_{f \perp
      \rho}\frac{1}{\deg(f)}\Big]^{-1}\bbK(M,\rho)
  \end{align*}
  Recall that $\E\Big[\sum_{f \perp \rho}\frac{1}{\deg(f)}\Big]^{-1}$ is
  interpreted as the ratio between vertices of $M$ to vertices of
  $M^\dual$.  While the average curvature per vertex is changed, this is
  only since the density of vertices is changed.  In particular, the
  average curvature of $(M,\rho)$ under $\P$ has the same sign as that
  of $(M^\dagger,\rho^\dagger)$ under $\P^\dagger$.
\end{example}

\begin{example}[Self-dual maps]
  Let $(M,\rho)$ be a self-dual unimodular random rooted map. Then
  \eqref{eq:dualdegree} implies that
  $\E[\sum_{f\perp\rho}\deg(f)^{-1}] = 1$ and so
  \begin{align*}
    \bbK(M,\rho)
    &= 2\pi - \E\left[\sum_{f\perp\rho}\frac{\deg(f)-2}{\deg(f)}\right]\pi\\
    &=2\pi-\E[\deg(\rho)]\pi+2\E\Big[\sum_{f\perp\rho}\deg(f)^{-1}\Big]\pi=
    4\pi - \E[\deg(\rho)]\pi.
  \end{align*}
\end{example}



\subsection{Curvature of submaps}

Let $M$ be a map with underlying graph $G$, and let $z$ be a proper
embedding of $M$ into an orientable surface $S$.  A \textbf{submap} of
$M$ is a map represented by a triple $(H,S,z)$, where $H$ is a connected
subgraph of $G$ such that the restriction of $z$ to $H$ is a proper
embedding of $H$ into $S$ (in particular, the faces are topological discs).
 If $M$ is simply
connected, then every connected subgraph of $G$ is also a submap of
$M$. (This may fail if $M$ is not simply connected.)

Since the average curvature is curvature per vertex, it is not
surprising that changing the map without adding or removing vertices
does not change $\bbK$.  Indeed, in the case of a finite map, deleting
edges gives an embedding of the submap on the same surface, with the
same total and hence same average curvature. Recall that we call a percolation $\omega$ on $(M,\rho)$ a bond percolation if $\omega(v)=1$ for every vertex $v$ of $M$ almost surely.

\begin{prop}[Curvature of random submaps]\label{prop:submaps}
  Let $(M,\rho)$ be a unimodular random rooted map with
  $\E[\deg(\rho)]<\infty$ and let $\omega$ be a connected bond percolation on $(M,\rho)$ that is almost surely a submap of $M$.  Then
  \[
    \bbK(\omega,\rho) = \bbK(M,\rho).
  \]
\end{prop}

\begin{remark}
  It is straightforward to extend of \cref{prop:submaps} to submaps
  $\omega$ of $(M,\rho)$, that do not include every vertex.  Let
  $(N,\rho)$ be the unimodular random map obtained from $(\omega,\rho)$
  by conditioning on the event $\rho\in\omega$, we have that
  \[
    \bbK(N,\rho) = \P(\rho \in \omega)^{-1} \bbK(M,\rho).
  \]
\end{remark}

\begin{proof}
  First suppose that $M$ and $\omega$ both have a locally finite duals
  a.s.  In this case, every face of $\omega$ is a union of finitely many
  faces of $M$.  For a face $f$ of $M$, denote by $\hat f$ the face of
  $\omega$ containing $f$.

  We define a mass transport in three stages as follows.  Each vertex $v$
  sends a mass of $(\deg(f)-2)/\deg(f)$ to each face $f$ incident to
  it, counting with multiplicity as usual.  Each face $f$ then sends all the
  mass it receives to $\hat f$. Finally, each face $\hat f$ of
  $\omega$ redistributes the total mass it receives uniformly to its
  vertices (again, counted with multiplicity).  This procedure yields a mass transport in which the mass sent from $u$ to $v$ is
  \[
    \sum_{u\perp f \subset \hat f \perp v}
    \frac{\deg(f)-2}{\deg(f)\deg(\hat f)},
  \]
  where the sum is over faces $f\subset \hat f$ containing $u$ and $v$
  respectively. In particular, we have
  \[\sum_{v\in V} \, \left|\sum_{u\perp f \subset \hat f \perp v}
    \frac{\deg(f)-2}{\deg(f)\deg(\hat f)}\right| \leq \sum_{u\perp f}
\left|    \frac{\deg(f)-2}{\deg(f)}\right| \leq \deg(u), \]
so that this mass-transport (which can be negative in the presence of degree $1$ faces) meets the integrability requirements needed to apply the signed Mass-Transport Principle.

  The total mass sent from a vertex $v$ is $\theta_M(v)/2\pi$.  The
  total mass passing through a face $f$ of $M$ is $\deg(f)-2$, while the
  total mass passing through a face $\hat f$ of $\omega$ is $\sum_{f\subset \hat f}
  \deg(f)-2$.  It is easy to see (by induction or using Euler's formula)
  that the latter quantity equals $\deg(\hat f)-2$, and it follows that the total mass
  received by a vertex $v$ is
  \[
    \sum_{\hat f \perp v} \frac{\deg(\hat f)-2}{\deg(\hat f)} =
    \frac1{2\pi}\theta_\omega(v).
  \]
  The mass transport principle gives that
  $\E \theta_M(\rho) = \E \theta_\omega(\rho)$, establishing the claim
  in this case.

  \medskip

  Next, suppose that $M$ has locally finite dual but that $\omega$ does
  not.  Every edge of $M$ not in $\omega$ separates some face of
  $\omega$ in two.  For each infinite face $\hat f$ of $\omega$, let
  $(v_i)_{i\in\Z}$ be the vertices at the corners of $\hat f$ in
  counterclockwise order.  These labels are well defined up to an
  additive constant.  For each $n \geq 1$, define a percolation
  $\omega_n$ by letting $e \in \omega_n$ if and only if either
  $e \in \omega$ or else if $e$ connects vertices $v_i$ and $v_j$ of some
  infinite face $\hat f$ with $|i-j|\geq n$.  Observe that, since $M$
  has a locally finite dual, $\omega_n$ has locally finite dual for
  every $n$.  It follows from the first case above that
  $\bbK(\omega_n,\rho) = \bbK(M,\rho)$.  The sequence of random
  variables $|\theta_{\omega_n}(\rho)|$ are bounded by $2\pi\deg(\rho)$,
  so that
  $\bbK(\omega_n,\rho) \to \bbK(\omega,\rho)$ as $n\to\infty$  by the dominated convergence theorem,
  completing the proof in this case.

  \medskip

  Finally, suppose that $M^\dagger$ is not locally finite.  We construct
  a unimodular map $M'$ with locally finite dual such that $M$ (and hence
  also $\omega$) is a submap of $M'$ with the same vertex set.  By the
  previous case, $\bbK(M,\rho) = \bbK(M',\rho) = \bbK(\omega,\rho)$.  To
  this end, assign to each corner $a=(e,e')$ of every infinite face of
  $M$ an i.i.d.\ random variable $U_a$, uniform in $[0,1]$.  For each
  face $f$ of $M$, add an edge between non-adjacent corners $a$ and $b$,
  if there is some $t$ so that $U_a,U_b\leq t$ and every other corner
  $c$ of $f$ between $a$ and $b$ has $U_c > t$.  Here, adding an edge
  between corners means the edge connects the vertices of the corners,
  and is inserted between the two edges of the corners in the cyclic
  edge order at each endpoint.  It is easy to see that this gives a
  triangulation of the infinite face of $M$, so that $(M')^\dual$ is
  locally finite.

  It is easy to see that these edges can be drawn in the face $f$ with
  no crossings: If $a,b,c,d$ are corners appearing in clockwise order in $f$ and the
  edge $(a,c)$ is added, then $U_b > U_c$ and so the edge $(b,c)$
  cannot be added.  The probability that a corner is connected to
  another corner at distance $k$ away along the face is
  $\binom{k+1}{2}^{-1}$, since their labels must be the two smallest of
  $k+1$ exchangeable labels.  Thus, the expected number of edges added at
  a corner of an infinite face is $\sum 2\binom{k+1}{2}^{-1} = 2$.
  It follows that $\E \deg_{M'}(\rho) \leq \E 3\deg_M(\rho) < \infty$, and
  we conclude by applying the previous case.
\end{proof}

\begin{figure}
   \includegraphics[width=0.8\textwidth]{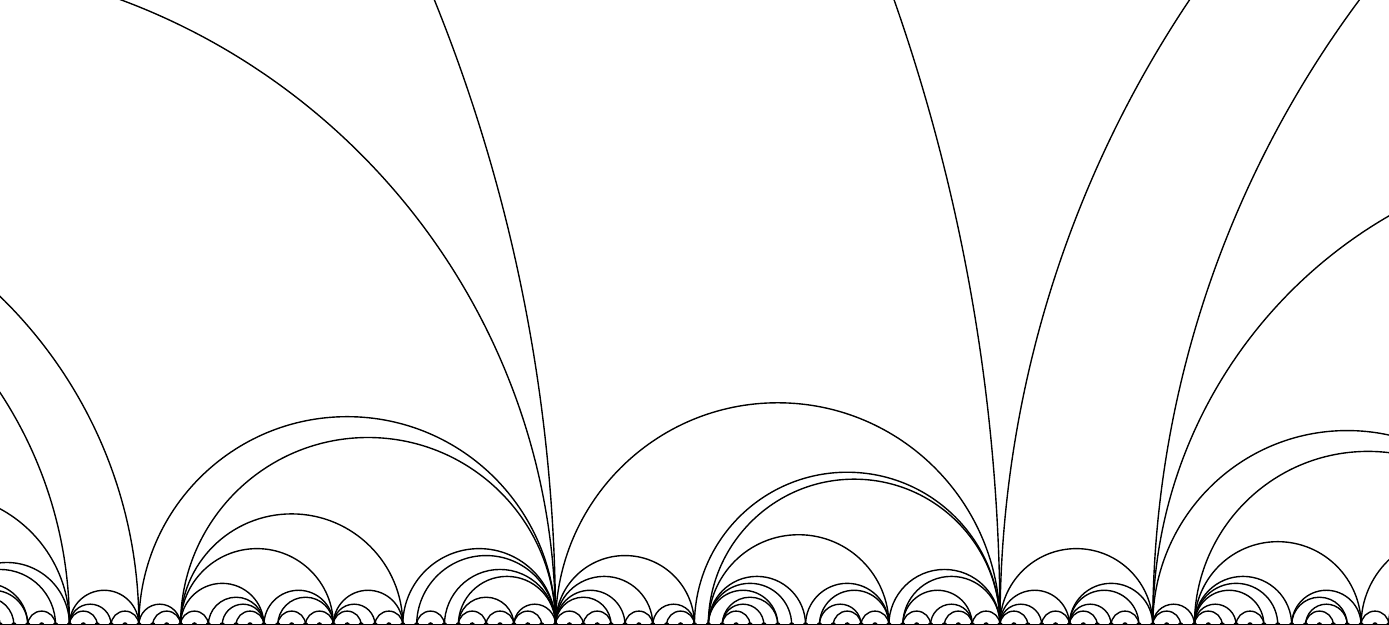}
   \caption{The map $M_1$ is defined by filling in each infinite face of
     $M$ with a system of arcs. This figure demonstrates this procedure
     applied to one of the infinite faces of $\Z$.}
   \label{fig:M1}
 \end{figure}

\begin{prop}
  \label{prop:semicontcurv}
  $\bbK$ is upper-semicontinuous on the set of unimodular maps with
  finite expected degree with respect to the local topology. Explicitly,
  Let $(M_n,\rho_n)$ be a sequence of unimodular random rooted maps with
  $\E[\deg(\rho_n)]<\infty$ that converge weakly to a unimodular random
  rooted map $(M,\rho)$ with $\E[\deg(\rho)]<\infty$.  Then
  \[
    \bbK(M,\rho) \geq \lim_{n\to\infty} \bbK(M_n,\rho_n).
  \]
\end{prop}

\begin{proof}
  First suppose that none of the maps $M_n$ have any faces of
  degree $1$.  In this case, $\theta_{M_n}(v)\geq 0$ for every $v$ and
  $n\geq 1$, and the claim follows from Fatou's lemma.

  Otherwise, we say that a self-loop is \textbf{redundant} if it is contractible in
  $S(M)$, and all the vertices of $M$ are in the same component of the
  complement of the loop.  Note that if $M$ is simply connected then
  every self-loop is contractible, and a redundant self-loop may
  surround other redundant self-loops but no other edge.
  Let $\hat M$ and $\hat M_n$ be obtained from $M$ and $M_n$ respectively by deleting all redundant self-loops.  \cref{prop:submaps} implies that $\bbK(\hat M,\rho) = \bbK(M,\rho)$ and $\bbK(\hat M,\rho) = \bbK(\hat M_n,\rho_n) = \bbK(M_n,\rho_n)$ for all $n\geq 1$.
 Since the operation of removing redundant self-loops
  is continuous in the weak toplogy, $(\hat M_n,\rho_n)$ converges
  weakly to $(M,\rho)$, and the claim follows from Fatou's Lemma as above.
\end{proof}

Combining \cref{prop:semicontcurv} with the equation \eqref{eq:Euler}
has the following immediate corollary.

\begin{corollary}[\cref{thm:dichotomy}, \eqref{iBS2} implies \eqref{iCurv}]
  \label{cor:bscurv}
  Let $(M,\rho)$ be a unimodular random rooted map with
  $\E[\deg(\rho)]<\infty$ that is a Benjamini-Schramm limit of a
  sequence of finite maps $(M_n)_{n\geq1}$ such that
  \[
    \frac{\genus(M_n)}{|V(M_n)|} \xrightarrow[n\to\infty]{} 0.
  \]
  Then $\bbK(M,\rho)\geq 0$.
\end{corollary}

We will prove in \cref{Sec:forests} that $\bbK(M,\rho) \leq 0$ for any
infinite, simply connected, unimodular random rooted planar map with
finite expected degree, so that Benjamini-Schramm limits of low genus
finite maps have $0$ average curvature.

\subsection{Invariance of the curvature} \label{Sec:surfaces}

In this section, we consider \emph{unimodular embeddings} of unimodular
random rooted maps.  We define extend the notion of average curvature to
such an embedding and show that, under certain integrability conditions,
the average curvature associated to the embedding agrees with the
average curvature that we defined combinatorially.  This shows that the
average curvature is a canonical quantity.


We define a \textbf{metric surface embedded map} (MSEM) to be a locally
finite map $M$ together with a proper embedding $z$ of $M$ into an
oriented metric surface $S$.  A rooted MSEM is a MSEM together with a
distinguished root vertex.  Two rooted MSEMs are isomorphic if they are
isomorphic as rooted maps, and there is an orientation preserving
isometry between the two surfaces sending one embedding to the other.

%
We define the local topology on the set of isomorphism classes of rooted
MSEMs by a variation on the local Gromov-Haussdorf topology: Namely, we
set the distance between two rooted MSEMs $(M_1,\rho_1,S_1,z_1)$ and
$(M_2,\rho_2,S_2,z_2)$ to be $e^{-r}$, where $r$ is maximal such that
there is a map isomorphism $\phi$ from the (graph distance) ball of
radius $r$ around $\rho_1$ in $M_1$ to the ball of radius $r$ about
$\rho_2$ in $M_2$ and a correspondence $R$ between the balls of (metric)
radius $r$ around $z_1(\rho_1)$ and $z_2(\rho_2)$ in $S_1$ and $S_2$
respectively, denoted $B_1(r)$ and $B_2(r)$, such that if $(x,x')$ and
$(y,y')\in R$ then $|d_1(x,y) - d_2(x',y')| \leq 1/r$, and such that if
$\phi(e)=e'$ then $R$ restricts to a correspondence between the points of
$e$ and of $e'$ that are contained in the metric balls of radius $r$ around $z_1(\rho_1)$ and $z_2(\rho_2)$ respectively.


%
We define doubly rooted MSEMs, the local topology on the set of
isomorphism classes of doubly rooted MSEMs, and unimodular random rooted
MSEMs similarly to the graph case.  It is straightforward (but rather
tedious) to encode the structure of a MSEM as a marking of the
underlying map of the MSEM, so that all of the usual machinery of
unimodularity transports to this setting.
If $(M,\rho)$ is a unimodular random rooted map and $(M,S,z,\rho)$ is a
unimodular random rooted MSEM with underlying map $(M,\rho)$, we call
$(S,z)$ a \textbf{unimodular embedding} of $M$.


In practice, unimodular embeddings often arise as measurable,
automorphism equivariant functions of the map, such as the embeddings
given by circle packing and the conformal embedding. These are easily
seen to be unimodular since every mass transport on the associated MSEM
is induced by a mass transport on the map.

\medskip



\begin{example}
  Given a map $M$ such that every face of $M$ has degree at least three,
  a natural way to embed it on a surface is as follows.  Associate to
  each face $f$ of degree $d$ a regular $d$-gon with sides of length
  $1$.  If two faces share an edge, we identify the corresponding edges
  of the polygons.  For infinite faces, the corresponding polygon is a
  half-plane $\{z \in \C : {\rm Im}(z) \ge 0\}$ with edges
  $\{[n,n+1]:n \in \Z\}$ along the boundary, which we think of as a
  regular $\infty$-gon.  In this surface, there is no curvature on any
  of the faces, as they are all flat pieces of $\R^2$.  The surface is
  also smooth along the edges, as two faces are glued along straight
  segments of unit length.  Thus, all of the curvature of the surface is
  concentrated on the vertices, and in fact the atom of curvature
  $\kappa_s(v)$ at $v$ (see below) is exactly the value $\kappa(v)$ that
  we defined combinatorially at the beginning of this section.
\end{example}

\begin{example}[glued discs]
  Similarly to the above, we can associate to each face of $M$ of degree
  $d$ a disc of circumference $d$ (a half-plane if $d=\infty$).  Split
  the boundary of each disc to unit length segments, and glue discs
  along these segments.  This is well defined also for maps with faces
  of degree $1$ and $2$.  In this case, the singular curvature at a
  vertex is $\kappa_s(v) = \pi(2-\deg(v)$, but there is also positive
  Gaussian curvature supported on the edges.
\end{example}

The two examples above are jointly generalized by constructions based on
gluing (possibly random) shapes assigned to faces. In what follows, we
restrict ourselves to surfaces that have a smooth Riemannian metric, except
possibly for conical singularities at vertices of the map, and assume that
the edges of the map are embedded as smooth curves in the surface. We call
a unimodular embedding of a map satisfying these conditions a
\textbf{smooth embedding}. For example, any construction based on gluing
polygons yields a smooth embedding.

Recall a corner is a pair $(e,e')$ of oriented edges such that $e'$ is the edge
following $e$ in the counterclockwise order at $e^-$. Given a smooth
embedding of a map $M$, for each corner of $M$, we let $\angl(e,e')$ be the
angle between $e$ and $e'$ at $v$. In a smooth embedding, every oriented
edge $e$ of the map has a well defined total geodesic curvature in the
embedding, which we denote $\kappa_g(e)$.

A smooth Riemannian surface with conical singularities has a Gaussian
curvature $\kappa$ associated with its metric, which is a signed measure on
the surface (see e.g.\ \cite[Chapter~5]{Aleks}). The curvature measure
$\kappa$ is absolutely continuous with respect to the area measure on the
surface except for possible atoms at the conical singularities at the
vertices. We denote by $\kappa_s(v)$ be the atom of curvature at a vertex
$v$, and by $\kappa(f)$ be the total curvature of the face $f$, which is
well-defined if $f$ has finite degree. In a smooth embedding of a map, the
mass of the atom of curvature at a vertex is given by
\[
\kappa_s(v) = 2\pi - \sum_{e:\, e^-=v} \angl(e).
\]
Given a smooth embedding of a map with a locally finite dual, we define the \textbf{total curvature} at a vertex $v$ by
\[
  \kappa(v) = \kappa_s(v) + \sum_{f\perp v} {\kappa(f) \over \deg(f)} .
\]

We are now ready to state our ``invariance of the average curvature''
theorem.

\begin{thm}\label{thm.invcurvature}
  Let $(M,\rho)$ be a unimodular random rooted map with locally finite dual
  and let $(S,z)$ be a smooth unimodular embedding of $M$.  Suppose further
  that either
  \begin{enumerate}[nosep]
  \item $\kappa_g(e)=0$ for every oriented edge $e$ of $M$, and $\kappa(f)
    \leq 0$ for every face $f$ of $M$, or
  \item $\sum_{e : e^- =\rho}|\kappa_g(e)|$ and $|\kappa(\rho)|$ both have
    finite expectation.
  \end{enumerate}
  Then
  \[
  \E [\kappa(\rho)] = \bbK(M,\rho).
  \]
\end{thm}

We remark that it is possible to use \cref{thm.invcurvature} to prove
\cref{prop:submaps} by embedding $M$ in the polygonal manifold
associated to $\omega$.  We have restricted \cref{thm.invcurvature} to
maps with locally finite duals mainly for the sake of brevity.  It is
possible to generalize this identity to maps with infinite faces under
suitable integrability assumptions.

 The necessity of some integrability condition
is demonstrated by the following example.
Consider the map obtained by a unimodular triangulation of each of the
two infinite faces of $\Z$ (e.g.\ as in \cref{fig:M1}).  This map admits
a unimodular embedding in either the Euclidean or hyperbolic plane, by
drawing the edges of $\Z$ as segments of length one along a bi-infinite
geodesic, and drawing the other edges of the map as semi-circular arcs.
This embedding is a measurable function of the map, and is therefore
unimodular.  There are no conical singularities in either case, but
$\sum_{f\perp v} \kappa(f)/\deg(f)$ is $0$ in the Euclidean embedding
and negative at every vertex in the hyperbolic embedding, violating
invariance of the curvature.  In fact, it is easily seen that the map
has average degree $6$, so that $\bbK(M,\rho)\geq 0$.  It is also easy
to see that the hyperbolic embedding does not satisfy the integrability
requirements of \cref{thm.invcurvature}.


\begin{proof}
  The proof relies on a generalization of the mass transport that we used to
  prove \cref{prop:submaps} in the case that both duals were locally finite.  Now, however, we need to allow for the curvature of the faces
  and the geodesic curvature of
  the edges, using the fact that the angle sum around a face is related to these quantities
  by the Gauss-Bonnet Theorem (see e.g.\ \cite[Chapter VI, Section
  7]{Aleks}).

  Each corner $(e,e')$ of $M$ is associated with a vertex $v=e^-$ and face
  $f=e^r$.  The Gauss-Bonnet Theorem implies that for any face $f$ of
  $M$,
  \[
    \sum_{e: e^r= f} [\pi-\angl(e,e')] + \kappa(f) + \sum_{e: e^r=f}
    \kappa_g(e) = 2\pi.
  \]
  Here, the sum over edges is taken for the directed clockwise cycle
  around $f$.  We rewrite this identity as
  \begin{equation}\label{eq:gaussbonnet}
    \sum_{e:e^r= f} \left[ \angl(e,e') - {\kappa(f) \over \deg(f)} -
      \kappa_g(e) \right] = (\deg(f)-2)\pi.
  \end{equation}

  We now define two mass transports on $M$.  First, for each directed edge
  $e=(u,v)$, let $u$ send to $v$ a mass of $\kappa_g(e)$.  The expected
  mass sent from $\rho$ is $\E \sum_{e^-=\rho} \kappa_g(e)$, while the
  expected mass received by $\rho$ is $\E \sum_{e^-=\rho} \kappa_g(-e)= -\E \sum_{e^-=\rho} \kappa_g(e)$.
  Equating the two we find that
  \begin{equation}
    \label{eq:kappa_g0}
    \E \sum_{e^-=\rho} \kappa_g(e) = 0.
  \end{equation}
  The integrability condition needed to apply the signed Mass-Transport Principle here is given by condition (2), while the equality is trivial
  under condition (1).

  Our second mass transport is defined in two stages as follows.  For
  each vertex $u$ and corner $(e,e')$ at $u$ with $e^r=f$, the vertex
  $u$ sends to the face $f$ a mass of
  $\angl(e) - {\kappa(f) \over \deg(f)} - \kappa_g(e)$.  Each face then
  redistributes the mass it recieves uniformly to the vertices at its corners, with the usual multiplicity.  This defines a mass transport $\phi$, given explicitly by
  \[
   \phi(u,v) = \sum_{f : f \perp u, f \perp v} \,\, \sum_{e: e^- =v, e^r =f}\,\, \frac{1}{\deg(f)} \left[ \angl(e,e') - {\kappa(f) \over \deg(f)} -
      \kappa_g(e) \right].
  \]

  Summing the mass sent from $\rho$ through a face $f$ to the corners of
  $f$ gives
  \[
    \sum_v \phi(\rho,v) = \sum_{e^-=\rho} \angl(e,e') - {\kappa(f) \over
      \deg(f)} - \kappa_g(e).
  \]
  If $(1)$ holds then $\phi(u,v)$ is non-negative for every $u$ and $v$, while if $(2)$  holds then the expectation $\E\sum_{v\in V} |\phi(\rho,v)|$ is finite, and we can apply the Mass-Transport Principle to $\phi$ in either case. Taking expectations and using \eqref{eq:kappa_g0}, we have
  \[
    \E \sum_v \phi(\rho,v) = \E\left[ \theta(\rho) - \sum_{e^-=\rho}
      {\kappa(f) \over \deg(f)} \right] = 2\pi - \E\kappa(\rho).
  \]
  Applying \eqref{eq:gaussbonnet}, it follows that the total mass passing through a face $f$
  is $(\deg(f)-2)\pi$, and so the expected total mass received by $\rho$ is
  \[
    \E \sum_v \phi(v,\rho) = \E \sum_{f\perp \rho}
    \frac{(\deg(f)-2)\pi}{\deg(f)} = 2\pi-\bbK(M,\rho),
  \]
  completing the proof. \qedhere
\end{proof}

\begin{example}[Voronoi diagrams and Delaunay triangulations]
  \label{ex:delaunay}
  We say that a set of points in the plane is in \textbf{general
    position} if no four points in the set lie on any given circle or
  line.  Given a set $Z$ of points in general position in either the
  Euclidean plane or the hyperbolic plane, the \textbf{Delaunay
    triangulation} of $Z$ is a simple triangulation with vertex set $Z$, defined so that three points $u,v,$ and $w$ of $Z$ form a triangle if and only if the unique
  disc containing $u,v,$ and $w$ in its boundary does not contain any
  other points of $Z$.

  Suppose that $Z$ is an isometry-invariant, locally finite point
  process in either the Euclidean plane or the hyperbolic plane, and let
  $\hat Z$ be the Palm version of $Z$ that is conditioned to have a
  point at the origin. Then the Delaunay triangulation of $\hat Z$,
  rooted at $0$ is unimodular.  See \cite{BPP} for a study of the
  Poissonian case.  This triangulation, embedded with geodesic segments for
  edges, satisfies condition $(1)$ of \cref{thm.invcurvature}, and we
  deduce that, unsurprisingly, Delaunay triangulations of hyperbolic
  point processes are hyperbolic while Delaunay triangulations of
  Euclidean point processes are parabolic.  (This also follows from the
  methods of \cite{AHNR15}.)  The dual of the Delaunay triangulation is
  the Voronoi diagram of the point process, which can also be made
  unimodular as in \cref{Sec:dual}.  The Voronoi diagram is also
  hyperbolic in the hyperbolic plane and parabolic in the Euclidean
  plane.
\end{example}


\section{Spanning forests}\label{Sec:forests}

It is evident from \cref{fig:logic} that the analysis of \emph{random
  spanning forests} in a map $M$ is central to the proof of
\cref{thm:dichotomy}.  The free and wired uniform spanning forests
$\FUSF$ and $\WUSF$ are known to encode properties of the random walk:
any graph $G$ has the intersection property if and only if its wired uniform spanning forest is
almost surely connected, while $G$ admits non-constant harmonic functions of
finite Dirichlet energy if and only if the laws of the free and wired uniform spanning forest are distinct.  The (free and wired) minimal spanning forests, on the other hand, are related to Bernoulli bond
percolation: $p_c(G) < p_u(G)$ if and only if the laws of the two
minimal spanning forests are distinct.

\subsection{Uniform Spanning Forests}

Our primary method to relate the average curvature of a map $M$ with the
various probabilistic properties listed in \cref{thm:dichotomy} is a
formula relating the average degree of the \emph{free uniform spanning
  forest} of $M$ to its average curvature.  To this aim, we begin with a
succinct review of uniform spanning forest.  We refer the reader to
\cite{BLPS,LP:book} for a comprehensive treatment.

For a finite graph $G$, let $\UST_G$ be the uniform measure on
spanning trees of $G$ (i.e.\ connected subgraphs of $G$ containing every
vertex and no cycles), which is the law of a percolation on $G$.  There
are two natural ways to define infinite volume limits of the uniform
spanning tree. Let $G=(V,E)$ be an infinite, locally finite, connected
graph. An \textbf{exhaustion} of $G$ is an increasing sequence
$\langle V_n \rangle_{n\geq1}$ of finite connected subsets of $V$ such
that $\bigcup_{n\geq1}V_n = V$. Given an exhaustion
$\langle V_n \rangle_{n\geq1}$ of $G$, we define $G_n$ to be the
subgraph of $G$ induced by $V_n$ for each $n\geq1$.
The \textbf{free uniform spanning forest} measure of $G$ is
defined as the weak limit of the uniform spanning tree measures on the graphs $G_n$. That is, for each finite set $S \subset E$,
\[
  \FUSF_G(S \subset T) := \lim_{n \to \infty} \UST_{G_n}(S \subset T).
\]

For each $n\geq1$, we also construct a graph $G^*_n$ from $G$ by
identifying every vertex in $V\setminus V_n$ into a single vertex
$\partial_n$, and deleting all of the resulting self loops from
$\partial_n$ to itself.  We then define the \textbf{wired uniform
  spanning forest} measure of $G$ to be the weak limit of the uniform
spanning tree measures on the graphs $G_n^*$.  That is, for each finite
set $S \subset E$,
\[
  \WUSF_G(S \subset T) := \lim_{n \to \infty} \UST_{G_n^*}(S \subset T).
\]
The study of uniform spanning forests was pioneered by Pemantle
\cite{Pem91}, who showed that that both limits exist for any graph $G$
and in particular are independent of the choice of exhaustion.

A crucial link between the USFs and amenability is the following result
of Aldous and Lyons \cite[Proposition 18.14]{AL07}.

\begin{thm}[\cref{thm:dichotomy}, \eqref{iAmen} implies \eqref{iUSF}]
  If $(G,\rho)$ is an invariantly amenable unimodular random rooted
  graph, then $\FUSF_G=\WUSF_G$ almost surely.
\end{thm}

The free and wired uniform spanning forests enjoy the following
properties:
\begin{description}
\item[Free dominates wired] The measure $\FUSF_G$ stochastically
  dominates the measure $\WUSF_G$ for every graph $G$.
\item[Domination and subgraphs] Let $H$ be a connected subgraph of
  $G$. Then the FUSF of $H$ stochastically dominates the restriction of
  the FUSF of $G$ to $H$.
\item[Expected degree of the WUSF] The expected degree in the WUSF of
  the root of any unimodular random rooted graph is $2$
  \cite[Proposition 7.3]{AL07}.
\end{description}

\medskip

Note that, since a \emph{connected} spanning forest cannot be strictly
contained in another spanning forest, the stochastic domination above
has the following immediate consequence.

\begin{prop}[\cref{thm:dichotomy}, \eqref{iWUSFConn} implies \eqref{iUSF}]
  Let $G$ be a graph.  If the wired uniform spanning forest of $G$ is
  connected almost surely, then the wired and free spanning forests of
  $G$ coincide.
\end{prop}

Let $(G,\rho)$ be a unimodular random rooted graph and $\F$ a sample of
either $\WUSF_G$ or $\FUSF_G$.  If $\F$ is interpreted as a percolation
on $G$, then the marked graph $(G,\rho,\F)$ is also unimodular.  Indeed,
since the definitions of $\FUSF_G$ and $\WUSF_G$ do not depend on the
choice of exhaustion, for each mass transport
$f:\cG_{\bullet\bullet}^{\{0,1\}}\to[0,\infty]$, the expectations
\begin{align*}
  f^F(G,u,v) = \FUSF_G\left[f(G,u,v,\F) \right] \quad\text{ and }\quad
  f^W(G,u,v)= \WUSF_G\left[f(G,u,v,\F) \right]
\end{align*}
are also measurable.  Using this observation, we deduce the
mass-transport principle for $(G,\rho,\F)$ from that of $(G,\rho)$.

\paragraph{Connections to random walk and potential theory.}

Although the uniform spanning tree of each $G_n$ or $G^*_n$ is
connected, the limiting random subgraph can be disconnected.  Indeed,
Pemantle \cite{Pem91} proved that WUSF and FUSF of $\Z^d$ coincide for
all $d\geq1$, and are connected if and only if $d\leq 4$.  A complete
characterisation of the connectivity of the WUSF was given by Benjamini,
Lyons, Peres and Schramm \cite{BLPS}.  A connected, locally finite graph
is said to have the \textbf{intersection property} if the traces of two
independent simple random walks started from any two vertices of the
graph have infinite intersection almost surely (or, equivalently, if the
two traces intersect almost surely).  A graph is said to have the
\textbf{non-intersection property} if the traces of two independent
simple random walks on the graph have finite intersection almost surely.

\begin{thm}[\cite{BLPS}: \cref{thm:dichotomy}, equivalence of
  \eqref{iIntersect} and \eqref{iWUSFConn}]
  Let $G$ be an infinite, locally finite, connected graph and let $\F$
  be a sample of $\WUSF_G$.  Then $\F$ is connected almost surely if and
  only if $G$ has the intersection property.  Moreover, if $G$ has the
  non-intersection property, then $\F$ has infinitely many connected
  components almost surely.
\end{thm}

In general, a graph may have neither of the intersection and
non-intersection properties.  For example, the graph formed by
connecting two disjoint copies of $\Z^3$ by a single edge between their
origins does not have either property.  However, it is easily seen that
this is not the case for reversible random rooted graphs.

\begin{lem}\label{lem:intersection}
  Let $(G,\rho)$ be a unimodular random rooted graph with
  $\E[\deg(\rho)] < \infty$.  Then almost surely $G$ has either the
  intersection property or the non-intersection property.  Consequently,
  the wired uniform spanning forest of a unimodular random rooted graph with finite expected
  degree is almost surely either connected or has infinitely many
  connected components.
\end{lem}

\begin{proof}
  By biasing by the degree we may assume that $(G,\rho)$ is reversible.
  We may assume also that $(G,\rho)$ is ergodic, otherwise we take an
  ergodic decomposition.  Let $(X_n)_{n\geq0}$ and $(X_{-n})_{n\geq0}$
  be independent random walks on $G$ started at $\rho$.  Then the event
  that the traces of $(X_n)_{n\geq0}$ and $(X_{-n})_{n\geq0}$ have
  infinite intersection is an invariant event for the stationary
  sequence $\left( G,(X_{n+k})_{n\in \Z} \right)_{k\in\Z}$ and therefore
  has probability either zero or one by ergodicity (see \cite[\S 4]{AL07} and \cite[Theorem 3.1]{AHNR15}).
\end{proof}


\medskip

Recall that a function $h:V\to\R$ defined on the vertex set of a graph
$G=(V,E)$ is said to be \textbf{harmonic}
if
\[
  h(v)=\frac{1}{\deg(v)}\sum_{u\sim v}h(u)
\]
for every vertex $v$ of $G$, or equivalently if $(h(X_n))_{n\geq0}$ is a
martingale when $(X_n)_{n\geq0}$ is a random walk on $G$.  A graph is
said to be \textbf{Liouville} if it does not admit any non-constant
bounded harmonic functions, and \textbf{non-Liouville} otherwise.
The following well-known proposition follows from the martingale convergence
theorem, see \cite[Exercise 14.28]{LP:book}.

\begin{prop}[\cref{thm:dichotomy}, \eqref{iIntersect} implies \eqref{iLiouville}]
  \label{P:inttoliou}
  Let $G$ be a connected graph.  If $G$ has the intersection property,
  then $G$ is Liouville.
\end{prop}

The converse of \cref{P:inttoliou} does not hold for general graphs.
For example, $\Z^d$ is Liouville for all $d\geq 1$ but has the
intersection property only for $d\leq 4$.  However, Benjamini, Curien
and Georgakopoulos \cite{BCG12} proved that the converse does hold for
planar graphs.

\begin{thm}[\cite{BCG12}; \cref{thm:dichotomy}, \eqref{iLiouville}
  implies \eqref{iIntersect}]
  Let $G$ be a planar graph.  Then $G$ is Liouville if and only if it
  has the intersection property.
\end{thm}

The \textbf{Dirichlet energy} of a function $f: V \to \R$ defined on the
vertex set of $G$ is defined to be
\[
  \cE(f) = \frac{1}{2}\sum_{e\in
    E^\rightarrow}\left(f(e^-)-f(e^+)\right)^2.
\]
The following is classical; see \cite[Exercise 9.43]{LP:book}.

\begin{prop}[\cref{thm:dichotomy}, \eqref{iLiouville} implies \eqref{iDirichelet}]
  Let $G$ be a connected graph.  Then the bounded harmonic functions of
  finite Dirichlet energy are dense in the space of harmonic functions
  of finite Dirichlet energy with norm $\|f\|^2 =\cE(f)$.  In
  particular, if $G$ admits a non-constant harmonic function of finite
  Dirichlet energy, then $G$ admits a bounded non-constant harmonic
  function of finite Dirichlet energy.
\end{prop}

Benjamini, Lyons, Peres and Schramm \cite{BLPS} also relate analytic
properties of $G$ to the WUSF and FUSF of $G$.

\begin{thm}[\cite{BLPS}; \cref{thm:dichotomy}, equivalence of items
  \eqref{iUSF} \eqref{iDirichelet}]
  Let $G$ be an infinite connected graph.  Then the measures $\FUSF_G$
  and $\WUSF_G$ are distinct if and only if $G$ admits non-constant harmonic
  functions of finite Dirichlet energy.
\end{thm}

\medskip

In general, the FUSF is far less understood as the WUSF: no criterion
for its connectivity is known, nor is it known whether the number of
components of the FUSF is non-random in every graph.  However, for
simply connected planar maps, the FUSF is relatively well understood
thanks to the following duality: Given a map $M$ and a set $W\subset E$,
let $W^\dagger:=\{e^\dagger \in E^\dagger : e \notin W\}$ be the set of
dual edges whose corresponding primal edges are not contained in
$W$. Observe that if $t$ is a spanning tree of a finite planar map $M$,
then the dual $t^\dagger$ is a spanning tree of $M^\dagger$ --- it is
connected because $t$ has no cycles, and has no cycles because $t$ is
connected.  This observation leads to the following.

\begin{prop}[USF Duality {\cite[Theorem~12.2]{BLPS}}]
  \label{prop:USFduality}
  Let $M$ be a simply connected map with locally finite dual $M^\dagger$
  and let $\F$ be a random variable with law $\FUSF_M$.  Then
  $\F^\dagger$ has law $\WUSF_{M^\dagger}$.
\end{prop}

In general, if $M$ is an infinite simply connected map and $\F$ is an
essential spanning forest of $M$ (that is, a spanning forest such that
every component is infinite), then $\F^\dagger$ is an essential spanning
forest of $M^\dagger$ --- it is a forest because every component of $\F$
is infinite, and is essential because $\F$ has no cycles.  Moreover, the
forest $\F$ is connected if and only if every component of $\F^\dagger$
is one-ended.  Furthermore, a combination of results from Aldous and
Lyons \cite{AL07} yields the following.

\begin{prop}[\cref{thm:dichotomy}, \eqref{iWUSFConn} implies \eqref{iAmen}]
  Let $(G,\rho)$ be a unimodular random rooted graph with
  $\E[\deg(\rho)]<\infty$ and let $\F$ be a sample of $\WUSF_G$.  If
  $\F$ is connected almost surely, then $(G,\rho)$ is invariantly amenable.
\end{prop}

\begin{proof}
  By Proposition~7.3 of \cite{AL07}, the average degree of $\F$ is $2$.
  By Theorem~6.2 there, this implies that $\F$ has at most two ends.
  Since $\F$ is connected, Theorem~8.9 of \cite{AL07} implies that
  $(G,\rho)$ is invariantly amenable.
\end{proof}

\subsection{Connectivity and degrees in the FUSF}

The first main result of this section relates the average curvature and
the expected degree of the FUSF in a simply connected unimodular random
rooted map.

\begin{thm}\label{thm:forestformula}
  Let $(M,\rho)$ be an infinite, simply connected unimodular random map.
  Suppose that $\E[\deg(\rho)]<\infty$, and let $\F$ be a sample of
  $\FUSF_M$. Then
  \begin{equation}\label{eq:FUSFformula}
    \E[\deg_{\F}(\rho)]=\frac{1}{\pi}\E[\theta(\rho)] =
    2-\frac{1}{\pi}\bbK(M,\rho).
  \end{equation}
\end{thm}

As an easy consequence we get the following component of \cref{thm:dichotomy}.

\begin{corollary}[\cref{thm:dichotomy}, equivalence of \eqref{iCurv} and
  \eqref{iUSF}  and $\bbK\leq 0$]\label{C:curv_USF}
  Let $(M,\rho)$ be an infinite, simply connected unimodular random map with
  $\E[\deg(\rho)]<\infty$.  Then the average curvature satisfies
  $\bbK(M,\rho)\leq 0$, and equals zero if and only if
  $\FUSF_M = \WUSF_M$ almost surely.
\end{corollary}

\begin{proof}
  Let $\F$ be a sample of $\FUSF_M$.  Since the expected degree of the
  WUSF in any unimodular random rooted graph is two, and $\FUSF_M$
  stochastically dominates $\WUSF_M$, we have that
  $\E[\deg_\F(\rho)]\geq2$ and that the measures $\FUSF_M$ and $\WUSF_M$
  differ almost surely if and only if this inequality is strict. We
  conclude by applying \cref{thm:forestformula}.
\end{proof}

\cref{thm:forestformula} follows as an immediate corollary of
\cref{prop:submaps} and the following theorem, which is the second main
result of this section.  In \cref{sec:expecteddegree}, we give an
alternative, duality-based proof of \cref{thm:forestformula} that does
not rely on \cref{thm:FUSFconnectivity} or \cref{prop:submaps}, and also
applies to the free minimal spanning forest.

\begin{thm}[Connectivity of the FUSF]\label{thm:FUSFconnectivity}
  Let $(M,\rho)$ be an infinite, simply connected unimodular random rooted map with
  $\E[\deg(\rho)] < \infty$.  Then the free uniform spanning forest of
  $M$ is connected almost surely.
\end{thm}

\cref{thm:FUSFconnectivity} is complemented by work by the
second and third authors \cite{HutNach15b}, who prove the corresponding
theorem for all simply connected planar maps with \emph{bounded
  degrees}.
Since the measure $\FUSF_G$ stochastically dominates $\WUSF_G$ for every
graph $G$, and has no cycles, we deduce the following immediate corollary.

\begin{corollary}[\cref{thm:dichotomy}, equivalence of \eqref{iUSF} and
  \eqref{iWUSFConn}]
  Let $(M,\rho)$ be an infinite, simply connected unimodular random map with
  $\E[\deg(\rho)]<\infty$.  Then $\FUSF_M=\WUSF_M$ if and only if
  the wired uniform spanning forest of $M$ is connected almost surely.
\end{corollary}

Having finite expected degree is necessary for this to hold.  Let $T_n$
be a binary tree of height $n$ drawn in the plane.  The
Benjamini-Schramm limit of $T_n$ as $n$ tends to infinity is known as
the canopy tree, and can be thought of as an `infinite binary tree
viewed from a leaf'.  Let $M_n$ be the finite map obtained by drawing
two copies of $T_n$ so that one is the reflection of the other, and
attaching these two copies together at their leaves (see
\cref{fig:trees}).  Replace each edge of $M_n$ at distance $k$ from the
leaves by $3^k$ parallel edges, and call the resulting map $M'_n$.  The
Benjamini-Schramm limit $(M',\rho) = \lim M'_n$ exists and is formed of
two similarly thickened canopy trees, attached together at their leaves.

\begin{figure}
  \centering
  \includegraphics[height=4cm]{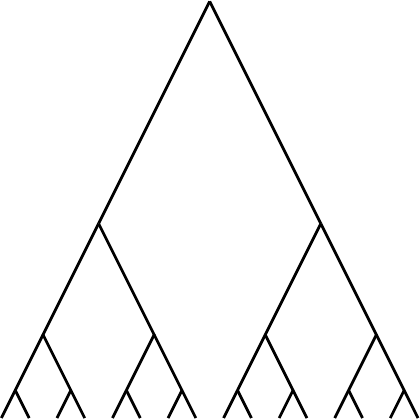} \qquad
  \includegraphics[height=4cm]{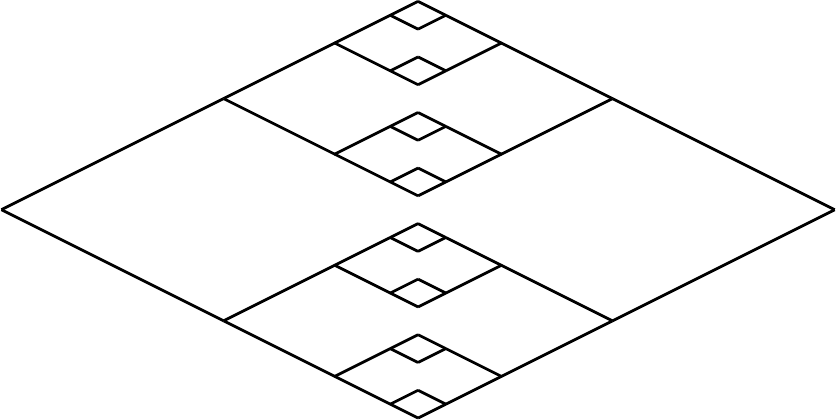}
  \caption{The maps $T_4$ (left) and $M_4$ (right).}
  \label{fig:trees}
\end{figure}

The random walk on $M'$ has constant drift away from the leaves, so it
is possible for two random walks to be absorbed in different halves of
$M'$.  Thus $M'$ does not have the intersection property (nor the
non-intersection property), and $\WUSF_{M'}$ is not connected.  However,
the space of bounded harmonic functions is spanned by the probability of
being absorbed at either side of $M'$.  It follows that bounded (and
hence all) harmonic functions have infinite Dirichlet energy, and that
$\FUSF_{M'} = \WUSF_{M'}$.

\medskip

Towards a proof of \cref{thm:FUSFconnectivity}, observe that if $M$ has
locally finite dual almost surely, then, by \cref{prop:USFduality}, it
suffices to showt that every component of the WUSF of $M^\dagger$ is
one-ended almost surely.  Fortunately, this is known to hold in several
large classes of graphs: The following was proven by the second author
\cite{H15,hutchcroft2015interlacements}.  See also earlier works by Pemantle
\cite{Pem91}, Benjamini, Lyons, Peres, and Schramm \cite{BLPS}, and
Aldous and Lyons \cite{AL07}.  See \cite{LMS08} for further
one-endedness results in the deterministic setting.

\begin{thm}[\cite{H15,hutchcroft2015interlacements}]
  \label{T:Hwired}
  Let $(G,\rho)$ be a transient unimodular random rooted graph. Then
  every component of the wired uniform spanning forest of $G$ is one-ended almost surely.
\end{thm}

Two-ended unimodular graphs necessarily have connected and two-ended
$\WUSF_G$.  However, it is an open question whether WUSF components in a
one-ended, recurrent, unimodular graphs can have two ends (more than two
is impossible).  The final result of this section settles the planar
case.

\begin{thm}\label{T:WUSFends}
  Let $(M,\rho)$ be a recurrent unimodular random rooted planar map with
  $\E[\deg(\rho)]<\infty$.  Then the wired uniform spanning forest of $M$ has the same number of
  ends as $M$ almost surely (which is either one or two since $M$ is
  recurrent).
\end{thm}




\subsection{Proof of \cref{thm:FUSFconnectivity,T:WUSFends}}

The proofs are closely linked and are split according to whether the
dual map $M^\dual$ is locally finite or not.

\begin{proof}[Proofs of when $M^\dual$ is locally finite]
  Let $(M,\rho)$ be a unimodular random rooted map with
  $\P[\deg(\rho)]<\infty$, and suppose that the dual $M^\dagger$ is
  locally finite almost surely.  We may assume that $(M,\rho)$ is
  ergodic, otherwise we take an ergodic decomposition.  Let $\F$ be a
  sample of $\FUSF_M$ and let $\F^\dagger$ be the dual forest.  Since
  $\E[\deg(\rho)]<\infty$, the law of $(M^\dagger,\rho^\dagger)$ is
  absolutely continuous w.r.t.\ a reversible random rooted map by
  \cref{prop:dualmaps}.

  If $M^\dagger$ is almost surely transient, \cref{T:Hwired} implies
  that every component of $\F^\dual$ is one-ended almost surely, and we
  deduce that $\F$ is connected almost surely.

  If $M^\dagger$ is almost surely recurrent, then $\F^\dual$ is
  connected almost surely.  Being unimodular and recurrent, this tree
  has at most two ends.  Since $M^\dual$ is locally finite, $M$ and
  $M^\dual$ have the same number of ends. If $\F^\dual$ was two-ended,
  while $M$ is one-ended, then $\F$ would have exactly two components,
  contradicting \cref{lem:intersection}.  If \asaf{$M$} is two-ended,
  then so is the spanning tree \asaf{$\F$}, giving \cref{T:WUSFends}.  In
  either case we deduce that $\F$ is connected almost surely.
\end{proof}


The remainder of this section is dedicated to the proof of
\cref{thm:FUSFconnectivity,T:WUSFends} in the presence of infinite
faces.  We begin by developing a variant of Wilson's algorithm that
allows us to sample the dual of the FUSF using random walks when the
dual is not locally finite.

Given a graph $G$ and a path $\gamma$ in $G$ that is either finite or
transient, i.e.\ visits each vertex of $G$ at most finitely many times,
the \textbf{loop-erasure} $\textsf{LE}(\gamma)$ is formed by erasing
cycles from $\gamma$ chronologically as they are created.  Formally,
$\textsf{LE}(\gamma)_i = \gamma_{t_i}$ where the times $t_i$ are defined
recursively by $t_0 = 0$ and
$t_i = 1 + \max \{ t \geq t_{i-1} : \gamma_t = \gamma_{t_{i-1}}\}$.  If
the graph has multiple edges, then $\textsf{LE}(\gamma)$ uses the edge
$\gamma$ uses between times $t_i-1$ and $t_i$.  The loop-erasure of a
simple random walk is known as \textbf{loop-erased random walk}, was
introduced by Lawler \cite{Lawler80}, and was extensively studied.

\textbf{Wilson's algorithm} \cite{Wilson96} is a method of sampling a
uniform spanning tree of a finite graph by joining together loop-erased
random walk paths.  Benjamini, Lyons, Peres and Schramm \cite{BLPS}
introduced a variant of Wilson's algorithm for sampling the WUSF of an
infinite transient graph, known as \textbf{Wilson's algorithm rooted at
  infinity}.  Let $G=(V,E)$ be a connected, locally finite graph and let
$(v_i)_{i \geq 1}$ be an enumeration of the vertex set $V$.  We sample a
sequence of forests $(\F_i)_{i \geq 0}$ in $G$ recursively as follows:
\begin{enumerate}
\item If $G$ is finite or recurrent, fix a vertex $v_0$ of $G$ and let
  $\F_0=\{v_0\}$.  If $G$ is transient set $\F_0=\emptyset$.
\item Given $\F_i$, start a simple random walk from $v_i$ in $M$,
  independent of everything already sampled, and stopped if and when it
  first hits a vertex already included in $\F_i$.
\item Take the loop-erasure of this random walk path, and let $\F_{i+1}$
  be the union of $\F_i$ and this loop-erased path.
\item Let $\F=\bigcup_{i\geq0}\F_i$.
\end{enumerate}
The resulting random forest has law $\UST_G$ when $G$ is finite
\cite{Wilson96} and $\WUSF_G$ when $G$ is infinite \cite{BLPS}.

If $M^\dual$ is locally finite, then we can sample $\FUSF_M$ by using
Wilson's algorithm to sample $\F^\dual$ with law $\WUSF_{M^\dual}$.  We
introduce a variation for sampling the dual of $\FUSF_M$, when $M^\dual$
is not locally finite.  Let $\Ffin$ and $F_\infty$ be respectively the
set of finite and infinite faces of $M$.  The following is called
\textbf{Wilson's algorithm rooted at $\{\infty\}\cup F_\infty$}.


Let $(f_i)_{i\geq 1}$ be an enumeration of the finite faces of $M$.  We
sample an increasing sequence of forests $(\F^\dual_i)_{i\geq 0}$ in
  $M^\dual$ recursively as follows:
\begin{enumerate}[label={(\arabic*$'$)}]
\item Let $\F^\dagger_0=F_\infty$.
\item Given $\F^\dagger_i$, start a simple random walk from $f_i$ in
  $M^\dual$, independent of everything we have already sampled, and
  stopped if and when it first hits a face already included in
  $\F^\dagger_i$.  In particular, the walk is stopped if it hits a face
  of infinite degree, so that it is well-defined.
\item Take the loop-erasure of this random walk path, and let
  $\F^\dagger_{i+1}$ be the union of $\F^\dagger_i$ and this loop-erased
  path.
\item Let $\F^\dagger=\bigcup_{i\geq0}\F^\dagger_i$.
\end{enumerate}

\begin{prop}\label{P:Fdualinfty}
  Let $M=(V,E,\sigma)$ be a simply connected map with dual $M^\dagger$
  and let $\F^\dagger$ be a random subset of $E^\dagger$ sampled by
  Wilson's algorithm rooted at $\{\infty\}\cup F_\infty$. Then
  $\F=(\F^\dagger)^\dagger$ is a sample of $\FUSF_M$.
\end{prop}

\begin{proof}
  Let $(V_n)_{n\geq0}$ be an exhaustion of $V$ such that the submap of
  $M$ induced by $V\setminus V_n$, denoted $M_n$, does not have any
  finite connected components for any $n$.  The dual of $M_n$ may be
  constructed from $M^\dagger$ by identifying every face $f$ of $M$ that
  does not have all of its vertices included in $V_n$ into a
  single vertex $\partial_n$, and deleting all the self-loops that are
  created.  In particular, all infinite faces of $M$ are identified into
  $\partial_n$ for every $n\geq1$.

  Note that a simple random walk $(X_n)_{n\geq0}$ on $M^\dagger$,
  started at some finite face $f$ and stopped if and when it hits
  $F_\infty$, is either transient or killed at a finite time almost
  surely.  Given these observations, the rest of the proof, below, proceeds
  similarly to the usual proof of the veracity of Wilson's algorithm
  rooted at infinity~\cite[Theorem 5.1]{BLPS}.

  Let $H$ be a finite set of edges of $M^\dagger$, and let
  $f_1,\dots,f_l$ be an enumeration of the set of faces $f$ of $M$ that
  are endpoints of at least one of the edges in $H$. Let
  $( ( X^i_j )_{j\geq0}: i = 1,\dots,l)$ be a
  collection of independent random walks in $M^\dagger$, where the walk
  $( X^i_j )_{j\geq0}$ is started at $f_i$ and stopped the
  first time that it hits an infinite face of $M$. Run Wilson's
  algorithm in $M_n^\dagger$, rooted at $\partial_n$ and starting with
  the faces $f_1,\ldots,f_l$ in that order, using the random walks
  $( X^i_j )_{j\geq0}$: For each $i\in[l]$, let $\tau^n_i$
  be the first time that the random walk $( X^i_j)_{j\geq0}$
  visits the portion of the spanning tree generated up to time $i-1$, so
  that
  \[
    \UST_{M^\dagger_n}(H \subset T) = \P\Big(H \subseteq \bigcup_{i\leq l}
    \mathsf{LE}\big((X^i_j)_{j=0}^{\tau^n_i}\big) \Big).
  \]

  Now, similarly, run Wilson's algorithm on $M^\dagger$ rooted at
  $\{\infty\}\cup F_\infty$, starting with the faces $f_1,\ldots,f_l$ in
  that order and using the random walks
  $( X^i_j )_{j\geq0}$, and let $\tau_i$ be the first time
  that the random walk $( X^i_j)_{j\geq0}$ visits the
  portion of the spanning tree generated up to time $i-1$ (which might
  now be infinite).  Since the walks $( X^i_j )_{j\geq0}$
  are finite or transient almost surely, we have that
  \[
    \tau^n_i \to \tau_i \quad \text{ and } \quad
    \mathsf{LE}\big((X^i_j)_{j=0}^{\tau^n_i}\big) \to
    \mathsf{LE}\big((X^i_j)_{j=0}^{\tau_i}\big)
  \]
  almost surely as $n\to\infty$.  It follows that
  \begin{align*}
    \FUSF_{M}(H \subset \F^\dagger) %
    &= \lim_{n\to\infty}\UST_{M^\dagger}(H \subset T) =
    \lim_{n\to\infty}\P\Big(H \subseteq \bigcup_{i\leq l}
    \mathsf{LE}\big((X^i_j)_{j=0}^{\tau^n_i}\big) \Big) \\
    &= \P\Big(H \subseteq \bigcup_{i\leq l}
    \mathsf{LE}\big((X^i_j)_{j=0}^{\tau_i}\big) \Big)
  \end{align*}
  completing the proof.
\end{proof}

\begin{prop}\label{prop:conn2}
  Let $(M,\rho)$ be a simply connected, unimodular random rooted map
  with $\P[\deg(\rho)]<\infty$ and suppose that the dual $M^\dagger$
  contains a vertex of infinite degree almost surely. Let $\F$ be a
  sample of $\FUSF_M$. Then every connected component of
  $\F^\dagger \setminus F_\infty$ is finite almost surely, and
  consequently $\F$ is connected almost surely.
\end{prop}


\begin{proof}[Proof of \cref{prop:conn2}]
  The main ingredient is the observation that a random walk on $M^\dual$
  started at a finite face will hit $F_\infty$ in a finite time almost
  surely.

  To see this, we first bias the law of $M$ by $\deg(\rho)$ to get a
  reversible map.  Next, we convert \asaf{$M^\dagger$} into a
  a (possibly disconnected) graph $G'$ with edge set $E^\dagger$ as
  follows. For each edge $e^\dagger \in E^\dagger$ with an endpoint in
  $F_\infty$, replace this endpoint with a new vertex of degree one.
  The vertex set of $G'$ is $\Ffin \cup L$, where $L$ is the set of new
  degree one vertices of $G'$ corresponding to edges of $M$ incident to
  $F_\infty$.
   We define a marking $m$ on the vertices of $G'$ where $m(v) = \mathbbm{1}(v\in L)$.
  Let $\eta$ be chosen uniformly from among the oriented edges emanating from $\rho$, and let $\rho'$ be the tail of $\eta^\dagger$ in $G'$.
  Let $G''$ denote the connected component of $G'$ containing $\rho'$.  
  A similar argument to that used in \cref{Sec:dual} shows that the random marked rooted graph $(G'',\rho',m)$ is reversible.
  Thus, since $G''$ contains a vertex in $L$ a.s., the random walk on $G''$ must visit $L$ a.s.
  We conclude
  by noting that the random walk on $M^\dagger$ started at a vertex
  $f\in \Ffin$ and stopped when it first hits $F_\infty$ can be coupled
  with the random walk on $G'$ started at the same $f$ and stopped when
  it first hits $L$ so that the two hitting times agree.

  To study the FUSF, bias the reversible law of $(G',\rho')$ by
  $1/\deg(\rho')$ to get a unimodular graph.
  Wilson's algorithm rooted at $\{\infty\}\cup F_\infty$ can be seen as
  generating a spanning forest $\F^\dual$ of $G'$, starting with
  $\F^\dual_0=L$ and adding loop erased random walks.
  \cref{P:Fdualinfty} implies that $\F^\dual$ is distributed as the dual of the free uniform spanning forest of $M$.
  The argument above implies that every component of $\F^\dual$
  contains a unique vertex of $L$.
  Consider the mass transport where each vertex of $G'$ sends a unit
  mass to the vertex of $L$ in its component of $\F^\dual$.  Then every
  vertex sends mass $1$, and hence no vertex receives infinite mass.
  Thus, every component of $\F^\dual$ is finite, and the result follows.
\end{proof}


\begin{proof}[Proofs when $M^\dual$ is not locally finite]
  If $M$ is simply connected then \cref{thm:FUSFconnectivity,T:WUSFends}
  follow immediately from \asaf{\cref{prop:conn2}}.  If $M$ is recurrent and
  not simply connected, then $M$ and $\F$ must both have two ends, giving
  \cref{T:WUSFends}. (\cref{thm:FUSFconnectivity} makes no claim about multiply connected maps.)
\end{proof}

\subsection{Percolation and minimal spanning forests}

While the uniform spanning forests are related to random walks, the
\emph{minimal} spanning forests are related to bernoulli bond
percolation.  We recall some of the connections, and refer the reader to
\cite[\S11]{LP:book} for a detailed account.
\textbf{Bernoulli-$p$ bond percolation} on $G$, denoted $\omega_p$, is
the random subgraph of $G$ defined by keeping each edge of $G$
independently with probability $p$ and deleting the rest.  The Bernoulli
bond percolations $\{\omega_p\}_{p \in [0,1]}$ on a graph $G$ may be
coupled monotonically by letting $\{U(e)\}_{e \in E}$ be a collection of
i.i.d.\ Uniform$([0,1])$ random variables indexed by the edge set of $G$
and setting $\omega_p(e)=\mathbbm{1}(U(e)\leq p)$ for every $e\in E$ and
$p\in[0,1]$.  The \textbf{critical probability} of $G$ is defined by
\[
  p_c(G) := \inf\left\{ p : \P(\omega_p \text{ has an infinite connected
      component}) = 1 \right\}.
\]
It is well-known \cite{AL07,NewSch81} that if $(G,\rho)$ is an ergodic
unimodular random rooted graph, then for each $p\in[0,1]$ the number of
infinite connected components of $\omega_p$ for any is non-random and is
in $\{0,1,\infty\}$ almost surely.  Moreover, if $(G,\rho)$ is a
unimodular random rooted graph and $p$ is such that $\omega_p$ has a
unique infinite cluster almost surely, then $\omega_{p'}$ has a unique
infinite connected component almost surely for every $p'\geq p$
\cite{HaggPe99,LS99,AL07}.
In light of this, the \textbf{uniqueness threshold} of a graph $G$ is defined to be
\[
p_u(G) = \inf\{ p : \P(\omega_p \text{ has a unique infinite
  connected component}) = 1 \}.
\]
Note that if $(G,\rho)$ is an ergodic, infinite unimodular random rooted
graph, then the quantities $p_c(G)$ and $p_u(G)$ are non-random.  It is
of interest to determine which graphs have a non-uniqueness phase for
Bernoulli bond percolation.
The following was proven by Aldous and Lyons \cite{AL07}, following work by in the transitive setting by Aizenman, Kesten, and Newman \cite{AizKesNew87}, Burton and Keane \cite{burton1989density}, and Gandolfi, Keane, and Newman \cite{gandolfi1992uniqueness}.
\begin{thm}[\cref{thm:dichotomy}, \eqref{iAmen} implies \eqref{iPerc}]
\label{thm:amenableuniquecluster}
Let $(G,\rho)$ be an invariantly amenable unimodular random rooted graph. Then $\omega_p$ has at most one infinite cluster almost surely for every $p\in [0,1]$, and in particular $p_c=p_u$ almost surely.
\end{thm}

\begin{proof} This follows by combining Corollary 6.11, Lemma 6.12,  Proposition 8.8, and Corollary 8.10 of \cite{AL07}.
\end{proof}
A long standing conjecture of Benjamini and
Schramm~\cite{bperc96} asserts conversely that every nonamenable
transitive graph has $p_c<p_u$.  \medskip

Lyons, Peres and Schramm \cite{LPS06} related the non-uniqueness phase
to minimal spanning forests.  Given a finite graph $G$ and an injective
\textbf{weight} function $U : E(G) \to \R$ the \textbf{minimal spanning
  tree} of $G$ with respect to $U$ is defined to be the spanning tree
$T$ of $G$ minimising the total weight $\sum_{e\in T}U(e)$.
Equivalently, an edge $e$ of $G$ is contained in $T$ if and only if
there does not exist a simple cycle in $G$ containing $e$ such that $e$
maximises $U(e)$ among the edges in this cycle.  We write
$\mathsf{MST}_G$ for the distribution on spanning trees of $G$ obtained
by letting $T$ be the minimal spanning tree of $G$ with respect to
weights $\{U(e)\}_{e \in E}$ given by i.i.d.\ Uniform$([0,1])$ random
variables.

This extends to infinite graphs using exhaustions, as for uniform
spanning trees.  Given an exhaustion $(V_n)_{n\geq0}$ of an infinite
graph $G$, we define the \textbf{free and wired minimal spanning
  forests} as the weak limits
\[
  \FMSF_G(S \subset T) := \lim_{n \to \infty} \mathsf{MST}_{G_n}(S \subset T)
\]
and
\[
\WMSF_G(S \subset T) := \lim_{n \to \infty} \mathsf{MST}_{G_n^*}(S \subset T).
\]
The limits exist and do not depend on the choice of exhaustion.  If
$(G,\rho)$ is a unimodular random rooted graph and $\F$ is a sample of
either $\WMSF_G$ or $\FMSF_G$, then $\F$ is a percolation on $(G,\rho)$.
Unlike in the uniform case, both of the minimal spanning forests may
also be defined directly on the infinite graph $G$ as follows.  Let
$\{U(e) : e \in E\}$ be a collection of i.i.d.\ Uniform$([0,1])$ random
variables indexed by the edge set of $G$.  An edge $e$ of $G$ is
included in free minimal spanning forest of $G$ if and only if it is not
the heaviest edge in any simple cycle in $G$.  An edge $e$ of $G$ is
included in the wired minimal spanning forest of $G$ if and only if it
is not the heaviest edge in any simple cycle in $G$ or in any
bi-infinite simple path (or 'cycle through infinity') in $G$.

\begin{thm}[\cref{thm:dichotomy},
  equivalence of \eqref{iMSF} and \eqref{iPerc}]
  Let $(G,\rho)$ be an infinite unimodular random rooted graph with
  $\E[\deg(\rho)]<\infty$.  Then $p_c(G)<p_u(G)$ if and only if
  $\FMSF_G \neq \WMSF_G$ if and only if there is at most one infinite
  cluster in Bernoulli $p$-percolation on $G$, at every $p\in[0,1]$.
\end{thm}

\begin{proof}
  Lyons, Peres and Schramm \cite{LPS06} proved that an infinite
  connected graph $G$ has $\FMSF_G=\WMSF_G$ if and only if for
  $\omega_p$ has a unique infinite cluster for Lebesgue-a.e.\
  $p\in[0,1]$.  Combining this with monotonicity of uniqueness
  \cite[Theorem 6.7]{HaggPe99,AL07} implies that if $\FMSF_G=\WMSF_G$
  then $p_c(G)=p_u(G)$ and hence there is at most one infinite cluster
  at every $p\in[0,1]$ except possibly at $p=p_c$.
  If $G$ is invariantly amenable, there is at most one infinite cluster
  at every $p$ by \cref{thm:amenableuniquecluster} and so also at $p=p_c$. If $G$ is nonamenable, then by
  \cite{AL07} there is no infinite cluster at $p=p_c$.
\end{proof}

The minimal spanning forests share several properties with their uniform
cousins:

\begin{description}
\item[Free dominates wired] The measure $\FMSF_G$ stochastically
  dominates the measure $\WMSF_G$ for every graph $G$.
\item[Domination and subgraphs] let $H$ be a connected subgraph of $G$.
  Then the FMSF of $H$ stochastically dominates the restriction of the
  FMSF of $G$ to $H$.
\item[Expected degree of the WMSF] The expected degree in the WMSF of
  root of any unimodular random rooted graph is two
  \cite[Proposition~7.3]{AL07}.
\item[Amenability and boundary conditions] If $(G,\rho)$ is an
  invariantly amenable random rooted graph, then $\FMSF_G=\WMSF_G$
  almost surely \cite[Proposition 18.14]{AL07}.
\item[Planar duality] If $M$ is a simply connected map with locally
  finite dual $M^\dagger$ and $\F$ is a sample of $\FMSF_M$, then
  $\F^\dagger$ has law $\WMSF_{M^\dagger}$ \cite[\S11.5]{LP:book}.
\end{description}

\medskip

From the above, we deduce that if $(G,\rho)$ a unimodular random rooted
graph, the measures $\FMSF_G$ and $\WMSF_G$ coincide almost surely if
and only if the expected degree of $\rho$ in the FMSF of $G$ is two.  As
for the uniform spanning forests, we relate the expected degree of the
FMSF to the average curvature.  The proof of \cref{T:FMSFFormula} is
general enough to provide an alternative proof of
\cref{thm:forestformula} that does not rely on connectivity of the
forest.

\begin{thm}[\cref{thm:dichotomy}, equivalence of \eqref{iCurv} and
  \eqref{iMSF}]
  \label{T:FMSFFormula}
  Let $(M,\rho)$ be an infinite simply connected unimodular random
  rooted map with $\E[\deg(\rho)]<\infty$ and let $\F$ be a sample of
  $\FMSF_M$.  Then
  \[\E[\deg_\F(\rho)] = 2 - \frac{1}{\pi}\bbK(M,\rho).\]
  In particular, $\FMSF_M=\WMSF_M$ almost surely if and only if
  $\bbK(M,\rho)=0$.
\end{thm}

The equivalence of \eqref{iPerc} and \eqref{iAmen} in
\cref{thm:dichotomy} can also be proven directly as follows. Let
$(M,\rho)$ be a simply connected, invariantly nonamenable unimodular random rooted
map. If $(M,\rho)$ has locally finite dual, we deduce that
$p_c(M)<p_u(M)$ by applying the following two results.

\begin{thm}[Benjamini, Lyons, Peres and Schramm \cite{BLPS99,BLPS99b};
  Aldous and Lyons \cite{AL07}] \label{Thm:nonamenableperc} Let
  $(G,\rho)$ be an invariantly nonamenable unimodular random rooted
  graph with $\E[\deg(\rho)]<\infty$. Then $\omega_{p_c}$ does not
  contain any infinite connected components almost surely.
\end{thm}

\begin{thm}[Benjamini and Schramm {\cite[Theorem 3.1]{BS00}}]
  Let $(M,\rho)$ be an invariantly nonamenable, simply connected
  unimodular random rooted map with locally finite dual $M^\dagger$ and
  suppose that $\E[\deg(\rho)]<\infty$. Then $\omega_p$ has a unique
  infinite connected component if and only if every component of
  $\omega^\dagger_p$ is finite. It follows that
  \[
    p_u(M) = 1 - p_c(M^\dagger)
  \] almost surely and that $\omega_{p_u}$ contains a unique infinite
  connected component almost surely.
\end{thm}

Since at $p_c$ there is no infinite cluster, and at $p_u$ one exists, it
follows that $p_c<p_u$.  We remark that Benjamini and Schramm proved
their theorem for transitive planar graphs, but their proof extends
immediately to our setting.  Since we provide an alternative proof via
the curvature, we omit further details.

If $M$ does not have locally finite dual, then it must have infinitely
many infinite faces by \cref{lem:hyperfinitefaces}, so that the
underlying graph of $M$ is infinitely ended. In this case, we have that
$p_u(M)=1$ almost surely (see \cref{propinftyendspu}), while $p_c(M)<1$
by \cref{Thm:nonamenableperc}.

Thus, we have the following.

\begin{corollary}[\cref{thm:dichotomy}, \eqref{iAmen} implies \eqref{iPerc}]
  Let $(M,\rho)$ be an infinite, simply connected, invariantly nonamenable
  unimodular random rooted map with $\E[\deg(\rho)]<\infty$.  Then
  $p_c(M)<p_u(M)$ almost surely.
\end{corollary}

\subsection{Expected degree formula}\label{sec:expecteddegree}


We prove \cref{T:FMSFFormula}.  Every property of the minimal spanning
forests that we use also holds for the uniform spanning forests, so that
we also obtain an alternative proof of \cref{thm:forestformula} that
does not rely on \cref{thm:FUSFconnectivity}.

\begin{proof}[Proof of \cref{T:FMSFFormula} (and
  \cref{thm:forestformula})]

\hspace{0mm}

\paragraph{Locally finite dual case.}
Let $\omega$ be a percolation on $(M,\rho)$, and let
$\omega^\dagger = \{e^\dagger \in E^\dagger : e \notin \omega\}$ be the
dual percolation.  As in \cref{Sec:dual}, let $\eta$ be chosen uniformly
at random from the set $E_\rho^\rightarrow$ of oriented edges of $M$
emanating from $\rho$, let $\rho^\dagger=\eta^r$, let
$\P^\mathrm{rev}$ be the $\deg(\rho)$-biasing of $\P$ and let
$\P^\dagger$ be the $\deg(\rho^\dagger)^{-1}$ biasing of
$\P^\mathrm{rev}$, so that $(M^\dagger,\rho^\dagger)$ is a unimodular
random rooted map under $\P^\mathrm{rev}$. We write $\E^\dagger$ for the
expectation operator associated to $\P^\dagger$.

\begin{lem}\label{lem:omegadagger}
  For any percolation $\omega$ we have
  \[
    \E[\deg_\omega(\rho)]= \E[\deg(\rho)] - \E\bigg[\sum_{f\perp \rho}
    \deg(f)^{-1}\bigg]\E^\dagger[\deg_{\omega^\dagger}(\rho^\dagger)].
  \]
\end{lem}

\begin{proof}
  Observe that, since $\eta$ is uniformly distributed on
  $E^\rightarrow_\rho$ conditional on $(M,\rho)$, we have
  \begin{align*}
    \E[\deg_\omega(\rho)]&=\E[\deg(\rho)\mathbbm{1}(\eta\in \omega)]
    = \E\left[\deg(\rho)\left(1-\mathbbm{1}(\eta^\dagger\in \omega^\dagger)\right)
    \right],
  \end{align*}
  and so
  \begin{align*}
    \E[\deg_\omega(\rho)]=\E[\deg(\rho)]\left(1-\P^\mathrm{rev}\left(\eta^\dagger\in\omega^\dagger\right)\right).
  \end{align*}
  Similarly,
  since under the measure $\P^\dagger$ and conditional on $(M^\dagger,\rho^\dagger)$, $\eta^\dagger$ is uniformly distributed on $E^\rightarrow_{\rho^\dagger}$,
  \begin{align*}
    \E^\dagger[\deg_{\omega^\dagger}(\rho^\dagger)]=
    \E^\dagger[\deg(\rho^\dagger)]\P^\mathrm{rev}\left(\eta^\dagger\in\omega^\dagger\right).
  \end{align*}
  It follows that
  \begin{align*}\E[\deg_\omega(\rho)]&=\E[\deg(\rho)]\left(1-\frac{\E^\dagger[\deg(\rho^\dagger)]}{\E[\deg(\rho)]}\E^\dagger[\deg_{\omega^\dagger}(\rho^\dagger)]\right).
  \end{align*}
  Applying the expected degree formula \eqref{eq:dualdegree}, we deduce
  that
  \begin{align*}\E[\deg_\omega(\rho)]&=\E[\deg(\rho)]\left(1-\frac{\E\big[\sum_{f\perp
          \rho}\deg(f)^{-1}\big]}{\E[\deg(\rho)]}\E^\dagger[\deg_{\omega^\dagger}(\rho^\dagger)]\right),
  \end{align*}
  which rearranges to give the desired expression.
\end{proof}

Let $\F$ have law $\FMSF_M$. By \cref{prop:USFduality}, the dual forest
$\F^\dagger$ is distributed according to $\WMSF_{M^\dagger}$.  Since the
expected degree at the root of the WMSF in any unimodular random rooted
graph is $2$, we have $\E^\dagger[\deg_{\F^\dagger}(\rho^\dagger)]=2$
and consequently, by \cref{lem:omegadagger},
\[
  \E[\deg_\F(\rho)] = \E[\deg(\rho)]-2\E\Bigg[\sum_{f\perp
    \rho}\deg(f)^{-1}\Bigg]= \E\left[\sum_{f\perp
      \rho}\frac{\deg(f)-2}{\deg(f)}\right]= 2 -
  \frac{1}{\pi}\bbK(M,\rho).
\]
This completes the proof in the case that the dual of $M$ is locally finite.

\paragraph{Non-locally finite dual case.}
We begin with an upper bound that holds in greater generality.  Given a
unimodular random rooted map $(M,\rho)$ with $M^\dual$ not locally
finite, recall the construction of the map $M'$ from the proof of \cref{prop:submaps}, in which
each infinite face of $M$ is triangulated.

\begin{lem}\label{L:forestbound}
  If $\omega$ is a percolation on an infinite unimodular random rooted map $(M,\rho)$
  that is almost surely a spanning forest of $M$, then
  \[
    \E[\deg_\omega(\rho)] \leq 2 -\frac1\pi \bbK(M,\rho).
  \]
\end{lem}

\begin{proof}
  If $M$ has a locally finite dual, then every component of
  $\omega^\dagger$ is infinite.  It follows from \cite[Theorem
  6.1]{AL07} that
  $\E^\dagger[\deg_{\omega^\dagger}(\rho^\dagger)]\geq2$, and the claim
  follows from \cref{lem:omegadagger}.
  If $M^\dual$ is not locally finite, note that $\omega$ is also a
  percolation on the map $M'$, which has the same curvature as $M$, and the claim follows as above.
\end{proof}

For the reverse inequality we approximate $M$ by maps with locally finite duals.
Delete each edge of $M'$ not in $E$ independently with
probability $1-\eps$, and call the resulting map $M'_\eps$.  Clearly
$(M'_\eps,\rho)$ converges locally to $(M,\rho)$ as $\eps\to 0$, and it is not hard to see that $M'_\eps$ has locally finite dual a.s.\ for every $\eps>0$.  Let
$\F_\eps$ be a sample of $\FMSF_{M'_\eps}$.  Since $(M'_\eps,\rho)$ has
locally finite dual, we have that
\[
  \E[\deg_{\F_\eps}(\rho)]=2-\frac{1}{\pi}\bbK(M'_\eps,\rho) = 2 -
  \frac{1}{\pi}\bbK(M,\rho),
\]
where the second equality follows from \cref{prop:submaps}.

Since the underlying graph of $M$ is a subgraph of the underlying graph
of $M'_\eps$, the forest $\F$ stochastically dominates the restriction
$\F_\eps\cap E$ for every $\eps>0$.  Hence, by the dominated convergence theorem, we have
\begin{align*}
  \E[\deg_{\F}(\rho)]
  &\geq \E[\deg_{\F_\eps}(\rho)]-\E[\deg_{M'_\eps}(\rho) - \deg_M(\rho)] \\
  &= 2 - \frac{1}{\pi}\bbK(M,\rho)
  - \E\big[ \deg_{M'_\eps}(\rho) - \deg_M(\rho) \big]
  \xrightarrow[\eps\to0]{} 2 - \frac{1}{\pi}\bbK(M,\rho),
\end{align*}
completing the proof.
\end{proof}

\section{The Conformal Type}
\label{sec:conformal}

Given a map $M$ such that every face of $M$ has degree at least three, we
may form a surface $S(M)$ (as described in \cref{Sec:maps}) by gluing regular unit
polygons together according to the combinatorics of $M$, the boundaries of
these polygons becoming the edges of $M$ embedded in $S(M)$.  Recall that
we consider the upper half-space $\{x+iy \in \C : y>0\}$ with edges
$\{[n,n+1]:n \in \Z\}$ to be a regular $\infty$-gon.
The surface $S(M)$ is endowed naturally with a conformal structure by
defining an atlas as follows.
\begin{itemize}
\item For each face $f$ of $M$, we take as a chart the identity map from
  the interior of the regular polygon corresponding to $f$ to itself.
\item For each edge $e$ of $M$, we define an open neighbourhood of the
  interior of $e$ in $S$ by adding to $e$ the two triangles formed by the
  endpoints of $e$ and the centres of the two faces adjacent to $e$ (if
  either face is infinite, we interpret this triangle to be the infinite
  strip starting at $e$ and perpendicular to the boundary of the face). To
  define a coordinate chart on this neighbourhood, we simply place the two
  triangles next to each other in the plane.  The reason to take only a
  triangle and not the entire face is that this chart is well defined
  even if both sides of $e$ are incident to the same face.
  \begin{center}
    \includegraphics[height=17mm]{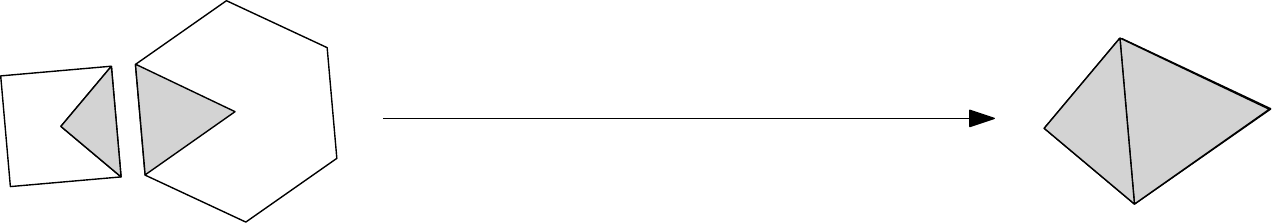}
  \end{center}
\item For each vertex $v$ of $M$, we define an open neighbourhood of $v$ in
  $S$ similarly by intersecting the corners of the faces adjacent to $v$
  with open discs of radius $1/2$ centred at $v$.  We define a chart on
  this neighbourhood by first laying the corners out sequentially around
  the origin (with possible overlapping), and then applying the function $z
  \mapsto z^{2\pi/\theta(v)}$, suitably interpreted to get an injective map
  into the plane.  The radius is chosen so that this definition remains valid
   when multiple corners of the same face are located at the same vertex.
  \begin{center}
    \includegraphics[height=5cm]{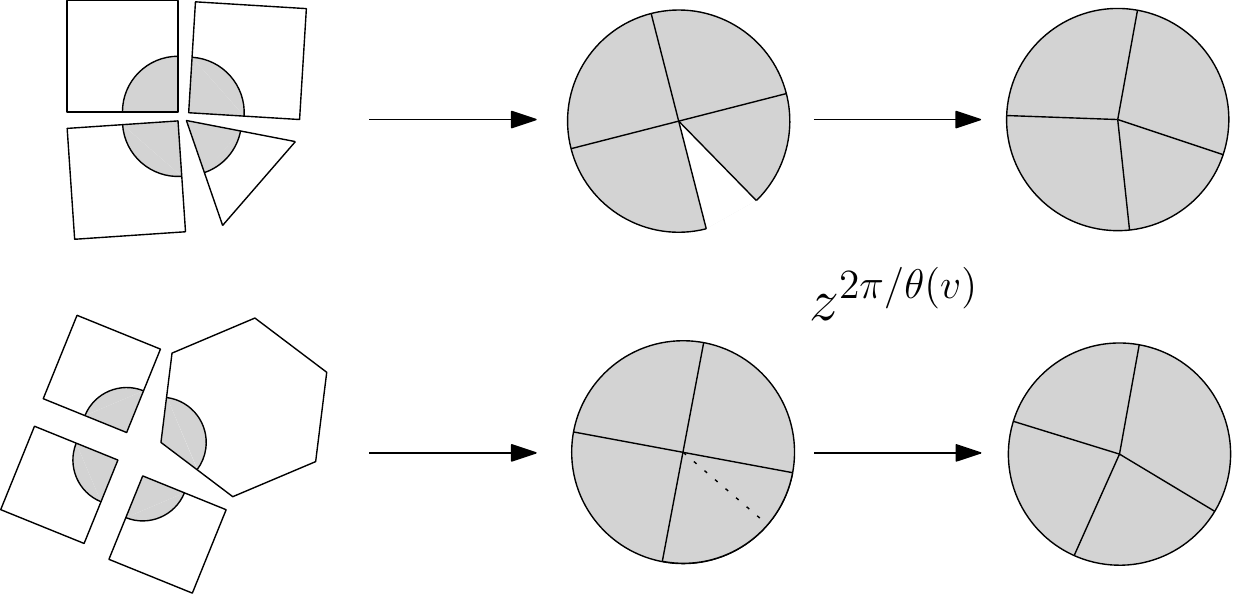}
  \end{center}
\end{itemize}
The coordinate changes are easily seen to be analytic, so that this atlas
does indeed define a Riemann surface structure on $S(M)$.  We denote this
Riemann surface by $\cR(M)$.

The definition can be extended to maps containing faces of degree 1 and
2 in various ways.  One of these is as follows: Given a map $M$, let
$\hat M$ be obtained from $M$ by triangulating faces of degree $1$ or
$2$ by adding a vertex inside each face of $M$ that has degree $1$ or
$2$, and connecting this vertex to each corner of the face.  Every face
of $\hat M$ has degree at least three, and we define
$\cR(M) = \cR(\hat M)$.  The map $M$ can be embedded in $\cR(M)$ by
restricting the natural embedding of $\hat M$ into $\cR(\hat M)$.  An
alternative way to deal with general maps is to apply this triangulation
procedure to all faces of finite degree, resulting in a map with only
faces of degree $3$ or $\infty$.

We remark that there are many other natural (and inequivalent) ways to
associate Riemann surfaces to maps. For example, we could associate to
each face of $M$ a disc of circumference $k$, with boundary split into
$k$ arcs of length one corresponding to the edges, and glue adjacent
faces according to arc length along their shared edges.  For
concreteness we proceed with the Riemann surface defined above, though
our arguments apply with minor adaptations for many other surfaces
associated with $M$.

If $M$ is simply connected, the uniformization theorem implies that
$\cR(M)$ is conformally equivalent to the sphere, the plane or the disc,
and we call $M$ \textbf{conformally elliptic, parabolic, or hyperbolic}
accordingly.    The
resulting embedding of $M$ into the sphere, plane, or disc given by
uniformizing $\cR(M)$ is unique up to M\"obius transformations of the
sphere, plane or disc as appropriate.  This is referred to as the
\textbf{conformal embedding} of $M$.  Conformal embeddings of unimodular
random planar maps are conjectured to play a key role in the theory of
two-dimensional quantum gravity, see for example \cite{curien13glimpse} and
references therein. Elliptic maps are necessarily finite --   the focus of this section is to distinguish conformally
parabolic maps from conformally hyperbolic maps.

The conformal parabolicity of an infinite planar map is equivalent to the recurrence of
Brownian motion on $\cR(M)$, which can heuristically serve as a proxy for recurrence of the
map itself.  Gill and Rohde \cite{GiRo13} proved that every
Benjamini-Schramm limit of finite planar maps with uniformly bounded face
degrees is conformally parabolic almost surely.  The main result of this
section generalises and, together with \cref{cor:invariantexhaustion},
provides a converse to their result.

\begin{thm}[\cref{thm:dichotomy}, equivalence of \eqref{iConf} and
  \eqref{iAmen} in the simply connected case.]
  \label{thm:conformal}
  Let $(M,\rho)$ be an infinite, ergodic, simply connected unimodular
  random rooted map with $\E[\deg(\rho)]<\infty$.  Then $M$ is conformally
  parabolic almost surely if and only if $(M,\rho)$ is hyperfinite.
\end{thm}

See \cref{Sec:multiplyconnected} for the equivalence of \eqref{iConf} and \eqref{iAmen}
for multiply connected maps.

\begin{remark}
Although it is natural to try to prove \cref{thm:conformal} as an application of \cref{thm.invcurvature}, it is difficult to verify the required integrability condition on the geodesic curvatures of edges. Indeed, we were able to do so only under a stronger moment assumption on $\deg(\rho)$.
\end{remark}

\begin{proof}[Proof of \cref{thm:conformal}, Conformal parabolicity implies
  hyperfiniteness]
  Let $(M,\rho)$ be a unimodular random rooted map such that $M$ is a.s.\
  conformally parabolic.  We claim that $(M,\rho)$ is hyperfinite (finite expected degree is not needed for this implication).  Let $z$
  be a conformal embedding of $M$ into the plane, which is uniquely
  determined up to translation and scaling.

  Consider the image of the vertices $z(V)$.   We first claim that the convex hull $C$ of $z(V)$ is either the
  full plane, a half plane, or a doubly infinite strip; the remainder of the proof will be split into cases accordingly.  Indeed, since
  $z(V)$ is locally finite, if $C$ is any other set then there must exist a
  vertex $v$ of $M$ such that $z(v)$ is an extreme point of $C$.  The external angles of $C$ at these
  vertices are positive, add up to at most $2\pi$, and are independent of the choice of
  $z$. Transporting a mass $\alpha$ from every vertex of $M$ to every vertex
  such that $z(v)$ is an extreme point of $C$ with external angle $\alpha$ gives a contradiction since
  the mass out from each vertex is bounded while vertices such that $z(v)$ is an extreme point of $C$
  receive infinite mass.
  We now show that in each of the three cases, $M$ is hyperfinite.

  \paragraph{Case 1: $C$ is a strip.}
  Without loss of generality we may assume the strip is parallel to
  $\R$.  (Otherwise apply the subsequent argument to a suitable rotation
  of $z$.)  Consider the bi-infinite path on $z(V)$ visiting the vertices of $M$
  in the order given by their real part (breaking ties randomly).  This
  defines a unimodular coupling between $(M,\rho)$ and $\Z$, and it
  follows from \cref{lem:hypcoupling} that $(M,\rho)$ is hyperfinite in
  this case.

  \paragraph{Case 2: $C$ is a half-plane.}
  Without loss of generality we may assume that $C$ is the upper
  half-plane.
  We will use the linear structure of the boundary of $H$ to show that
  $(M,\rho)$ is coupling equivalent to $\Z$, so that it will follow from
  \cref{lem:hypcoupling} that $(M,\rho)$ is hyperfinite. Since $z(V)$ is
  locally finite, there exist vertices $v$ such that the cone $\{w :
  \arg(z(v)-w)\in(\pi/3,2\pi/3)\}$ contains no other points of $z(V)$. (The
  choice of angles is arbitrary; any cone with finite intersection with the
  half-plane will do.)  Call such vertices \textbf{exposed}.

  It follows from ergodicity and unimodularity that a.s.\ there are
  infinitely many exposed vertices.  Observe that, again using the fact
  that $z(V)$ is locally finite, for every bounded interval $[a,b]$,
  there exist at most finitely many exposed vertices $v$ that have
  $\Re z(v) \in [a,b]$.  Define a path whose vertices are the exposed
  vertices of $M$, arranged in order of their real part.
  It follows that $(M,\rho)$ and $\Z$ are coupling equivalent as claimed.

  \paragraph{Case 3: $C$ is the full plane.}
  We define $D$ to be the Delaunay tessellation with vertex set given by
  the set of points $z(V)$.
  That is, $D$ is a map embedded in $\C$ with vertex set $z(V)$.  Faces
  of $D$ correspond to discs (or half planes) with no vertices in their
  interior and at least $3$ vertices on their boundary, so that each
  face is circumscribed in the corresponding disc.
  Note that the isomorphism class of $(D,\rho)$ is independent of the
  choice of the conformal embedding, which is unique up to homothety,
  and so $(M,\rho)$ and $(D,\rho)$ are coupling equivalent.

  Since every half plane intersects $z(V)$, the faces of $D$ are all
  finite.  If $D$ itself is not locally finite, then for every infinite
  degree vertex $v$ of $D$ there exists an infinite sequence of
  vertices $u_i$ and closed discs or half-planes $C_i$ such that $C_i$
  contains both $z(v)$ and $z(u_i)$ in its boundary and no points of
  $z(V)$ in its interior. By taking a subsequential limit, it follows
  from the fact that $z(V)$ is locally finite that there must exist a
  half-space containing $z(v)$ in its boundary and no points of $z(V)$
  in its interior, contradicting the assumption that $C$ is the full plane.

  Thus $D$ and its dual are locally finite.  In this case, it follows
  from the measurability of the conformal embedding that $(D,\rho)$ is a
  unimodular random rooted map, and that $(M,\rho)$ and $(D,\rho)$ are
  coupling equivalent.  Thus, by \cref{lem:hypcoupling}, it suffices to
  prove that $(D,\rho)$ is hyperfinite.

  Define a mass transport as follows.  For each vertex $u$ and face $f$
  of $D$ incident to $u$, let $\ang(f,u)$ be the angle of the corner of
  $f$ at $u$.  Transport a mass of $\ang(f,u)/\deg(f)$ from $u$ to each
  of the vertices incident to $f$, including $u$ itself.  The mass sent
  out by each vertex $u$ is $2\pi$.   Since each face is a polygon
  in the plane, the mass received is
  \[
    \sum_{f\perp u} \frac{\text{ sum of internal angles of $f$}}{\deg(f)}
    = \sum_{f\perp u} \frac{\deg(f)-2}{\deg(f)}\pi.
  \]
  Applying the mass-transport principle yields that the average
  curvature of $(D,\rho)$ is zero.  It follows from the equivalence
  of items \eqref{iCurv} and \eqref{iAmen} of \cref{thm:dichotomy} that
  $(D,\rho)$ is hyperfinite as claimed.
\end{proof}

\medskip

%
For the converse, suppose $M$ is a conformally hyperbolic, so that we
have a conformal map $\phi$ from $\cR(M)$ to the hyperbolic plane $\H$,
which is unique up to isometries.  Our strategy is to use $\phi$ to give
a unimodular coupling of $M$ with another map which is known to be
invariantly nonamenable.  One possibility is to use (as above) the
Delaunay tessalation on $\phi(V)$.  The mass transport above would be
used to prove that $D$ has negative mean curvature, and hence is
nonamenable.  The difficulty lies in establishing the local finiteness
of the resulting map and its dual.  Indeed with no further assumptions
the resulting map can have infinite faces, and a modification is needed.

A crucial first step is establishing that $\phi(V)$ is reasonably dense
in $\H$, in the sense that there is a unimodular partition of $\H$ which
assigns finite mean area to every vertex.  For each corner $c$ at each
vertex $v$ of $M$, let $U_c$ be the quadrilateral in the face of $c$
with corners given by $v$, the centre of the face, and the midpoints of
the two edges forming the corner.  If the face is infinite, we take
$U_c$ to be the half-infinite strip in $f$ with right angled corners at
the mid-points of the two edges of the corner (see \cref{F:partition}).
We define a partition of $\cR(M)$ by
\[
U_v = \bigcup_{c\perp v} U_c,
\]

\begin{figure}
\includegraphics[height=0.25\textwidth]{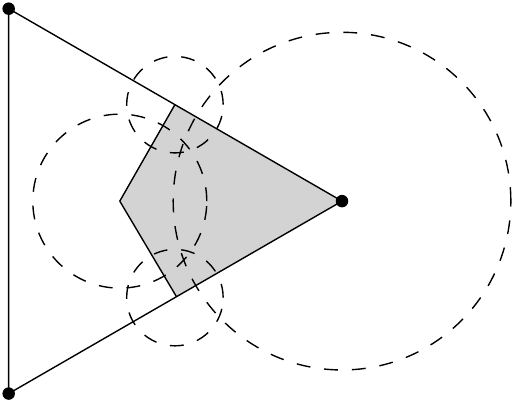} \qquad \includegraphics[height=0.25\textwidth]{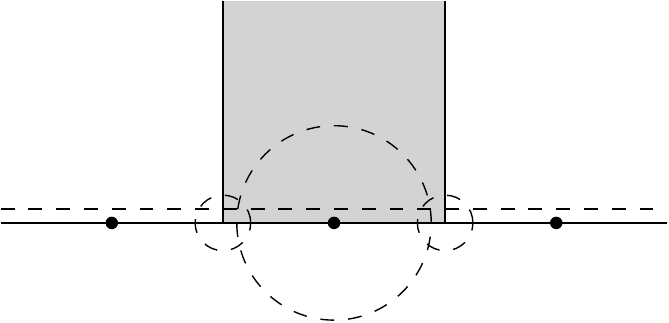}
\caption{The covering of $U_e$ (grey) by discs (dashed boundaries) used
  in the proof of \cref{lem:hypareaestimate}. Left: the case that $e^r$
  has degree three. Right: the case that $e^r$ has infinite degree.}
\label{F:partition}
\end{figure}


Given sets $K\subset S \subset \cR(M)$, with $S$ simply connected, let
$\phi_S$ be a conformal equivalence between $S$ and the unit disc.  We write
$\area_{\bbH}(K,S)$ for the hyperbolic area of $\phi_S(K)$ (which is
independent of the choice of $\phi_S$).  We also write
$\area_{\bbH}(K) = \area_{\bbH}(K,\cR(M))$.  Recall the Schwarz-Pick Lemma: If
$\cR$ is a Riemann surface that is conformally equivalent to the disc,
$S$ is a simply connected open subset of $\cR$, and $K$ is a Borel
subset of $S$, then the hyperbolic area of $K$ considered as a subset of
$S$ is greater than or equal to the hyperbolic area of $K$ considered as
a subset of $\cR$.  Thus $\area_{\bbH}(K) \leq \area_{\bbH}(K,S)$.

\begin{lem}\label{lem:hypareaestimate}
  There exists a constant $C$ such that the following holds.  Let $M$ be
  a conformally hyperbolic, simply connected map such that every face of
  $M$ has degree at least three.  Then for every vertex $v$,
  \[
    \area_\H(U_v) \leq C \deg(v).
  \]
\end{lem}

We shall require the following fact: For every $0<\eps<1$, the
hyperbolic area of the set $\{ z \in \C : |z|<1-\eps\}$, considered as a
subset of the open unit disc, is given by
\begin{equation}\label{eq:hyparea}
  \frac{4\pi(1-\eps)^2}{1-(1-\eps)^2} \leq \frac{2\pi}{\eps}.
\end{equation}

\begin{proof}
  For each corner $c$ incident to vertex $v$ and face $f$, let
  $U^1_c,\dots,U^4_c$ be the following subsets of the polygon associated
  to $f$:
  \begin{enumerate}
  \item $U^1_c$ and $U^2_c$ are the $1/8$ neighbourhoods of the midpoints
    of the edges of $c$.
  \item $U^3_c$ is the $7/16$ neighbourhood of $v$.
  \item If $f$ has finite degree, then $U^4_c$ is the intersection of
    $U_c$ with the disc that is centred at the centre of the polygon
    corresponding to $f$ and that reaches distance $1/16$ from the
    boundary of the polygon.  If $f$ has infinite degree, we let $U^4_c$
    be the part of $U_c$ at distance at least $1/16$ from the boundary
    of the half-plane corresponding to $f$).
  \end{enumerate}
  It is easily verified by elementary trigonometry that the $U_c$ is
  contained in the union $\bigcup_{i=1}^4 U^i_c$.  See
  \cref{F:partition} for an illustration.

We first prove that each of $U^1_c, U^2_c,$ and $U^4_c$
  have uniformly bounded hyperbolic areas.
  For $U^1_c$, let $S$ be the ball of radius $1/4$ around the midpoint
  of an edge $e$, so that $U^1_c \subset S$. Recall that for an edge
  $e$, the chart at $e$ includes the two triangles formed by the
  endpoints of $e$ and the centres of the faces adjacent to $e$.  These
  triangles have disjoint interiors even if the same face is on both
  sides of $e$.  It is easily verified that $S$ is always contained in
  the domain formed by placing the two triangles together.
  Since $U^1_c$ is the ball of radius $1/8$ around the midpoint of $e$,
  it follows by Schwarz-Pick and \eqref{eq:hyparea} that
  $\area_{\bbH}(U^1_c) \leq \area_{\bbH}(U^1_c,S) \leq C$ for some constant
  (namely, $4\pi$).  The corresponding claim for $U^2_c$ follows
  similarly.

  We next claim that $\area_{\bbH}(U^4_c) \leq 16$.  If the corresponding
  face $f$ has finite degree, let $S$ be the disc circumscribed in the
  polygon of $f$, and let $K$  be the disc that has the same centre as $S$ and radius
  $1/16$ smaller than that of $S$.  The circumference of $S$ is less than the
  perimeter of $f$, and hence the radius $\rho$ of $S$ is at most
  $\deg(f)/2\pi$.  Thus, it follows from \eqref{eq:hyparea} that
  \[
    \area_{\bbH}(K,S) \leq 16 \deg(f).
  \]
  Since $U^4_c$ is a wedge of $K$, it follows by the symmetry of $S$ that
  $\area_{\bbH}(U^4_c,S) \leq 16$.  By Schwarz-Pick, the same bound holds for
  $\area_{\bbH}(U^4_c)$.
  If $f$ has infinite degree, then as a subset of $f$, $U^4_c$ is
  isometric to the set $\{x+iy : x\in[0,1], y\geq1/16\}$ in the half
  plane representation of $\H$. This set is easily verified to have
  hyperbolic area $16$. (One could also deduce this from the finite degree case by taking a
  limit.)

  Finally, we claim that for every vertex $v$, the hyperbolic area of
  $U^3_v :=\bigcup_{c \perp v} U^3_c$ is at most $C \deg(v)$, where $C$
  is a universal constant.  Observe that $U^3_v$ is simply the ball of
  radius $7/16$ around $v$ in $\cR(M)$.  Let $S$ be the ball around $v$
  of radius $1/2$.  The chart provided for $S$ maps $U^3_v$ to a ball of
  radius $(7/16)^{2\pi/\theta(v)}$, while $S$ gets mapped to the ball of
  radius $(1/2)^{2\pi/\theta(v)}$ with the same centre.  It follows from \eqref{eq:hyparea} that
  \[
    \area_{\bbH}(U^3_v,S) \leq  \frac{2\pi}{1-(7/8)^{2\pi/\theta(v)}} \leq C \theta(v) \leq \pi C
    \deg(v),
  \]
  where $C$ is a constant. (Verifying that such a constant exists is a
  simple calculus exercise.) The claim now follows from Schwarz-Pick.

  The lemma follows by combining the estimates given for $U^3_v$ and for each of $U^1_c$, $U^2_c$, and $U^4_c$
  for each corner at $v$.
\end{proof}

We remark that a simple additional argument gives the slightly stronger
statement that $\area_{\bbH}(U_c)$ is uniformly bounded for each corner.

\begin{proof}[Proof of \cref{thm:conformal}, Conformal hyperbolicity
  implies invariant nonamenability.]
  We may assume that every face of $M$ has degree at least three.
  Otherwise, we consider the map $\hat M$ used to define $\cR(M)$, all the faces of which have
  degree at least $3$, and which is coupling equivalent to $M$.  Let
  $\phi$ be a conformal equivalence between $\cR(M)$ and the hyperbolic
  plane.

  Let $Z$ be an independent Poisson point process of intensity $1$ on
  the hyperbolic plane, and let $D$ be the Delaunay triangulation
  associated to $Z$.  Since $D$ is a Poisson-Delaunay triangulation of
  the hyperbolic plane, $D$ is invariantly nonamenable (see
  \cite{BPP,AHNR15}).  To prove nonamenability of $M$ we give
  a unimodular coupling of $M$ and $D$, using a larger graph $G$, and
  conclude by applying \cref{lem:hypcoupling}.

  The graph $G$ has vertex set $V\cup Z$, and has as edges the edges of
  $M$, the edges of $D$, and an edge connecting each $v \in V$ to every
  point $z \in Z \cap \phi(U_v)$.  We mark the edges of $G$ according to
  which of these three types they come from.  Note that the law of
  $(G,\rho)$ does not depend on the choice of $\phi$.  It follows from
  \cref{lem:hypareaestimate} that the expected number of points in
  $Z \cap \phi(U_{\rho})$ is finite.  It is easily verified, using the
  measurability of the conformal embedding, that if we sample $(G,\rho)$
  biased by $1+|Z \cap \phi(U_\rho)|$ and then let $\hat \rho$ be
  uniform on the set $\{\rho\} \cup (Z\cap\phi(U_\rho))$, then the
  resulting random rooted graph $(G,\hat \rho)$ is unimodular.
  Similarly, if we sample $(M,\rho)$ and $Z$ biased by
  $|Z \cap \phi(U_\rho)|$, and let $\rho'$ be uniform on the set
  $Z \cap \phi(U_\rho)$, then the resulting graph $(D,\rho')$ is
  unimodular.  Thus, we have defined a unimodular coupling between
  $(M,\rho)$ and $(D,\rho')$.
\end{proof}

\section{Multiply-connected maps}\label{Sec:multiplyconnected}

\subsection{The topology of unimodular random rooted maps.}

In this section we study multiply-connected unimodular random rooted
maps.  We begin by classifying the possible topologies of the surface
associated to a unimodular random rooted map $(M,\rho)$.  Biringer and
Raimbault \cite{BirRai14} classified the possible topologies of
unimodular random rooted complete, orientable, hyperbolic surfaces.
Their methods readily generalise to our setting, yielding the following
theorem.  In fact, the proof is slightly less technical in our setting,
and we provide a quick sketch below.

\begin{thm}[Topology of unimodular random rooted maps]\label{thm:topology}
  Let $(M,\rho)$ be an infinite unimodular random rooted map.  Then the
  surface associated to $M$ is almost surely homeomorphic to one of the
  following surfaces: the plane, the cylinder, the Cantor tree, the
  infinite prison window, Jacob's ladder, or the blossoming Cantor tree.
  The type is determined by how many ends the surface has and whether or not it is planar.
\end{thm}

Here, the \textbf{cantor tree} is a `tree made of tubes', and is
homeomorphic to the complement of the Cantor set in the sphere, the
\textbf{infinite prison window} is `the lattice $\Z^2$ made of tubes',
\textbf{Jacob's ladder} is `an infinite ladder made of tubes', and the
\textbf{blossoming Cantor tree} is a Cantor tree with a handle attached
near each bifurcation. See \cref{fig:topology} for illustrations.  Be
warned that homeomorphism is an extremely weak notion here.  For
example, `the lattices $\Z^d$ made of tubes' are all homeomorphic to
each other for all $d\geq2$.
We leave it to the reader to verify that each of
the surfaces listed in \cref{thm:topology} can occur as the almost sure homeomorphism class of a unimodular random rooted map.

\begin{figure}
  \centering
  \includegraphics[width=0.95\textwidth]{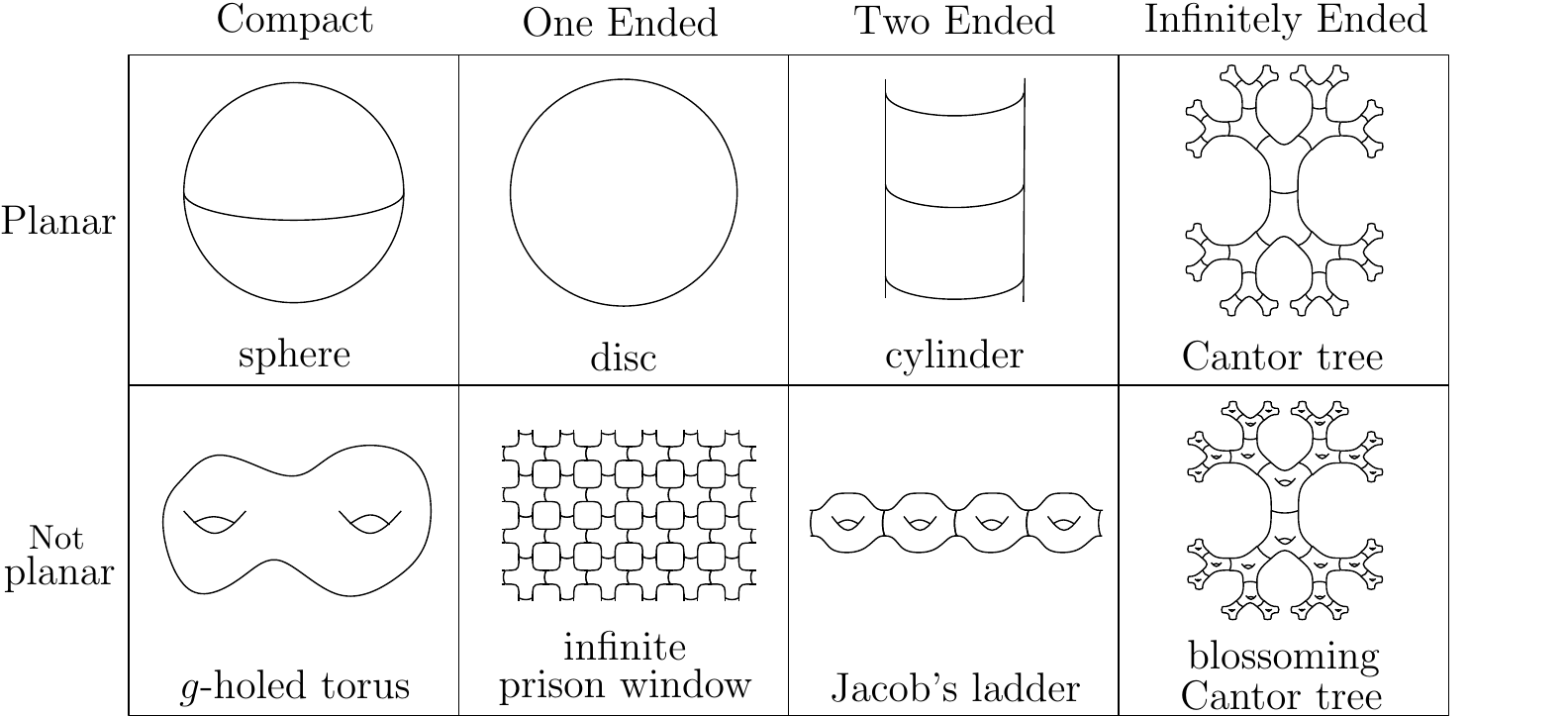}
  \caption{Possible topologies of a unimodular random map. The surface
    $S(M)$ associated to a unimodular random map $(M,\rho)$ is almost
    surely homeomorphic to one of the above.}
  \label{fig:topology}
\end{figure}

In \cite{BirRai14}, Biringer and Raimbault must also allow for the
surfaces above to be punctured at a locally finite set of points,
corresponding to isolated ends of the surface. This does not occur in
our setting, as the surfaces corresponding to unimodular random rooted
maps do not have isolated ends.

\begin{proof}[Sketch of proof]
  An \textbf{end} $\xi$ of an infinite graph $G$ may be defined as a
  function that assigns a connected component $\xi_K$ of $G\setminus K$
  to each finite set of vertices $K$ of $G$, and satifies the
  consistency condition that $\xi_{K'}\subseteq \xi_{K}$ whenever
  $K' \supseteq K$.  The \textbf{space of ends} of $G$, denoted
  $\partial\sE(G)$ is the topological space with the set of ends of $G$
  as its underlying set and with a basis of open sets given by sets of
  the form $A_{K,W} = \{\xi \text{ an end of $G$} : \xi_K = W\}$, where
  $K \subset V$ is finite and $W \subset V$ is a connected component of
  $G \setminus K$. Note that the basis sets are also closed, so that the
  space of ends is always \emph{zero-dimensional}, that is, its topology
  is induced by a basis of sets that are both open and closed.  The
  space of ends of a surface is defined similarly, replacing instances
  of the word `finite' by `compact' above, and is also zero-dimensional.

  It is well-known that every unimodular random rooted graph either has
  one, two, or infinitely many ends, and, in the last case, the space of
  ends does not have any isolated points
  \cite[Proposition 6.10]{AL07}.  Since the space of ends of any graph
  is also compact, it follows in the last case that the space of ends is
  homeomorphic to the Cantor set (which, by Brouwer's Theorem
  \cite{brouwer1910structure}, is the only compact, zero-dimensional
  Hausdorff space with no isolated points). A similar proof applies to
  show that if $(M,\rho)$ is a random rooted map with associated surface
  $S=S(M)$, then $S$ has either one, two, or infinitely many ends and in
  the last case the space of ends of $S$ is homeomorphic to a Cantor
  set.  Next, a standard mass transport argument (each vertex sends unit
  mass to the `nearest handle') shows that if $S$ contains handles, then
  the handles of $S$ accumulate towards every end of $S$.  That is, if
  $S$ has handles then, for every compact subset $K$ of $S$, every
  non-precompact connected component of $S\setminus K$ contains a
  handle.

  We next apply the classification theorem for non-compact surfaces due
  to Ker\'ekj\'art\'o and Richards \cite{richards1963classification},
  which states that if two non-compact orientable surfaces $S_1$ and
  $S_2$ have the same number of handles (which in our case will be zero
  or infinity) and there exists a homeomorphism
  $\phi:\partial\sE(S_1)\to\partial\sE(S_2)$ such that the handles of
  $S_1$ accumulate to $\xi \in \partial\sE(S_1)$ if and only if the
  handles of $S_2$ accumulate to $\phi(\xi)$, then $S_1$ and $S_2$ are
  homeomorphic and $\phi$ extends to a homeomorphism from the ends
  compactification of $S_1$ to the ends compactification of $S_2$. Thus,
  by the above discussion, the homeomorphism class of $S=S(M)$ is
  determined almost surely by its number of ends and by the existence or
  non-existence of handles. This yields the six different possibilities
  listed in the statement of the theorem (see Figure 6).
\end{proof}


The main result of this section is that the average curvature of a
unimodular random rooted map restricts the possible topologies of the
map.

\begin{thm}[Topology and average curvature]
  \label{thm:conformaltype}
  Let $(M,\rho)$ be an ergodic unimodular random map. Then the almost
  sure conformal type of $M$ is determined by its average curvature:
  Either
  \begin{enumerate}[nosep]
  \item The average curvature of $(M,\rho)$ is positive, in which case
    $M$ is conformally elliptic and $S(M)$ is homeomorphic to the sphere
    almost surely,
  \item the average curvature of $(M,\rho)$ is zero, in which case $M$
    is conformally parabolic and $S(M)$ is homeomorphic to the plane,
    the cylinder, or the torus almost surely,
  \end{enumerate}
  or else
  \begin{enumerate}
  \item[(3)] the average curvature of $(M,\rho)$ is negative, in which
    case $M$ is conformally hyperbolic $S(M)$ is homoemorphic to either
    the plane, the blossoming Cantor tree, Jacob's ladder, the infinite
    prison window, or a compact surface of genus at least two almost
    surely. \end{enumerate}
\end{thm}

The theorem will follow by combining \cref{thm:conformal},
\cref{thm:topology} and the notion of the universal cover of a map.
%
Recall that a surjective, holomorphic function $\Pi:S\to S'$ between two
Riemann surfaces is a \textbf{holomorphic covering} if it is locally a
homeomorphism, that is, if for every $x\in S$ there exists an open
neighbourhood $U$ of $S$ such that the restriction of $\Pi$ to $U$ is a
homeomorphism between $U$ and its image.  Given a Riemann surface $S$,
the \textbf{universal cover} of $S$ is a simply connected Riemann
surface $\tilde S$ together with a covering $\Pi:\tilde S\to S$. The
universal cover exists for any $S$, and is unique in the sense that if
$\Pi':\tilde S '\to S$ is another simply connected Riemann surface
covering $S$, then there exists a conformal equivalence
$\phi:\tilde S' \to \tilde S$ such that $\Pi' = \Pi \circ \phi$.

The universal cover of a map is defined analogously.  Given a pair of
maps $M=(G,\sigma)$ and $M'=(G',\sigma')$, we say that a graph
homomorphism $\phi: G \to G'$ is a \textbf{map homomorphism} if
$\sigma'\circ\phi=\phi\circ\sigma$, and say that $\phi$ is a
\textbf{covering} if for every vertex $v$ and every face $f$ of $M$, the
restriction of $\phi$ to each of $\{e \in E^\rightarrow : e^-=v\}$ and
$\{e \in E^\rightarrow : e^r=f\}$ is injective.  The \textbf{universal
  cover of a map} $M$ is a simply connected map $\tilde M$ together with
a covering $\pi:\tilde M \to M$. Every map has a universal cover, and
the universal cover of a map $M$ is unique in the sense that if
$\pi':\tilde M' \to M$ is a covering from a simply connected map
$\tilde M'$ to $M$, there exists an isomorphism of maps
$\phi: \tilde M \to \tilde M'$ such that $\pi'\circ\phi=\pi$.

We remark that the universal cover $\pi:\tilde M\to M$ of $M$ may be constructed by
lifting every edge of $M$ in the surface $S(M)$ to the universal
cover $\Pi:\tilde S(M)\to S(M)$ (see e.g.\ \cite[p. 60]{Hatcher}
for the topological notion of path lifting).  In particular, if
$\cR(\tilde M)$ is the Riemann surface associated to $\tilde M$, then
there exists a conformal equivalence $\Phi:\cR(\tilde M)\to\tilde\cR(M)$
such that $\Pi\circ\Phi\circ\tilde z = z\circ\pi$.  (See e.g.\
\cite[Section 9.2]{St05} for a direct construction.)

Covers of unimodular maps are unimodular.  The following is proven in
\cite[Section 4.1]{AHNR15}.

\begin{lem}
  Let $(M,\rho)$ be a unimodular random rooted map. Let $(\tilde M,\pi)$
  be the universal cover of $M$ and let $\tilde \rho$ be an arbitrary
  element of $\pi^{-1}(\rho)$. Then $(\tilde M, \tilde\rho)$ is a
  unimodular random rooted map.
\end{lem}


\begin{proof}[Proof of \cref{thm:conformaltype}]
  Observe that $\kappa_{\tilde M}(\tilde \rho)=\kappa(\rho)$ for any
  rooted map $(M,\rho)$, so that
  $\bbK(\tilde M,\tilde \rho) = \bbK(M,\rho)$.  By applying
  \cref{thm:conformal} and the classical theory of Riemann surfaces, we
  obtain that
  \begin{itemize}
  \item $\bbK(M,\rho)>0$ if and only if $M$ is finite and simply
    connected and $\tilde\cR(M) = \cR(M)$ is conformally equivalent to
    the sphere,
  \item $\bbK(M,\rho)=0$ if and only if $\tilde\cR(M)$ is conformally
    equivalent to the plane, if and only if $\cR(M)$ is conformally
    equivalent to one of the plane $\C$, the cylinder $\C/\Z$ or a torus
    $\C/\Lambda$ for some lattice $\Lambda\subset \C$, and
  \item $\bbK(M,\rho)<0$ if and only if $\tilde\cR(M)$ is conformally
    equivalent to the hyperbolic plane (or disc).
  \end{itemize}
  In the last case there are many possibilities for the conformal
  equivalence class $\cR(M)$; any topology other than the sphere, plane,
  cylinder or a torus is possible.  We conclude by applying the
  additional topological constraints on $\cR(M)$ imposed by
  \cref{thm:topology}.
\end{proof}

We next connect the topology of $S(M)$ to the number of ends of the
underlying graph.  If $S(M)$ is homeomorphic to the plane or the
infinite prison window, it is possible for the underlying graph to have
$1$, $2$, or infinitely many ends. (Examples in the plane case are
$\Z^2$, $\Z$ and a regular tree; Add handles to these for examples with
the topology of the infinite prison window.)  The number of ends of the
underlying graph of a map $M$ is at least the number of ends of $S(M)$.
Thus if $S(M)$ is homeomorphic to the Cantor tree or the blossoming
Cantor tree, and at least two ends if $S(M)$ is homeomorphic to the
cylinder or Jacob's ladder.  The case of two-ended $S(M)$ is covered by
the following:

\begin{lem}\label{lem:twoendcylinder}
  Let $(M,\rho)$ be a unimodular random rooted map with
  $\E[\deg(\rho)]<\infty$.  If $S(M)$ is homeomorphic to the cylinder or
  to Jacob's ladder almost surely, then the underlying graph of $M$ is recurrent and
  two-ended almost surely.
\end{lem}


\begin{proof}
  We bias by $\deg(\rho)$ and prove the equivalent statement for a
  reversible map.  We may also assume that $(M,\rho)$ is ergodic.
  Suppose for contradiction that $M$ is either transient or has more than two ends almost surely. In the latter
  case, $M$ has infinitely many ends almost surely, is invariantly
  nonamenable, and hence transient almost surely.

  Since $S(M)$ is two-ended almost surely, there exists some $r$ and $D$
  such, with positive probability, the ball $B_r(M,X_n)$ of radius $r$
  about $X_n$ in $M$ has degree sum at most $D$ and the complement
  $S(M)\setminus z(B_r(M,X_n))$ has two non-precompact connected
  components, each of which is necessarily one-ended.  Denote this event
  by $A_n$.
  By stationarity, $A_n$ occurs for infinitely many $n$ almost surely.
  Let $(n_m)_{m\geq0}$ be a sequence of times at which $A_n$ occurs, and
  such that the balls $B_r(M,X_{n_i})$ are disjoint.  Let $W_1,W_2$ be
  the subgraphs of $M$ in the two unbounded components of
  $S\setminus z(B_r(M,X_{n_0}))$.  Since $M$ is transient almost surely,
  the simple random walk $(X_n)_{n\geq0}$ eventually stays in one of the
  $W_i$, and so $W_i$ must be transient.  However, the balls
  $B_r(M,X_{n_m})$ yield an infinite collection of disjoint cutsets of
  degree sum at most $D$ separating $\rho$ from infinity in the subgraph
  induced by $W_i$.  Thus, this graph is recurrent by the Nash-Williams
  criterion \cite{LP:book}, a contradiction.
\end{proof}


\subsection{\cref{thm:dichotomy} in the multiply-connected planar case}

In light of \cref{thm:topology} we need to consider maps where $S(M)$ is homeomorphic to either
the cylinder or the Cantor tree.

Suppose that $(M,\rho)$ is an infinite, multiply-connected unimodular
random planar map with $\E[\deg(\rho)]<\infty$.  If $\bbK(M,\rho)=0$,
then the proof of \cref{thm:conformaltype} implies that $\cR(M)$ is
conformally equivalent to the cylinder.  \cref{lem:twoendcylinder} then
implies that the underlying graph of $M$ is recurrent and two-ended almost surely. We deduce that
\cref{iAmen,iLiouville,iDirichelet,iIntersect,iVEL,iRecurrent,iTree,iTree2}
of \cref{thm:dichotomy} hold for $(M,\rho)$ as an immediate consequence of recurrence.  The remaining items of
\cref{thm:dichotomy} hold for $(M,\rho)$ as a consequence of invariant
amenability.

Now suppose that $\bbK(M,\rho)<0$.  In this case,
\cref{thm:conformaltype} implies that $S(M)$ is almost surely
homeomorphic to the Cantor tree and consequently that the underlying
graph of $M$ is infinitely-ended almost surely by
\cref{Lem:endscomparison}, and hence that $(M,\rho)$ is invariantly
nonamenable.  The following two propositions, which are well-known to
experts, show that the negations of \cref{iPerc,iDirichelet} hold for
$(M,\rho)$.  The negations of the remaining items of
\cref{thm:dichotomy} follow from the negations of
\cref{iPerc,iDirichelet} using implications that are valid for all
unimodular random rooted graphs and that we have already reviewed
earlier in the paper; see the green and blue arrows in \cref{fig:logic}.

\begin{prop}\label{propinftyendspu}
  Let $(G,\rho)$ be a unimodular random rooted graph with
  $\E[\deg(\rho)]<\infty$, and suppose that $G$ is infinitely ended
  almost surely. Then $p_u(G)=1$ almost surely.
\end{prop}

\begin{proof}[Proof of \cref{propinftyendspu}]
  If $p_c=1$ the claim holds trivially, so suppose not.  Fix some
  $p\in(p_c,1)$, and let $I\subset V$ be the union of all infinite
  components of Bernoulli $p$-percolation.  We first observe that for
  every finite $K\subset V$, each infinite component of $G\setminus K$
  intersects $I$ almost surely.  Indeed, if not, we get a contradiction
  by transporting unit mass from each vertex to the nearest vertex of
  $I$.  There is some finite set $K$ so that $G\setminus K$ has multiple
  infinite connected components.  With positive probability all edges
  incident to $K$ are closed, and thus there are multiple infinite
  percolation components.
\end{proof}

\begin{prop}\label{prop:inftyendsdirichlet}
  Let $(G,\rho)$ be a unimodular random rooted graph with
  $\E[\deg(\rho)]<\infty$, and suppose that $G$ is infinitely ended
  almost surely. Then $G$ admits non-constant harmonic functions of
  finite Dirichlet energy almost surely.
\end{prop}

\cref{prop:inftyendsdirichlet} will require the following lemma.

\begin{lem}\label{lem:transientends}
  Let $(G,\rho)$ be a unimodular random rooted graph, and suppose that
  $G$ is infinitely ended almost surely.
  Then for every finite set $K \subset V$, every infinite connected
  component $W$ of $G \setminus K$ is transient.
\end{lem}

\begin{proof}
  Since $G$ has infinitely many ends a.s., we have that $(G,\rho)$ is
  invariantly nonamenable and hence that $G$ is transient.  Let $F(u,v)$
  be the probability that a simple random walk started at $u$ hits $v$.
  Suppose for a contradiction that there exists a finite set
  $K\subset V$ and a recurrent infinite connected component $W$ of
  $G \setminus K$.  It follows that for every vertex $v\in W$, a random
  walk started at $v$ must hit the set $K$ almost surely, and hence
  that, for each vertex $u$ of $K$, we have
  \[
    \inf_{v \in W} F(v,u) \geq \inf_{v \in W} F(v,K) \inf_{w\in K}
    F(w,u) > 0.
  \]
  Since $G$ is transient but $W$ is recurrent, we have that for every $\eps>0$ there exist at most
  finitely many vertices $v \in W$ such that $F(u,v) \geq \eps$, i.e.\
  $F(u,v)\to 0$ as $v\to\infty$ in $W$.  Symmetry of the Green's function
  $\deg(x) G(x,y)$ gives
  \[
    \frac{F(u,v) G(v,v)}{\deg(v)} = \frac{F(v,u) G(u,u)}{\deg(u)}.
  \]
  It follows that $G(v,v)/\deg(v)$ tends to infinity as $v\to\infty$ in
  $W$.

  Choose $C$ sufficiently large that $G(\rho,\rho)/\deg(\rho) \leq C$
  with positive probability.  Define a mass transport by, for each
  vertex $v$ of $G$, transporting a mass of $1$ to the closest vertex to
  $v$ that has $G(w,w,)/\deg(w) \leq C$.  (If there are multiple choices
  of the vertex $w$, choose one uniformly.)  Then every vertex sends a
  mass of at most one but, in the situation described, some vertices
  recieve an infinite amount of mass.  This contradicts the mass
  transport principle.
\end{proof}

\begin{proof}[Proof of \cref{prop:inftyendsdirichlet}]
  It is well-known that if $G$ is a graph and there exists a finite set
  $K$ such that $G\setminus K$ has multiple transient connected
  components, then $G$ admits a non-constant harmonic Dirichlet
  function: The probability that the walk eventually stays in a
  particular connected component is such a function (see
  \cite[Exercise~9.23]{LP:book}).
\end{proof}

\section{Soficity}

We are now ready to prove \cref{thm:soficmaps}.
A unimodular random rooted graph $(G,\rho)$ is called
\textbf{strongly sofic} if $(G,\rho,m)$ is sofic for every marking $m$
of $(G,\rho)$.  (The constant marking shows this implies soficity.)
One main ingredient to \cref{thm:soficmaps} is the following,
which was proven for Cayley graphs of free groups by \cite{Bowen03}.

\begin{thm}[Bowen \cite{Bowen03}; Elek \cite{Elek10}; Elek and Lippner
  \cite{ElekLipp10}; Benjamini, Lyons and Schramm \cite{URT}]
  \label{thm:sofictrees}
  Every unimodular random rooted tree is strongly sofic.
\end{thm}

The following theorem is an adaptation of a related theorem of Elek and
Lippner \cite{ElekLipp10} in the setting of group actions.  Even in our
setting, it is well-known to experts that treeable unimodular random
graphs (i.e., unimodular random graphs admitting a unimodular random
spanning tree) are sofic.

\begin{thm}\label{thm:soficcouplings}
  Let $(G_1,\rho_1)$ and $(G_2,\rho_2)$ be coupling equivalent
  unimodular random rooted graphs.  Then $(G_1,\rho_1)$ is strongly
  sofic if and only if $(G_2,\rho_2)$ is strongly sofic.
\end{thm}

The proof can be summarised as follows: Suppose $(G_1,\rho_1)$ is
strongly sofic, and let $m$ be a marking of $(G_2,\rho_2)$.  We can
encode both the structure of $G_2$ and the marks $m$ as a marking
$\hat m$ of $(G_1,\rho_1)$.  The strong soficity of $(G_1,\rho_1)$ allows
us to approximate $(G_1,\rho_1,\hat m)$ by a sequence of finite
graphs.  We then use this sequence to define an approximating sequence
for $(G_2,\rho_2,m)$.  Making this argument rigorous takes some care.

\begin{proof}
  Since $(G_1,\rho_1)$ and $(G_2,\rho_2)$ can both be considered as
  percolations on some unimodular random graph $(G,\rho)$, it suffices
  to prove that A unimodular graph is strongly sofic if and only if a
  connected percolation on it is strongly sofic.  Explicitly,
  if $(G,\rho)$ is a unimodular random rooted graph and
  $\omega$ is an almost surely connected percolation on $G$, then
  $(G,\rho)$ is strongly sofic if and only if the unimodular random
  rooted graph $(H,\rho')$ obtained from $(\omega,\rho)$ by conditioning
  on $\omega(\rho)=1$ is strongly sofic.

  First suppose that $(G,\rho)$ is strongly sofic and let $m$ be a
  marking of $(H,\rho')$.
  By \cref{lem:couplingmarks2}, there exists a marking $m$ of $G$ such
  that $(G,\rho,\omega,m)$ is unimodular and such that the law of
  $(H,\rho',m)$ coincides with the law of $(\omega,\rho,m)$ conditional
  on $\omega(\rho)=1$.  Since $(G,\rho)$ is strongly sofic, there exists
  a sequence of finite unimodular random marked graphs
  $(G_n,\rho_n,\omega_n,m_n)$ converging to $(G,\rho,\omega,m)$ in
  distribution.  Let $(H_n,\rho'_n,m_n)$ be the unimodular random rooted
  graph obtained from $(\omega_n,\rho_n,m_n)$ by conditioning on
  $\omega(\rho)=1$.  The sequence $(H_n,\rho'_n,m_n)$ converges to
  $(H,\rho',m)$ and, since $m$ was arbitrary, $(H,\rho')$ is strongly
  sofic.

  Suppose conversely that $(H,\rho')$ is strongly sofic and let $m$ be
  an $\bbX$-marking of $(G,\rho)$. By \cref{lem:couplingmarks}, we may
  assume that $(G,\rho,\omega,m)$ is unimodular. For each vertex $u$ of
  $G$, let $v(u)$ be chosen uniformly from the set of vertices in
  $\omega$ that minimize the graph distance to $u$, and for each vertex
  $v$ of $\omega$ let $U_v=\{v(u)=v\}$. Transporting mass $1$ from $u$
  to $v(u)$ for every vertex $u$ of $G$ shows that $\E|U_\rho|<\infty$,
  and in particular $U_v$ is finite for every vertex $v$ of $\omega$
  almost surley.
Conditional on $(G,\rho,\omega,m)$, let
$\{U(v):v\in V\}\cup\{U(e):e\in E\}$ be a collection i.i.d.\ uniform
$[0,1]$ random variables indexed by the vertices and edges of $G$, and
for each vertex $v$ of $G$ such that $\omega(v)=1$, define
\[ \hat{m}(v) = \left\{\left(U(u),m(u),\{ (U(e),m(e)) : e \text{ is an
        edge incident to } u \text{ in } G\}\right): u \in
    U_v\right\}. \]
If we denote the space of finite subsets of the metric space $X$ by $X^\ast$, then $\hat m$ takes values in the metric space
\[\left([0,1]\times\bbX \times \left([0,1]\times\bbX\right)^\ast\right)^\ast.\]
By conditioning on $\omega(\rho)=1$, we obtain a unimodular random
rooted marked graph $(H,\rho,\hat m)$.  Since $(H,\rho)$ is strongly
sofic, there exists a sequence of finite unimodular random rooted marked
graphs $(H_n,\rho_n,\hat{m}_n)$ converging in distribution to
$(H,\rho,\hat{m})$, and we may assume that $\E|\hat m_n(\rho_n)|<\infty$
for each $n\geq 1$. For each $\eps>0$ and $R \in \N$, bias
$(H_n,\rho_n,\hat m_n)$ by $|\hat m(\rho)|$ and construct a finite
unimodular random rooted marked graph
$(G_n^{R,\eps}, \rho_n, m_n^{R,\eps})$ from $(H_n,\rho'_n,\hat{m}_n)$ as
follows.
\begin{enumerate}
\item Let the vertex set of $G_n^{R,\eps}$ be the union
  $\bigcup_{v\in H_n}\hat m(v)$, and let $\rho_n$ be chosen uniformly
  from $\hat m(\rho'_n)$. For each vertex $u$ of $G_n^{R,\eps}$, let
  $u=(u^1,u^2,u^3)$ be the three coordinates of $u$, and let $m(u)=u^2$.
\item Draw an edge between two vertices $u_1$ and $u_2$ of $G_n^{R,\eps}$ if and only if
\begin{enumerate}\item
$u_1\in \hat m_n(v_1)$ and $u_2\in\hat m_n(v_2)$ for some vertices $v_1$ and $v_2$ of $H_n$ that are at distance at most $R$ in $H_n$, and
\item there exists a pair of points
  \[(t,x) \in u^3\subset [0,1]\times \bbX\text{ and }(s,y)\in u^3\subset
    [0,1]\times\bbX\] such that the distance between $(t,x)$ and $(s,y)$
  is less than $\eps$ in the product metric. If there are multiple such
  pairs, we draw multiple edges as appropriate.
\end{enumerate}
\item Let $\{Z_n(e)\}$ be a collection of i.i.d.\ Bernoulli-$1/2$ random
  variables indexed by the edges of $G_n^{R,\eps}$. For each edge $e$ of
  $G_n^{R,\eps}$, let $(t,x)$ and $(s,y)$ be the matching pair of points
  in $[0,1]\times\bbX$ that led us to draw $e$ in step (2), and let
  $m_n^{R,\eps}(e)=x$ if $Z_n(e)=0$ and $m_n^{R,\eps}(e)=y$ if
  $Z_n(e)=1$.
\end{enumerate}
This construction is continuous for the local topology, and hence for
each fixed $\eps$ and $R$, the finite marked graphs
$(G^{R,\eps}_n,\rho_n,m_n^{R,\eps})$ converge to the marked graph
$(G^{R,\eps},\rho,m^{R,\eps})$ defined by applying the same procedure to
$(H,\rho',\hat{m})$. Taking $\epsilon \to 0$, we obtain the marked graph
$(G^R,\rho,m^R)$ which consists of those edges of $(G,\rho)$ whose
endpoints are of distance at most $R$ in $\omega$.  Finally, taking
$R \to \infty$ we recover $(G,\rho,m)$. Thus, $(G,\rho,m)$ is a weak
limit of sofic unimodular random rooted marked graphs, and it follows
that $(G,\rho,m)$ is sofic.
\end{proof}

\begin{proof}[Proof of \cref{thm:soficmaps}.] Let $(M,\rho)$ be a simply
  connected unimodular random rooted map, let $(G,\rho)$ be the
  underlying graph of $M$, and let $\F$ be a sample of $\FUSF_M$. By
  \cref{thm:FUSFconnectivity}, $\F$ is connected almost surely, and so
  $(G,\rho)$ is coupling equivalent to the unimodular random rooted tree
  $(\F,\rho)$. It follows from \cref{thm:sofictrees,thm:soficcouplings}
  that $(G,\rho)$ is strongly sofic.  By encoding the map $(M,\rho)$ as
  a marking of $(G,\rho)$ as in \cref{Sec:unimod} and \cite[Example
  9.6]{AL07}, we conclude that the unimodular random rooted map
  $(M,\rho)$ is also sofic.
\end{proof}

\section{Open Problems}

We expect that the dichotomy of \cref{thm:dichotomy} extends to many further properties of planar unimodular random rooted maps. In this section, we discuss several such properties that might be addressed.

\subsection{Rates of escape of the random walk}

Can the type of a unimodular random planar map be determined by the rate of escape of the random walk? The work of Ding, Lee, and Peres \cite{ding2013markov} (together with the characterization of parabolic unimodular random planar maps as Benjamini-Schramm limits of finite planar maps) implies that the random walk is at most diffusive on any parabolic unimodular random planar map of finite expected degree.

\begin{thm}[\cite{ding2013markov}]
There exists a universal constant $C$ such that for every
parabolic unimodular random rooted map $(M,\rho)$ with $\E[\deg(\rho)]<\infty$, we have
\[\E\left[\deg(\rho) d(\rho,X_n)^2\right] \leq Cn \, \E\left[\deg(\rho)\right]\]
for all $n\geq 0$.
\end{thm}

On the other hand, if $(M,\rho)$ is a hyperbolic unimodular random rooted map with $\E[\deg(\rho)]<\infty$ that has at most exponential growth, meaning that
\[\limsup_{n\to\infty} \frac{1}{n}\log |B(\rho,n)| < \infty,\]
then the random walk on $M$ has positive speed, that is,
\[\lim_{n\to\infty} \frac{1}{n}d(\rho,X_n) > 0 \]
a.s., where the limit exists a.s.\ by Kingman's subadditive ergodic theorem. (Note that the exponential growth condition always holds for graphs of bounded degree.) This can be seen in several ways: it is an easy consequence of a theorem of Benjamini, Lyons, and Schramm \cite[Theorem 3.2]{BLS99} (see also \cite[Theorem 8.13]{AL07} and \cite[Theorem 3.2]{AHNR15}) that every invariantly nonamenable unimodular random rooted graph with finite expected degree and at most exponential growth has positive speed. Meanwhile, it
is a result of Benjamini and Curien~\cite{BC2011}, generalizing the work of  Kaimanovich, Vershik, and others \cite{kaimanovich2002boundary,kaimanovich1983random,kaimanovich2004boundaries,kaimanovich2003random,kaimanovich1998hausdorff,kaimanovich1990boundary}, that every non-Liouville unimodular random rooted graph with finite expected degree and at most exponential growth has positive speed.

In general, however, there do exist invariantly nonamenable, non-Liouville, unimodular random rooted graphs with finite expected degree such that the random walk has zero speed almost surely. An example of such a graph will appear in a forthcoming paper by the second author. We do not know of a planar example, which motivates the following question.

\begin{question}
Let $(M,\rho)$ be a hyperbolic unimodular random rooted planar map with $\E[\deg(\rho)]<\infty$. Does the random walk on $M$ have positive speed almost surely?
\end{question}

See \cite{AHNR15} for a related result concerning the positivity of the speed in the hyperbolic metric induced by the circle packing of a hyperbolic unimodular random rooted triangulation with $\E[\deg(\rho)^2]<\infty$.

\subsection{Positive harmonic functions}

\cref{thm:dichotomy} states that the existence of non-constant bounded harmonic functions and of non-constant harmonic Dirichlet functions are both determined by the type. We conjecture that a similar result holds for positive harmonic functions.

\begin{conjecture}
Let $(M,\rho)$ be a parabolic unimodular random rooted map and suppose that $\E[\deg(\rho)]<\infty$. Then $M$ does not admit any non-constant positive harmonic functions almost surely.
\end{conjecture}

This conjecture would follow from a positive answer to the following question.

\begin{question}\label{Question:trace}
  Let $(M,\rho)$ be a parabolic unimodular random rooted map and suppose
  that $\E[\deg(\rho)]<\infty$. Let $\langle X_n \rangle_{n\geq0}$ be a
  simple random walk on $M$. Is every component of the complement of the
  trace of $\langle X_n \rangle_{n\geq0}$ finite almost surely?
\end{question}

Note that the answer to \cref{Question:trace} is trivially positive if $M$ is recurrent.

\subsection*{Acknowledgment}

OA was supported by NSERC and the Simons Foundation.
TH was supported by a Microsoft Research PhD Fellowship.
AN was supported by ISF grant 1207/15, and ERC starting grant 676970 RANDGEOM.
GR was supported in part by EPSRC grant EP/I03372X/1.
Part of this work was conducted at the Isaac Newton Institute in
Cambridge, during the programme `Random Geometry' supported by EPSRC
Grant Number EP/K032208/1.

\bibliographystyle{abbrv}
\bibliography{unimodular}







\end{document}